\documentclass[11pt,reqno]{article}

\usepackage{amsmath, amsfonts, amssymb,  mathrsfs,  array, stmaryrd,  indentfirst, amsthm,  hyperref, comment, color, mathtools,bbm,todonotes}
\usepackage{graphicx, enumitem, tabularx}

\usepackage[margin=0.85in]{geometry}

\usepackage{setspace}
\usepackage{comment}
\numberwithin{equation}{section}
\newtheorem{defn}{Definition}[section]
\newtheorem{thm}{Theorem}[section]
\newtheorem{lemma}{Lemma}[section]
\newtheorem{prop}{Proposition}[section]

\newtheorem{assumption}{Assumption}[section]
\newtheorem{remark}{Remark}[section]
\newtheorem{example}{Example}[section]

\newcommand{\cA}{\mathcal{A}}

\newcommand{\cC}{\mathcal{C}}
\newcommand{\cD}{\mathcal{D}}

\newcommand{\cF}{\mathcal{F}}
\newcommand{\cG}{\mathcal{G}}

\newcommand{\cL}{\mathcal{L}}
\newcommand{\cM}{\mathcal{M}}
\newcommand{\cN}{\mathcal{N}}

\newcommand{\cR}{\mathcal{R}}
\newcommand{\cS}{\mathcal{S}}

\newcommand{\cX}{\mathcal{X}}

\def\cN{\mathcal{N}}

\newcommand{\Om}{{\Omega}}

\newcommand{\ka}{{\kappa}}

\newcommand{\pa}{{\partial}}
\newcommand{\one}{\mathbbm{1}}
\newcommand{\ignore}[1]{}

\renewcommand{\leq}{\leqslant }
\renewcommand{\geq}{\geqslant }

\def \a{\alpha}

\def \e{\varepsilon}
\def \d{\delta}
\def \1{\mathbf 1}
\def \es{\varepsilon^{m^*}}
\def \ets{\varepsilon^{2m^*}}
\def \eths{\varepsilon^{3m^*}}
\def \efs{\varepsilon^{4m^*}}
\def \ems{\varepsilon^{-m^*}}
\def \emts{\varepsilon^{-2m^*}}

\def \s{\mathfrak{S}}
\def \tauet{\tau^{\e,t,w,s,h}}
\def \eeta{\e^{2m^*}}
\def \emeta{\e^{-2m^*}}

\def \E{\mathbb{E}}
\def \F{\mathbb{F}}

\def \N{\mathbb{N}}
\def \P{\mathbb{P}}

\def \R{\mathbb{R}}

\newcommand{\comm}[1]{}

\makeatletter
\def\@setcopyright{}
\def\serieslogo@{}
\makeatother

\begin{document}

\title{Asymptotics for Small Nonlinear Price Impact: a PDE Approach to the Multidimensional Case\thanks{We would like to acknowledge L. C. G . Rogers of Cambridge University for challenging us to derive the asymptotic expansions for non-linear trading costs using the theory of viscosity solutions. We also would like to thank Marco Cirant of Universit\`a di Padova for fruitful discussions and helpful comments.}}

\date{\today}
\author{Erhan Bayraktar\thanks{University of Michigan, Department of Mathematics, 530 Church Street, Ann Arbor, MI 48109. email \texttt{erhan@umich.edu}. Supported in part by the National Science Foundation (NSF) and the Susan M. Smith Professorship.}
\and
Thomas Cay\'e\thanks{Dublin City University, School of Mathematical Sciences, Glasnevin, Dublin 9, Ireland, email \texttt{thomas.caye@dcu.ie}. Supported in part by the Swiss National Science Foundation (SNF) under grant 175133.}
\and
Ibrahim Ekren\thanks{Florida State University, Department of Mathematics, 1017 Academic Way 
Tallahassee, FL 32306, email \texttt{iekren@fsu.edu}}}
\maketitle

\abstract{We provide an asymptotic expansion of the value function of a multidimensional utility maximization problem from consumption with small non-linear price impact. In our model cross-impacts between assets are allowed. In the limit for small price impact, we determine the asymptotic expansion of the value function around its frictionless version. The leading order correction is characterized by a nonlinear second order PDE related to an ergodic control problem and a linear parabolic PDE. We illustrate our result on a multivariate geometric Brownian motion price model.}

\bigskip
\noindent\textbf{Mathematics Subject Classification (2010):} 91G10, 91G80, 35K55.

\bigskip
\noindent\textbf{Keywords:} asymptotic expansion, viscosity solutions, homogenization, portfolio choice, non-linear price impact.

\tableofcontents

\section{Introduction}

Choosing optimally the assets to own and to rebalance portfolios are the main tasks of portfolio managers. Classical finance first provided a systematic way of doing so in markets without frictions -- in what follows we will call it the Merton problem (without frictions) for the Nobel laureate Robert Merton who was the first to give its systematic solution. Although elegant, the results obtained with these models are of little help to real world investors, as the optimal strategies prescribed would lead to immediate ruin in the presence of the smallest frictions. Indeed, financial markets do not behave in the idealized way described by the assumptions of these early models. Trading actively costs, and all the more so when the size and speed of the portfolio rebalancing increases. The type of trading costs that interests us in this article stems from the lack of market depths and the insufficient liquidity of the assets traded. To trade, an investor needs counter-parties, be it over-the-counter or in the limit-order-book. Attracting enough of theses counter-parties for a large trade necessitates a price move in the direction adverse to the investor. It is accepted that the size of this shift increases with the size of the trade and the speed of trading. The seminal works of Bertsimas and Lo \cite{bertsimas.lo.98} and of Almgren and Chriss \cite{almgren.chriss.01,almgren.03} set the models that are now widely used in the literature: the price impact is assumed proportional to a power $\a>0$ of the trading rate. For tractability, price impacts were first taken to be linear ($\a=1$), and many interesting results sprouted in the literature (see e.g. \cite{guasoni.weber.17}, \cite{collindufresne.etal}, \cite{bank.soner.voss.17}, \cite{MMKS} and references therein)\footnote{Note that in these works and in our study price impact are only temporary: the asset price immediately reverts to its ``base value'' after the trade. For a review of the different models for transient (and temporary) price impact, see the survey by Gatheral and Schied \cite{gatheral.schied.13} and the work of Roch and Soner \cite{roch.soner.13}. Recent developments include \cite{schied.schoeneboren.09,GP,gueant.pu.17, ekren.MK.18, bank.voss.18}.}. However, empirical studies such as \cite{lillo.etal.03} and \cite{almgren.thum.hauptmann.li.05} suggest that the power $\a$ of the trading speed is in the interval $(0,1)$, which leads to superlinear but subquadratic trading costs. 

Portfolio choice problems with these types of frictions have only been recently studied. The first article to analyze these frictions is by Guasoni and Weber \cite{guasoni.weber.16}, in which the investor maximizes power utility of final wealth in a Black-Scholes market. Trading induces price impacts proportional to a power $\a$ of a ``volume-renormalization'' of the trading rate. They characterize the optimizer, identify a family of asymptotically optimal strategies and find the relative loss of utility to the leading order in the limit of small price impacts. These are expressed in terms of the solution of an ordinary differential equation that depends on $\a$. More recently the problem was solved in \cite{caye.herdegen.muhlekarbe.18} for constant absolute risk averse (CARA) investors and general one-dimensional markets where the price follows a not necessarily Markovian It\^o diffusion. The ODE found in \cite{guasoni.weber.16} is still a crucial building block of the solution. In a parallel strand of literature, Cai, Rosenbaum and Tankov \cite{cai.rosenbaum.tankov.15,cai.rosenbaum.tankov.16} embed these problems in the study of tracking problems.

While the previous works give a good understanding of the behaviour of one-dimensional markets, they {\color{black}do not take into }account the important effects of interdependence in multi asset markets through correlation and cross price impacts. {\color{black}Indeed, it is not unreasonable to think that, asset prices being correlated, a liquidity strain on an asset may very well have a liquidity effect on another asset. For proportional transaction costs and also for linear price impact the consequences of the asset liquidity interdependence on portfolio optimization problems are relatively well understood, see for example \cite{PST,bichuch.guasoni.18,GWM,GP,MMKS}. It is found that in contrast to the one-dimensional case, it is not anymore optimal to trade directly towards a target strategy, which turns out to be the frictionless optimizer, unless the current displacement is an eigenvector of the inverse price impact matrix multiplied by the volatility matrix. The optimal strategy is, in general, to trade in a direction given by a symmetric matrix depending on the price impact and volatility matrices. This direction needs not be pointing to the currently optimal portfolio.}

In the present article, we study a portfolio choice problem in a multidimensional market where transaction costs are superlinear and subquadratic. We consider an investor with a constant relative risk aversion (CRRA) who is maximizing the utility of her consumption on a finite time horizon. Our model allows for general patterns of cross-impacts and we consider a general Markovian It\^o price process set-up. The presence of superlinear costs makes this problem very challenging to solve explicitly and, as first done by Shreve and Soner in \cite{shreve.soner.94} in the context of transaction costs, we turn to the powerful machinery of viscosity solutions developed by Crandall, Ishii and Lions \cite{CIL} to characterize its solution. In particular, one of the challenging features of the cost structures we consider is that the associated Hamilton-Jacobi-Bellman equation is degenerate with a Hamiltonian growing superquadratically in the first derivatives with power $m=1+\frac{1}{\a}>2$. Considering that in practice transaction costs are small, it is natural to make an asymptotic analysis of the problem for small price impact using techniques developed by Soner and Touzi for proportional transaction costs in~\cite{ST}. These techniques, inspired by the homogenization theory for PDE with periodic terms, mixed with general stability results for viscosity solutions and the identification of two corrector equations rather than one, were then extended to multidimensional markets with proportional transaction cost by Possama\"i, Soner, and Touzi  \cite{PST} and later by Moreau, Muhle-Karbe and Soner \cite{MMKS} for linear price impact.\footnote{Note that to use this methodology, the setting of the problem needs to be Markovian. This constraint can be relaxed in one-dimensional markets through the use of convex duality (see \cite{GR2015}) and ergodic theory for one-dimensional diffusions, see e.g. \cite{kallsen.MK.17}, \cite{AK2013}, \cite{caye.herdegen.muhlekarbe.18}}

Based on the approach developed in \cite{ST,PST,MMKS} and using tools from homogenization theory \cite{lions.papanicolaou.varadhan.86,evans.89,evans.92, souganidis.99}, we characterize the asymptotic expansion of the value function of the problem in terms of the solutions of the so-called corrector equations. The first corrector equation is an ergodic type Hamilton-Jacobi-Bellman equation with superquadratic Hamiltonian similar to the equations studied in \cite{I, cirant.2014}\footnote{The study of portfolio problems with frictions is closely linked to ergodic control problem as was already made clear in \cite{ST}, \cite{AK2013}, \cite{caye.herdegen.muhlekarbe.18}. For recent developments in ergodic control theory and on how to approximate their solutions, see \cite{cirant.2014} and \cite{cacace.camilli.2016} and the references therein.}.  This PDE, whose structure depends neither on the price dynamics nor on the utility function, seems to be ubiquitous in the nonlinear price impact problems. It generalizes the ODE derived in \cite{guasoni.weber.16} by Guasoni and Weber, which is also a crucial element of the solution in \cite{caye.herdegen.muhlekarbe.18} in the setting of one-dimensional markets with nonlinear price impact.

This first corrector equation cannot be solved explicitly and we have to rely on existence results from \cite{I, cirant.2014}. In particular, in order to prove the asymptotic expansion, we need second derivative estimates for the solution of this equation and they are, at the moment, only available for a particular choice of the price impact functional (see~\eqref{eq:def.f.example}). For more general impacts, we expect that, under some technical assumptions, such second derivative bounds can still be obtained by using the adjoint method, but this is beyond the scope of this article.

The loss of utility at the leading order is then characterized by the second corrector equation which is a linear PDE. Similarly to the previous results in the literature, the second corrector equation is a linearization of the PDE solved by the value function of the frictionless problem. Its source term is provided by the first corrector equation and it reflects the local utility loss generated by the optimal control of the fast variable in an ergodic control problem. The solution of the second corrector equation can also be expressed as the expectation of a function of the state variables integrated over the trading horizon as in \cite{ST,PST,MMKS,caye.herdegen.muhlekarbe.18}. 

{\color{black}
Due to the fact that the first corrector equation cannot be solved explicitly in general, one needs to rely on numerical solutions of this equation to obtain approximations of the utility loss. Alternatively, some particular impact functions and choice of model allow to factorize its solution and express the utility loss in closed form (up to the solution of a one-dimensional ODE). We provide such a factorization in Proposition~\ref{factoifsscale} for the particular case of the Merton problem and cost functionals satisfying a scaling property. In this case, the dependence of the solution to the first corrector equation in the price, the time, and the wealth of the investor disappears. The only state variable for this partial differential equation is the rescaled deviation from the Merton portfolio. The solution of this equation is a couple $(\tilde \varpi,\lambda)$. Then, the utility loss can be directly written as a linear function of $\lambda$ and $(t,w)$ -- where the state variable tuple represents time and current wealth (marked to market). Although we do not prove here that they are, we also provide a family of strategies conjectured to be asymptotically optimal. These strategies can also be written using $\tilde \varpi$ and the dual friction function $\Phi$. Thus, answers to natural questions about the portfolio optimization problem with price impact are contained in $(\tilde \varpi,\lambda)$. For example, knowing whether or not the asymptotically optimal strategy consists in trading towards the Merton proportion can be answered by studying the functions $\Phi_x$ and $\tilde \varpi_x$. 

Besides this general statement on the Merton problem, we provide an example of impact function where the first corrector equation can be reduced to a one dimensional PDE and $(\tilde \varpi, \lambda)$ can be explicitly computed. Thus, for this impact function and similarly to \cite{GWM,MMKS}, $\lambda$ can be explicitly computed and the utility loss can be explicitly obtained. We also mention other examples of impact functions that could be considered.
}


The proof of the expansion relies on general stability results for viscosity solutions and locally uniform bounds for the difference between the frictionless and frictional value functions.  Regarding the viscosity solution aspects, 
unlike in \cite{MMKS,ekren.MK.18}, the solution of the first corrector equation is not a simple quadratic function of the deviation from the frictionless target. Thus, generating the test functions to use the viscosity property of the frictional value is more challenging. Compared to previous literature, our test functions have additional dependencies and do not always scale as a power function. Additionally, these features have to be included in the remainder estimates in Proposition~\ref{prop:remainder.estimate} which renders these estimates more challenging.   

The complexity of the solution of the first corrector equation also generates difficulties to check the assumptions needed to apply our main theorem. In particular, a new method is needed to obtain the locally uniform bounds for the difference between the frictionless and frictional value functions. Indeed, in our case we cannot simply apply It\^{o}'s formula to the squared difference between the current position and the target position to obtain these bounds. Also, we are not able to find a smooth subsolution of the frictional PDE as in \cite{PST} to obtain them. In order to exploit the strong mean-reversion property of the position to the frictionless target one needs to find an appropriate Lyapunov function. In fact, in our multidimensional framework, the solution of the first corrector equation is this Lyapunov function. However, the application of It\^{o}'s formula to the Lyapunov function gets extremely technical and is left to the Appendix where we give guidelines to obtain the locally uniform bounds of the renormalized difference between the two value functions for a model with Geometric Brownian Motion. {\color{black} Refining these technical estimates would prove the optimality of the candidate family of strategies. We are confident that these computations would succeed, but their length and tediousness would not bring more insight to an already quite technical article, and they are left to courageous readers.}

{\color{black} Finally, our result, with additional assumptions on the initial wealth and constraints on the available strategy can be used to determine derivative prices by utility indifference, and find the asymptotically optimal partial hedging strategies by substracting its final payoff from the final consumption of our investor (while ensuring that this quantity stays non-negative). See \cite{caye.herdegen.muhlekarbe.18} for example, where the case of a smoothed and convexified put option  with exponential utility of final wealth is considered.
}

The paper is organized as follows. We state the problem of interest in Section~\ref{s.model}. We then state our main result in Section~\ref{s.mainresults}. In particular, in Section~\ref{s.blackScholes}, we provide an expansion for the particular case of the multidimensional Merton problem and a fairly general class of impact functions. In Section~\ref{sec:exampl} we also provide an explicit expression of impact function that satisfies the assumptions of our main results and we give the expansion of the value function for this example. 
In Section~\ref{s.remainder}, we state the remainder estimates which are the main estimates to carry out the viscosity theory proofs in Section~\ref{s.semi}. In the Appendix, we prove a technical lemma on the distance between the asymptotically optimally controlled state variables and the frictionless target necessary to check the assumptions of the main Theorem~\ref{thm:expansion} for our example.

In the rest of this section, for readers' convenience, we will list the frequently used notation. 

\subsection{Notation}
\label{s.notations}
For a smooth function $\phi$,
$\phi_y$ denotes the derivative of $\phi$ in $y$ (we reserve the notation $\pa_y \phi$ for the derivatives of  any composition in $y$) and $\phi_{xy}$ denotes the second order derivative with respect to $x$ and $y$. We denote by $|\cdot|$ the Euclidian norm for a vector or matrix, $\cS_d$ stands for the set of symmetric matrices. For a deterministic vector $x$, $x_i$ denotes its $i$th component and for a vector valued process $(X_t)$, $X^i_t$ denotes its $i$th component. We also define the sets $\cD:=\{(t,w,s)\in[0,T)\times \R_{++}\times \R_{++}^d\}$ and $\cD_T:=\{(t,w,s)\in\{T\}\times \R_{++}\times \R_{++}^d\}$ where we write $\R_{++}:=(0,\infty)$. In the article, we write $C$ for a generic positive constant or a positive continuous function on $\cD$ with polynomial growth, that may change from a line to the next.
For vectors $x,y\in \R^d$ and $\a\geq 0$, we denote by $x\cdot y$ their Euclidian scalar product,  $x^{(\a)}=(sign(x_1) |x_1|^{\a},\ldots,sign(x_d) |x_d|^{\a} )^\top$, $x^{|\a|}=( |x_1|^{\a},\ldots, |x_d|^{\a} )^\top$ and $x\times y=( x_1y_1,\ldots, x_dy_d )^\top$.
\section{The Model}
\label{s.model}
\subsection{The Merton Problem Without Friction}
\label{subsec.frictionless.model}
Let $\big(\Om,\cF,\F=\left(\cF_{t}\right)_{t\in [0,T]},\P\big)$ be a filtered probability space. The financial market consists of a money market account with interest rate $r(S)$ and stocks. The dynamics of the assets are as follows
\begin{align}
\label{eq:SDE.S}
dS^j_t&=S^j_t(\mu^j(S_t) dt+\sigma^{j}(S_t) \cdot dB_t), &\mbox{with }~&S^j_0 =s^j_0\ &\mbox{for }1{\leq} j{\leq} d,~\\
dS^0_t&=S^0_t r(S_t) dt, &\mbox{with }~&S^0_0= s^0_0,&\notag
\end{align}
where $B$ is a $d$-dimensional Brownian motion, $r$, $\mu^j$ are continuous real-valued functions, $\sigma^j$ a continuous vector-valued function, for $1{\leq} j{\leq} d$. We also assume that $s\mapsto s_j\mu^j(s)$ and $s\mapsto s_j \sigma^j(s)$ are Lipschitz continuous for all $j=1,\ldots d$. We denote by $\sigma$ the $d\times d$ matrix whose $j$-th row is $\sigma^j$. We assume furthermore that $\sigma\sigma^T$ is positive definite. We write indifferently $r_t$ for $r(S_t)$, $\mu_t$ for $\mu(S_t)$ and $\sigma_t$ for $\sigma(S_t)$. \footnote{Note that the interest rate and the drift and volatility of the price could depend on time and a multi-dimensional factor process. All the analysis of the article goes through. We choose to omit them for readability and simplicity of notations.}

The investor chooses at time $t$ a consumption rate $c_t\geqslant 0$ and the number of shares $H_t= (H_t^1,...,H_t^d)^\top$ to hold in the stock. Her wealth evolves according to the following dynamics:
\begin{equation}
\label{eq:SDE.W0}
dW^0_t=(r_tW^0_t-c_t)dt+\sum_{j=1}^{d} H_t^jS_t^j(\mu^j_t-r_t)dt+\sum_{i,j=1}^{d} H_t^jS_t^j\sigma^j_i(S_t)dB^i_t,
\end{equation}
in vector notation,
\begin{equation} \label{eq:dyn.W0}
dW^0_t=(r_tW^0_t-c_t)dt+(H_t\times S_t)\cdot(\mu_t-r_t\one)dt+ (H_t\times S_t)^T\sigma_t dB_t,
\end{equation}
where $\one$ is the vector of $\R^d$ with $1$ for every component.
The objective is to maximize the expected utility from consumption. The value function for $(t,w,s)\in \cD\cup\cD_T$ is,
\begin{equation}
\label{eq:value.fcn}
V^0(t,w,s)=\sup_{(c,H)\in\mathcal{A}^0(w)}\mathbb{E}\bigg[\int_t^T U(c_r)dr+{U(c_T)}\bigg | W_t=w,S_t=s\bigg],
\end{equation} $U(x)=x^{1-R}/(1-R), R>0,\ R\neq 1$ is the utility function, and $\mathcal{A}^0$ is the set of admissible strategies (i.e. number of shares in each asset and consumption) which guarantee that the SDE for the wealth process has a strong solution and that $W^0_t\geqslant 0$ and $c_t\geqslant 0$ for all $t\geqslant 0$ and $c_T\leqslant W^0_T$ . 
\begin{remark}
	(i) Note that although $S_t$ appears in the wealth dynamics, we could easily remove it by treating the amount in units of num\'eraire invested in the stock, $H_tS_t$, as the control. Here we consider the number of shares as control in order to have more resemblance with the frictional case introduced in the next section.
	
	(ii) Note also that unlike the classical problem, the consumption at the final time is controlled. This choice is made so that the frictionless and the frictional problem \eqref{eq:Ve} have the same structure. However, we obviously have the same value as the classical Merton problem since the optimal control here is $c_T=W^0_T$ as it is the case in the classical Merton problem.
\end{remark}

Under some assumptions (e.g. the existence of regular solutions to the Hamilton-Jacobi-Bellman equation and of an optimal policy for \eqref{eq:value.fcn}) it can be proved that the frictionless value function $V^0$ satisfies the Hamilton-Jacobi-Bellman equation for all $(t,w,s)\in \cD$
\begin{align} \label{eq:v0}
&\cG^0(V^0)(t,w,s):=-\pa_t V^0-\sup_c\{ U(c)- V^0_w c \} -\cL^s V^0- V^0_w r w\\
&\qquad-\sup_{h} \Bigg\{ V^0_w\sum_{j=1}^{d} h_j s_j (\mu^j(s)-r)+\sum_{i,j=1}^{d}V^0_{w s_i} s_i h_j s_j\sigma^i(s) \cdot\sigma^j(s) +\frac{ V^0_{ww}}{2} \Bigg|\sum_{j=1}^{d} h_j s_j \sigma^j(s)\bigg|^2\Bigg\}=0,\notag\\
&V^0(T,w,s)=U(w),\label{eq:final.cond.V0}
\end{align}
where $\cL^s$ is the infinitesimal generator associated to $S$, and the above supremum is attained pointwise by 
\begin{align}
h^0(t,w,s):=\text{arg}\max_{h\in \R^d} \Bigg\{& V^0_w(t,w,s)\sum_{j=1}^{d} h_j s_j (\mu^j(s)-r(s))+\sum_{i,j=1}^{d}V^0_{w s_i}(t,w,s) s_i h_j s_j\sigma^i(s) \cdot\sigma^j(s)\notag\\
& +\frac{ V^0_{ww}(t,w,s)}{2} \Bigg|\sum_{j=1}^{d} h_j s_j \sigma^j(s)\Bigg|^2\Bigg\}.\label{eq:h0}
\end{align}
Then $h^0$ satisfies the first order condition for all $(t,w,s)\in \cD$
\begin{equation}
\label{eq:foc.h0}
V^0_w s_j\left(\mu^j-r\right)+\sum_{i=1}^{d}V^0_{w s_i}s_i s_j\sigma^i(s)\cdot\sigma^j(s)+V^0_{ww}\sum_{i=1}^{d}h^0_is_is_j\sigma^i(s)\cdot\sigma^j(s)=0,~~\mbox{for all }~1{\leq} j{\leq} d.
\end{equation}
We define additionally $H^0_t=h^0(t,W^0_t,S_t)$, for $t\geqslant 0$.
Denote by
\begin{equation}
\label{eq:def.Utilde}
\widetilde{U}(y):=\sup_{x}\{U(x)-xy\}=\frac{R}{1-R}y^{-\frac{1-R}{R}},\ \mbox{for}\ y>0
\end{equation}
the convex dual of $U$ and the optimal consumption rate
\begin{equation}\label{eq:c0}
c^0(t,w,s):=- \tilde U' ( V^0_w(t,w,s)).
\end{equation}
Finally, define $(t,w,s)\mapsto c^{h^0}(t,w,s)$ the function such that 
\begin{equation}\label{eq:qvarH}
c^{h^0}_t=c^{h^0}\left(t,W^0_t,S_t\right)=\frac{d\left\langle h^0\left(t,W^0_t,S_t\right)\right\rangle}{dt}\in \cS_d
\end{equation}
is the quadratic variation of $\left(H^0_t\right)_{t\in[0,T]}$.

\begin{assumption}
\label{assum:frictionless} We make the following assumption on the frictionless problem.
\begin{enumerate}
\item[(i)] $V^0,c^0,h^0$ are $\cC^{1,2,2}(\cD\cup \cD_T)$ functions, and $c^{h^0}$ is continuous and positive on $\cD\cup \cD_T$.
\item[(ii)] The strategy $h^0$ does not allow short selling, borrowing, nor zero position in any of the assets: \begin{equation}
\label{eq:simplex.condition}
\frac{h^0\times s}{w}\in(0,1)^d\ \mbox{ and }\ \sum_{i=1}^{d}\frac{s_ih^{0}_i}{w}<1.
\end{equation}
\item[(iii)] We have $V^0_w>0$ and $V^0_{ww}< 0$ on $\cD\cup \cD_T$.
\end{enumerate}
\end{assumption}
\begin{remark}
	(i) {\color{black}Short-selling and borrowing to invest are forbidden as in \cite{guasoni.weber.16} to avoid ruin in the presence of market impact. Indeed, in the presence of price impact, the investor might not be able to liquidate her short position or her over-invested portfolio fast enough in case of a down-turn of the market and may face ruin, which is not allowed by the power utility function. 
	
	(ii) However, unlike  \cite{guasoni.weber.16}, in \eqref{eq:simplex.condition},  we assume that the strategy strictly belongs to the interior of the simplex for technical reasons: it is necessary to ensure admissibility of our candidate strategy in the treated example; see Sections~\ref{s.def.candidate} and~\ref{s.local.boundedness} for details.}
	\end{remark}

With these assumptions and notation we can rewrite the equation \eqref{eq:v0} as 
\begin{align}
\label{eq:V0.simple}
0&=-\pa_t V^0-\tilde U( V^0_w) -\cL^s V^0- V^0_w r w\\
&- V^0_w (h^0\times s) \cdot(\mu-r\one)-\sum_{i,j=1}^{d}V^0_{w s_i} s_i h_j^0(t,w,s) s_j\sigma^j \cdot\sigma^i -\frac{ V^0_{ww}}{2} \Bigg|\sum_{j=1}^{d} h^0_j(t,w,s) s_j \sigma^j\Bigg|^2=0.\notag
\end{align}
\begin{example}\label{ex:frictionless}
Throughout the article we will illustrate our results for geometric Brownian motion as the multi-dimensional price process and an investor with risk aversion $R\in(0,1)$. This means that we take $\mu, \sigma$ and $r$ constant, such that $\sigma\sigma^\top$ is positive definite and denote 
\begin{align}\label{def:ss}
\s=\left(\sigma\sigma^T\right)^{1/2}.
\end{align}
In this case one can show by verification (using the following ansatz, Equations \eqref{eq:foc.h0} and \eqref{eq:V0.simple}) that the value function takes the form 
\begin{align}\label{eq:val.frictionless}
V^0(t,w,s)&=g(t)U(w),\\
c^0(t,w,s)&=g(t)^{-\frac{1}{R}}w,\label{eq.conso.frictionless}\\
h^0(t,w,s)&=\left(\pi_1\frac{w}{s_1},...,\pi_d\frac{w}{s_d}\right)^\top,\ \mbox{with}\ \pi=R^{-1}\left(\sigma\sigma^\top\right)^{-1}\left(\mu-r\one\right),\label{eq:strat.frictionless}
\end{align}
where 
\begin{equation}\label{eq:g}
g(t)=\left(\frac{1+(\nu-1)e^{-\nu(T-t)}}{\nu}\right)^R
\end{equation}
satisfies $g(T)=1$ and $\nu=(R-1)\left(\frac{r}{R}+\frac{\left(\mu-r\one\right)^\top\left(\sigma\sigma^\top\right)^{-1}\left(\mu-r\one\right)}{2R^2}\right)\neq 0.$

The Assumption~\ref{assum:frictionless} means that we require $\pi_i(t)=\pi_i>0$ for all $i=1,\ldots d$ and $\sum_{i=1}^d  \pi_i<1$.
\end{example}
\begin{remark}
	\label{rem:g} Note that the function $g$ is bounded and bounded away from $0$ on $[0,T]$. 
\end{remark}

\subsection{The Merton Problem with Price Impact}
Let the financial market be the same as in the frictionless case except that the execution price $\widetilde{S}_t$ may be different from the fundamental price $S_t$. {\color{black}More precisely, we consider a temporary price impact model,
\begin{align}
\label{eq:effective.price}
\widetilde{S}^j_t=S^j_t+f_j(S_t, \varepsilon \theta_t),
\end{align}
where\footnote{In the presence of price impacts, only strategies absolutely continuous with respect to the Lebesgue measure can be optimal: infinite variation leads to immediate ruin.} $\theta_t=\dot H_t$ is the trading rate, $\varepsilon$ is a small parameter, and $(f_j)$ is a family of functions satisfying the following Assumption~\ref{assum:phi}. We additionally define the convex conjugate $\Phi(s,.)$ of the cost functional $\theta\mapsto \theta\cdot f(s,\theta)$ in $\R^d$ by 
\begin{align}\label{eq:defphi}
\Phi(s,x):=\sup_{\theta\in \R^d}\bigg\{x\cdot \theta-\sum_{j=1}^{d} \theta_j  f_j(s,\theta)\bigg\},
\end{align}
on which the Assumption~\ref{assum:phi} on $f$ can be equivalently formulated.
\begin{assumption}\label{assum:phi}
	The function $f$ is continuous in both of its variables, continuously differentiable in $\theta$ and positive homogeneous of degree $\alpha\in(0,1)$ and odd in $\theta$. Additionally, we assume that for all $s\in(\R_{++})^d$, the function $\theta\mapsto \theta\cdot f(s,\theta)$ is strictly convex (it is then also homogeneous of degree $\alpha +1$\footnote{This is the degree of homogeneity of the transaction cost in the trading speed, parameter called $p$ in \cite{caye.herdegen.muhlekarbe.18}}). 
	
	$\Phi$ is assumed to be continuous in both of its variables and continuously differentiable in $x$. $\Phi$ is convex positive homogeneous of degree $m>2$ in $x$, and $\Phi_x$, its derivative in $x$, is an odd function in $x$ that is homogeneous of degree $m-1$, where $m = 1 + \tfrac{1}{\alpha}>2$.
\end{assumption}
In Section~\ref{s.blackScholes}, we provide examples of impact functions satisfying this assumption. 
\begin{remark}\label{remark:mm*}
   (i) The constant $\alpha\in (0,1)$ (or equivalently and more conveniently for the PDE results used later the constant $m=1+\frac{1}{\a}>2$) will be crucial in our study. It expresses the scaling of market impact as a function of the trading rate. We fix $\alpha \in (0,1)$ (and therefore $m>2$) and define $m^*= \frac{\alpha}{\alpha+3}=\frac{1}{3m-2}$. 
    
   (ii) Note that the function $f$ is such that the instantaneous transaction costs function $\theta\to \theta \cdot f(s,\theta)$ is non-negative, convex and positive homogeneous of order $\a+1=\frac{m}{m-1}\in(1,2)$: this is the $d$- dimensional version of the sub-quadratic trading cost as it appears in \cite{guasoni.weber.16} and \cite{caye.herdegen.muhlekarbe.18} where the market is composed of a single risky asset and a bank account. The importance of the dual friction $\Phi$ was already identified by Dolinsky and Soner in \cite{DS2013} in discrete time frictional markets and used by Guasoni and R\'asonyi in \cite{GR2015} to provide a characterization of the optimizer of the solution of the nonlinear frictions in a one-dimensional set-up using convex duality.
\end{remark}
With these frictions, the wealth of the investor starting at $t=0$ with a marked-to-market wealth of $w_0$,} satisfies the following SDE (with initial condition $W^\e_0=w_0$) on $[0,T]$
\begin{align}\label{eq:frictionalwealthdynamics}
	dW^\e_t&=(r_tW^\e_t-c_t)dt+\sum_{j=1}^{d}\left(H^j_tS^j_t[(\mu^j_t-r_t)dt+\sigma^j_t\cdot dB_t]-\theta^j_tf_j(S_t,\varepsilon \theta_t)dt\right).
\end{align}
The value function corresponding to this new control problem is
\begin{equation}
\label{eq:Ve}
	V^\varepsilon:(t,w,s,h)\in\cD\times\R^d\mapsto \sup_{(c,\theta)\in\mathcal{A}^\varepsilon(w)}\mathbb{E}\bigg[\int_t^T U(c_r)dr{+U(c_T)}\Big | W_t=w, H_t=h, S_t=s\bigg],
\end{equation}
where $(c,\theta)\in\mathcal{A}^\varepsilon$ if $W_s^{\e}-\sum_{i=1}^dH^i_s S_s^i\geqslant 0$, $H^i_s\geqslant 0$ and $c_s\geqslant 0$ for all $s\in [t,T)$ and if $c_T\leq W_T^{\e}-\sum_{i=1}^dH^i_T S_T^i$.
\begin{remark}
We note that in our framework, the total wealth $w$ represents the frictionless liquidation value of the portfolio where $h$ is the vector of number of shares held in each asset and $w-\sum_{i=1}^d h^i s^i$ is the cash holdings of the agent. In other words, $W^\e$ represents the \emph{``marked-to-market''} wealth of the investor.
\end{remark}
The associated Hamilton-Jacobi-Bellman equation is
\begin{align}
\label{eq:HJB-veps}
&-{\pa_t}V^\e-\sup_c\{ U(c)- V^\e_w c \} -\cL^s V^\e- V^\e_w r w-\sup_\theta \bigg\{ V^\e_h \theta- V^\e_w\sum_{j=1}^{d} \theta^j_tf_j(s,\varepsilon \theta_t)\bigg\}\\
&\qquad\qquad\qquad\qquad\qquad- V^\e_w\sum_{j=1}^{d} h_j s_j (\mu^j-r)-\sum_{i,j=1}^{d}V^\e_{w s_i} s_i h_j s_j\sigma^j \cdot\sigma^i -\frac{ V^\e_{ww}}{2} \Bigg|\sum_{j=1}^{d} h_j s_j \sigma^j\Bigg|^2=0.\notag
\end{align}

Optimizing over $c$ and $\theta$, and using the homogeneity of $f$ and $\Phi$ we rewrite the Hamilton-Jacobi-Bellman equation \eqref{eq:HJB-veps} as
\begin{align}\label{eq:PDE-veps}
0=\cG^\e(V^\e)(t,w,s,h)&:=-{\pa_t}V^\e-\tilde U( V^\e_w) -\cL^s V^\e- V^\e_w r w-\frac{ (V^\e_w)^{1-m}}{\e}\Phi\left(s, V^\e_h\right)\\
&- V^\e_w\sum_{j=1}^{d} h_j s_j (\mu^j-r)-\sum_{i,j=1}^dV^\e_{w s_i} s_i h_j s_j\sigma^j \cdot\sigma^i -\frac{ V^\e_{ww}}{2} \Bigg|\sum_{j=1}^{d} h_j s_j \sigma^j\Bigg|^2.\notag
\end{align}
We make the following assumption on the frictional value function.
\begin{assumption}[Characterization of the frictional value function]\label{assum:Ve.viscosity}
	For $\e>0$ the value function $V^\e$ is locally bounded and is a (possibly discontinuous) viscosity solution of the Hamilton-Jacobi-Bellman equation 
	\begin{equation}
	\begin{cases}
	\cG^\e(V^\e)(t,w,s,h) &=0 \quad\quad\quad\quad\quad\quad\qquad\mbox{ for } ~(t,w,s,h)\in\cD\times \R^d\\
	V^\e(T,w,s,h) &= U(w-\sum_{i=1}^{d}h_js_j) ~~ \mbox{ for } ~(w,s,h)\in\R_{++}\times\R_{++}^d\times \R^d,\label{eq:final.cond.Ve}
	\end{cases}
	\end{equation}
	where $\cG^\e$ is defined in \eqref{eq:PDE-veps}.
\end{assumption}
\section{Main Results}
\label{s.mainresults}
\subsection{The Corrector Equations}
{\color{black}
\paragraph{Heuristic motivation of the equations.} We postulate an asymptotic expansion for the value function of the frictional problem of the form
\begin{equation*}
V^\e(t, w, s, h) = V^0(t, w, s) -\ets u(t, w, s) - \efs \varpi\bigg(t, w, s, \frac{h- h^0(t, w, s)}{\es}\bigg),
\end{equation*}
for some functions $u$ and $\varpi$ satisfying integrability and growth conditions to be determined, and where the fourth argument of $\varpi$ is the ``fast variable'' 
\begin{align}\label{eq:fastaux1}\xi^\e(t, w, s, h) = \frac{ h- h^0(t, w, s)}{\es}.
\end{align} 
Introducing this ansatz in the HJB equation \eqref{eq:PDE-veps} we obtain (recall that $V^0$ satisfies \eqref{eq:V0.simple} and $\Phi$ is homogeneous of order $m$)
\begin{align*}
	\cG^\e(V^\e)(t,w,s,h) &= \e^{2m^*}\Big(E_1\big(t,w,s,\xi^\e,\e^{3m^*-\frac{1+2m^*}{m}}\varpi_\xi(t,w,s,\xi^\e),\varpi_{\xi\xi}(t,w,s,\xi^\e)\big)\notag\\
	&\quad + E_2\big(t,w,s,u_t, u_w, u_s,u_{ww},u_{ws},u_{ss}\big)\notag\\
	&\quad + R^\e(t,w,s,h,V^\e),
\end{align*}
where the functional $E_1$ and $E_2$ are defined as follows

\begin{align}
E_1(t,w,s,\xi,p,X)&:=-\frac{V^0_{ww}}{2} \Bigg|\sum_{j=1}^{d} \xi_j s_j \sigma^j\Bigg|^2-|V^0_w|^{1-m}\Phi\left(s,p\right)+\frac{1}{2}\text{Tr}\left(c^{h^0} X\right),\label{correc1}\\
E_2(t,w,s,x,p,q,&X_{ww},X_{ws},X_{ss}):=T(t,w,s,p,X_{ww},X_{ws};h^0(t,s,w))\notag\\
&+ x+(q\times s)\cdot \mu(s)+r wp+\frac{1}{2}\sum_{i,j=1}^{d} (\sigma\sigma^\top)_{i,j} s_is_j(X_{ss})_{i,j}+\tilde U'(V^0_w)p,\label{correc2}
\end{align}
 and we denote 
\begin{align}\label{eq:defT}
T(t,w,s, p, X_{ww}, X_{ws};h):=& p(h\times s) \cdot(\mu-r\one)+\left(X_{ws}\times s\times h \right)^T\sigma\sigma^T s\\
&+\frac{X_{ww}}{2}\left(h\times s\right)^T\sigma\sigma^T\left(h\times s\right)\notag
\end{align}
and $c^{h^0}(t,w,s)$ is the quadratic variation of the frictionless strategy defined in \eqref{eq:qvarH}. 

The value of the scaling factor $m^*= \frac{1}{3m-2}$ is in fact the only choice that allows us to claim $\e^{3m^*-\frac{1+2m^*}{m}}=1$ and the expansion yields 
\begin{align}\label{eq:heuristic_corrector_equation}
	\cG^\e(V^\e)(t,w,s,h) &= \e^{2m^*}\Big(E_1\big(t,w,s,\xi^\e,\varpi_\xi(t,w,s,\xi^\e),\varpi_{\xi\xi}(t,w,s,\xi^\e)\big)\notag\\
	&\quad + E_2\big(t,w,s,u_t, u_w, u_s,u_{ww},u_{ws},u_{ss}\big)\Big)\notag\\
	&\quad + R^\e(t,w,s,h,V^\e),
\end{align}
The last term of~\eqref{eq:heuristic_corrector_equation}, under technical conditions on $u$ and $\varpi$ converges to $0$ as $\e\downarrow 0$ for $h$ in a neighbourhood of $h^0(t, w, s)$. See Section~\ref{s.remainder} and  in particular Proposition~\ref{prop:remainder.estimate} for further details.

Thus, for the correct expansion, we expect $\varpi$ and $u$ to solve the so called first and second corrector equations,
\begin{align}\label{1st-corrector-eq}
E_1(t,w,s,\xi,\varpi_\xi(t,w,s,\xi),\varpi_{\xi\xi}(t,w,s,\xi))&=a(t,w,s)~\mbox{on}~(\cD\cup\cD_T)\times\R^d,\\
\varpi(t,w,s,0)&=0 ~\mbox{on}~\cD\notag\\
\label{2nd-corrector-eq}
-E_2(t,w,s,u_t, u_w, u_s,u_{ww},u_{ws},u_{ss})&=a(t,w,s)~\mbox{on}~(\cD\cup\cD_T),\\
u(T,w,s)&=0~\mbox{on}~\R_+^*\times\left(\R_+^*\right)^d.\notag
\end{align}

}

\begin{remark}
(i) For the first corrector equation \eqref{1st-corrector-eq}, the triplet $(t,w,s)\in\cD$ is in fact just a parameter of the equation. The variable of the equation is $\xi\in \R^d$ as in \cite{BM,cirant.2014,I}. The solution of the first corrector equation is a couple $(\varpi,a)$. The term $a$ of the solution of the first corrector equation is a source term in the second corrector equation.

(ii) The first corrector equation \eqref{1st-corrector-eq} is the most important object in the study of small transaction cost asymptotic theory. Indeed, as the transaction parameter $\e$ goes to $0$, the deviation from the frictionless position $\xi^\e_t:=\frac{H_t^\e-H^0_t}{\varepsilon^{m^*}}$ oscillates faster around $0$. The first corrector equation expresses the trade-off the agent needs to make between keeping this quantity close to $0$ and the transaction cost she has to pay. In the first corrector equation, the utility loss due to the displacement $\xi^\e$ comes from the term $-\frac{V^0_{ww} (t,w,s)}{2} \big|\sum_{j=1}^{d} \xi_j s_j \sigma^j\big|^2\geq 0$ and the transaction cost comes from the term $(V^0_w(t,w,s))^{1-m}\Phi\left(s,p\right)$. The local control problem that the agent solves is an infinite horizon problem similar to \cite[Equation (2.2)]{I}.  This comes from the difference of time scale in which the different variables evolve. Indeed, as $\e$ goes to $0$ the variable $\xi^\e$ (the so-called fast variable in the homogeneization theory) oscillates faster and faster and sees any infinitesimally small interval (under the right scaling) as infinite, while the variable $(t,w,s)$ stays almost constant on such small intervals.

{\color{black}
(iii) We note also that in the system \eqref{1st-corrector-eq}-\eqref{2nd-corrector-eq} the only nonlinearity is the dual friction function $\Phi$ that is assumed to be convex positive homogeneous of degree $m>2$. $\Phi$ is also the only term in this system that depends on the price impact function $f$.}
\end{remark}

We now state the different assumptions needed for our main result.
\begin{assumption}[Local boundedness of the renormalized utility loss]
\label{assum:bound}
For all $(t_0,w_0,s_0)\in\cD$ there exists $\e_0>0$ and $r_0>0$ such that 
\begin{align}\label{eq:control-deviation}
\sup\left\{\frac{V^0(t,w,s)-V^\e(t,w,s,h)}{\e^{2m^*}}:\left|(t,w,s,h)-(t_0, w_0, s_0, h^0(t_0,w_0,s_0))\right|{\leq} r_0,\e\in(0,\e_0)\right\}<\infty,
\end{align}
 where $m^*=\frac{1}{3m-2}$ is defined in Remark~\ref{remark:mm*}.
\end{assumption}

{\color{black}
Assumption~\ref{assum:bound}, is sufficient to define on $\cD$ the semi-limits
\begin{align}
\label{eq:def-semilimits-upper_sec3}
u^*(t,w,s):=\limsup_{\e\downarrow 0, (t',w',s',h')\to (t,w,s,h^0(t,w,s))}  \frac{V^0(t',w',s')-V^\e_*(t',w',s',h')}{\e^{2m^*}},\\
\label{eq:def-semilimits-lower_sec3}
u_*(t,w,s):=\liminf_{\e\downarrow 0, (t',w',s',h')\to (t,w,s,h^0(t,w,s))} \frac{V^0(t',w',s')-V^{\e,*}(t',w',s',h')}{\e^{2m^*}},
\end{align}
used in Section~\ref{s.semi}, where $V^\e_*$ and $V^{\e,*}$ are respectively the lower and upper semicontinuous envelopes of $V^\e$. They allow us to characterise the deviation from the frictionless value and prove our main theorem stated below. }

\begin{remark}
This assumption is a property of the frictional value $V^\e$ and is in fact the most difficult assumption to check. There are two ways of checking the assumption. Similarly to \cite{PST}, one can try to exhibit a smooth subsolution of the frictional PDE \eqref{eq:HJB-veps} that is close enough to the frictionless value function  $V^0$. However, in \cite{PST}, this method works due to a proportionality between the frictionless value $V^0$ and the solution to the second corrector equation $u$. In our framework, such a proportionality does not hold and we cannot employ this method. 

The second method is to exhibit an admissible control whose performance gives the required bounds. This is the one we use for the example treated in Section~\ref{s.blackScholes}. However, our task is rendered technical by the controls admissibility condition and the fact that power utility is only defined for positive consumptions. 
\end{remark}
Before making assumptions on the solutions to the corrector equations we define the following class of functions.

\begin{defn}\label{def:Cm}
	 We call $\cC_m$ the subset of $\cC^{1,2,2,2}\big([0,T]\times\R_{++}\times\left(\R_{++}\right)^d\times\R^d, \R_+\big)$ containing the functions $\chi$ with the following growth rates at infinity
	\begin{align}
		&\sup_{\xi\in\R^d}\left(1+|\xi|^2\right)^{-\left(\frac{1}{2}+\frac{1}{m}\right)}\left(\left|\chi\right|+\left|\chi_w\right|+\left|\chi_s\right|+\left|\chi_{ww}\right|+\left|\chi_{ws}\right|+\left|\chi_{ss}\right|+\left|\chi_{\xi\xi}\right|\right)(t,w,s,\xi)\notag\\
		&\qquad\qquad\qquad\qquad\qquad+\sup_{\xi\in\R^d}\left(1+|\xi|^2\right)^{-\frac{1}{m}}\left(\left|\chi_\xi\right|+\left|\chi_{w\xi}\right|+\left|\chi_{s\xi}\right|\right)(t,w,s,\xi){\leq} C(t,w,s)\label{eq:estimate.Cvarpi},
	\end{align}
	where $C$ is a continuous function such that $(w,s)\mapsto \sup_{t\in[0,T]}C(t,w,s)$ is locally bounded.
\end{defn}

\begin{assumption}[Assumption~on the corrector equations]\label{assum:2ndcorrect}
(i) The first corrector equation \eqref{1st-corrector-eq} admits a solution $(a,\varpi)$ 
\begin{align*}
	a:(t,w,s)\in[0,T]\times \R_{++}\times \R_{++}^d \mapsto \R_+,\\
	\varpi:(t,w,s,\xi)\in[0,T]\times \R_{++}\times \R_{++}^d\times \R^d\mapsto \R_+.
\end{align*}
These two functions are in $\cC^0([0,T]\times\R_{++}\times\R^d_{++},\R_+)$ and $\cC_m$ respectively. The function $\varpi$ is such that $\varpi(t,w,s,\xi)>0$ on $[0,T]\times \R_{++}\times \R_{++}^d\times (\R^d\backslash\{0\})$ and $\varpi(t,w,s,0)=\varpi_{\xi}(t,w,s,0)=0$ for $(t,w,s)\in[0,T]\times \R_{++}\times \R_{++}^d$.

(ii) There exists a class of functions $\F_{comp}$ such that for all $u_1\in\F_{comp}$ (resp. $u_2\in\F_{comp}$ ) viscosity subsolution (resp. supersolution) of \eqref{2nd-corrector-eq} with 
$$u_1(T,w,s){\leq} u_2(T,w,s)\mbox{ for all } (w,s)\in \R^{*}_+\times\left(\R^{*}_+\right)^d,$$ one has 
$$u_1(t,w,s){\leq} u_2(t,w,s)\mbox{ for all }(t,w,s)\in \cD.$$

(iii) $u^*$ and $u_*$ defined in \eqref{eq:def-semilimits-upper_sec3} and \eqref{eq:def-semilimits-lower_sec3} belong to $\F_{comp}$.
\end{assumption}

\begin{remark}
\label{rem:definition.u}
(i) Note that the functions $u^*$ and $u_*$ are in fact only defined inside the domain $\cD$, then extended to $\cD_T$ by upper and lower semi continuity. We prove in Proposition~\ref{prop:term} that this extension is $0$.

(ii) Note also that we only assume the comparison for the second corrector equation and our Main Theorem~\ref{thm:expansion} provides existence of solutions to this PDE via the Perron's method. 

(iii) Results of \cite{I} show that the first corrector equation has indeed a unique solution with $\varpi(t,w,s,0)=\varpi_{\xi}(t,w,s,0)=0$, $\varpi(t,w,s,\xi)>0$ whenever $\xi\neq 0$. Furthermore, the growth condition for $\varpi$ and $\varpi_{\xi}$ can be obtained using \cite[Proposition 4.2]{I} and \cite[Proposition 3.4]{BM}. We study in detail in Section~\ref{s.blackScholes} the corrector equations and show the bounds on $ \varpi_{\xi\xi}$ needed to claim $\varpi\in \cC_m$ for our main example.  We also provide a factorization of $\varpi$ in some natural models. 
\end{remark}

\subsection{The Main Result}
\begin{thm}\label{thm:expansion}
Under Assumptions~\ref{assum:frictionless},~\ref{assum:phi},~\ref{assum:Ve.viscosity},~\ref{assum:bound} and~\ref{assum:2ndcorrect},
there exists a unique viscosity solution $u\in\F_{comp}$ of the second corrector equation \eqref{2nd-corrector-eq} and
 for all $(t,w,s)\in\cD$
we have the following expansion of the value function $V^\e$:
\begin{align}\label{eq:expansion}
V^\e(t,w,s,h^0(t,w,s))= V^0(t,w,s)-\e^{2m^*} u(t,w,s)+o(\e^{2m^*}),
\end{align}
where $m^*=\frac{1}{3m-2}$.
\end{thm}
The proof of this theorem is provided in Section~\ref{s.semi} after proving the necessary intermediate results. 
\begin{remark}
(i) {\color{black} A more detailed analysis of our proof in fact shows that for all $(t,w,s)\in \cD$, there exists a neighborhood (depending on $\e$) of $h^0(t,w,s)$ such that the same expansion of $V^\e(t,w,s,h)$ as in Theorem~\ref{thm:expansion} holds for any $h$ in this neighborhood. }

(ii) The term $a$ is the source term of the second corrector equation~\eqref{2nd-corrector-eq}. This equation is linear and its solution is the first order utility loss in \eqref{eq:expansion}, the function $a$ governs the first order utility loss induced by the presence of friction, and we have the Feynman-Kac representation
\begin{equation*}
u(t,w,s)=\E\bigg[\int_{t}^{T} a(r, W^{0}_r, S_r)dr~\Big |~W^{0}_t=w,\ S_t=s\bigg].
\end{equation*}

(iii) In the one-dimensional case, studied by Guasoni and Weber in \cite{guasoni.weber.16} and Cay\'e, Herdegen and Muhle-Karbe in \cite{caye.herdegen.muhlekarbe.18}, the first corrector equation simplifies to an ordinary differential equation. Its solution, a couple consisting of a function and a constant, gives similarly both the speed of trading as a function of the displacement from the frictionless optimizer and the leading order utility loss induced by the presence of frictions.

(iv) {\color{black}
 As the small parameter $\e$ in our model appears inside the function $f$ in Equation \eqref{eq:effective.price}, it corresponds to $\lambda^{\frac{1}{\alpha}}$ in \cite{guasoni.weber.16} and \cite{caye.herdegen.muhlekarbe.18}. Then, our expansion order $2m^*$ is equal to the expansion order $\frac{2}{p+2}$ found in these two articles, where the parameter $p$ corresponds to the transaction cost function degree of homogeneity ($\alpha + 1$ in our model).

(v) In the limiting case $\alpha\to 1$ (or equivalently $m \to 2$) we formally recover the order of the utility loss $\frac{1}{2}$ found in \cite{MMKS}. Note additionally that the displacement from the target position is rescaled by $\e^{m^*}$ in \eqref{eq:fastaux1} and as $m\to 2$ this converges to $\e^{1/4}$ which is the rescaling of the displacement in \cite{MMKS}.}
\end{remark}

One can show with some further work that a family of asymptotically optimal investment strategies is ``essentially'' given by
\begin{align}
\dot H^\e_t:=&\e^{-1}\Phi_x\left(S_t,\frac{-\e^{3m^*}}{V_w^0(t,W^\e_t,S_t)}\varpi_\xi\left(t,W^\e_t,S_t,\frac{H_t^\e-h^0_t}{\varepsilon^{m^*}}\right)\right)\notag\\
=&-\frac{\left(V^0_{w}(t,W^\e_t,S_t)\right)^{1-m}}{\es}\Phi_x\left(S_t,\varpi_\xi\left(t,W^\e_t,S_t,\frac{H_t^\e-h^0_t}{\varepsilon^{m^*}}\right)\right)\label{eq:optimal.strat}
\end{align}
and consumption rate 
\begin{align}
\label{eq:def.ce}
c^\e_t:=-\tilde U'(V^0_w (t,W^\e_t,S_t))=c^0(t,W^\e_t,S_t).
\end{align} 
We will not prove this claim for the brevity of presentation. Indeed, although the proof of this claim can be carried out similarly to \cite{caye.herdegen.muhlekarbe.18} under appropriate assumptions, in our case, due to the admissibility condition, one needs to modify these candidate strategies at appropriate hitting times to avoid short-selling (similar to \eqref{eq:def.tau.e}). This in turn necessitates to prove properties of the solution $\varpi$ of the first corrector equation that are beyond the scope of this work. Technical difficulties also arise when proving that stopping the strategy before the end of the trading interval happens with an asymptotically small probability. This requires to additionally stop other processes depending on the primitive of the model and obtain uniform moment existence for the renormalized displacement $\xi^\e(\cdot,W^\e_\cdot,S_\cdot,H^\e_\cdot)$. These technical issues are similar to the ones appearing in the case of proportional costs (see \cite{AK2013}, \cite{HMK2017}) or  with nonlinear price impact (see \cite{caye.herdegen.muhlekarbe.18}).

\subsection{Factorization of Corrector Equations for Black-Scholes Markets}\label{s.blackScholes}
Consider the Merton problem of Example~\ref{ex:frictionless} with constant coefficients. 
Let us compute
\begin{align*}
	\frac{d\langle \pi_i \frac{W^0_t}{S^i_t},\pi_j \frac{W^0_t}{S^j_t}\rangle}{dt} =& \frac{\pi_i\pi_j}{S^i_tS^j_t}\Big(\sum_{k,l=1}^{d}H^{0,k}_tS^k_tH^{0,l}_tS^l_t\sigma^k\cdot\sigma^l-W^0_t\sum_{k=1}^{d}S^k_tH^{0,k}_t\sigma^k\cdot\left(\sigma^i+\sigma^j\right)\\
	&\qquad\qquad\qquad\qquad\qquad\qquad\qquad\qquad+\left(W^0_t\right)^2\sigma^i\cdot\sigma^j\Big)\\
	=&\left(W^0_t\right)^2\frac{\pi_i\pi_j}{S^i_tS^j_t}\Big(\pi^\top \sigma\sigma^\top\pi-\pi^\top\sigma\sigma^\top(e_i+e_j)+e_i^\top\sigma\sigma^\top e_j\Big)\\
	=& \left(W^0_t\right)^2\frac{\pi_i\pi_j}{S^i_tS^j_t}( \mu-{r\one}-R\sigma\sigma^\top e_i)^\top \left(R^2\sigma\sigma^\top\right)^{-1}(\mu-{r\one}-R\sigma\sigma^\top e_j),
\end{align*}
where $\{e_i\}_{i=1,\ldots,d}$ is the canonical basis of $\R^d$. Then
\begin{align*}
	c^{h^0}_{i,j}(t,w,s)=\frac{w^2\pi_i\pi_j}{R^2s_is_j}( \mu-{r\one}-R\sigma\sigma^\top e_i)^\top \left(\sigma\sigma^\top\right)^{-1}(\mu-{r\one}-R\sigma\sigma^\top e_j)
\end{align*}
is the function appearing in the first corrector equation~\eqref{1st-corrector-eq} (recalling \eqref{correc1}).
Denote by $\Sigma$ the matrix whose $i$-th column is the vector 
\begin{equation*}
\pi_i\s^{-1}(\mu-r\one-R\sigma\sigma^\top e_i)=\pi_i R\s(\pi-e_i),\quad \mbox{for}\quad 1\leqslant i\leqslant d. 
\end{equation*}

{\color{black}
We now state a Proposition on the factorization of the corrector equations for the geometric Brownian motion.
\begin{prop}\label{factoifsscale}
Assume that the frictionless problem is the Merton problem with constant coefficients described in Example~\ref{ex:frictionless} and the Assumptions of Theorem~\ref{thm:expansion} hold. Assume also that 
\begin{align}\label{eq:propphi}
\Phi(s,x)=\tilde \Phi\left(\frac{x}{s}\right)
\end{align}
and the elliptic equation 
\begin{align}
\label{1st-corrector-eq-bis-cstt2}
\frac{R}{2}|\s x|^2-\tilde \Phi(\tilde\varpi_x(x))+&\frac{1}{2R^2}\text{Tr}\left(\tilde \varpi_{xx}(x)\Sigma^\top\Sigma\right)=\lambda,\mbox{ for }x \in\R^d,
\end{align} 
admits a unique solution $(\tilde \varpi,\lambda)$ satisfying $\tilde \varpi(0)=0$ and
\begin{align}\label{eq:boundtildevarpi}		\sup_{x \in\R^d}\left(\frac{\left|\tilde \varpi(x)\right|+\left|\tilde \varpi_{xx}(x)\right|}{\left(1+|x|^2\right)^{\left(\frac{1}{2}+\frac{1}{m}\right)}}+\frac{\left|\tilde \varpi_x(x)\right|}{\left(1+|x|^2\right)^{\frac{1}{m}}}\right)<\infty.
\end{align} 
Then, the solution of the first corrector equation has the form
\begin{align}\label{eq:factorisation3}
	\varpi(t,w,s,\xi)&=g(t) w^{1-R+4m^*}\tilde \varpi\left(\frac{\xi\times s}{w^{1+m^*}}\right),\\
	a(t,w,s)&=\lambda g(t)w^{3mm^*-R}\mbox{ for all }(t,w,s)\in [0,T]\times \R_{++}\times \R_{++}^d.\notag
\end{align}
and the expansion of the value function is 
\begin{align}\label{eq:expansioninprop}
V^\e(t,w,s,h^0(t,w,s))=U(w)\left(g(t)-\lambda (1-R)(w\e)^{2m^*} \bar g(t)\right)+ o(\ets),
\end{align}
where the function $\bar g$ solves the linear ODE 
\begin{align}
&\bar{g}'(t) +\bar{g}(t)\left[-\beta g(t)^{-\frac{1}{R}}+\beta r+\left[\frac{\beta}{R}+\frac{\beta(\beta-1)}{2R^2}\right]  \left(\mu-r\one\right)^\top\big(\sigma\sigma^\top\big)^{-1}\left(\mu-r\one\right)\right]=- g(t),\notag\\
&\bar{g}(T)=0,\label{eq:bar.g}
\end{align}
$\beta=3mm^*-R$
and has the explicit expression
\begin{equation}\label{sol:tildeg}
	\bar{g}(t)=\int_{t}^{T}g(s)\exp\bigg(\int_{t}^{s} \Big[-\beta g(u)^{-\frac{1}{R}}+\beta r+\bigg[\frac{\beta}{R}+\frac{\beta(\beta-1)}{2R^2}\bigg] \left(\mu-r\one\right)^\top\left(\sigma\sigma^\top\right)^{-1}\left(\mu-r\one\right)\Big]du\bigg)ds.
\end{equation}

\end{prop}
\begin{remark}
(i) The function $g$ is bounded from above and away from $0$ on $[0,T]$, therefore, so is $\bar{g}$.

(ii) Similarly to \cite{guasoni.weber.16,MMKS, caye.herdegen.muhlekarbe.18}, the utility loss is proportional to a constant that we denote $\lambda$ which is the only term in the expansion that depends on the expression of $f$($\bar g$ depends on $m>0$ which is assumed to be fixed). In order to understand the joint effect of correlation and impact functions on $\lambda$ (and therefore the expansion), we need to study an ergodic type Hamilton-Jacobi-Bellman equation, \eqref{1st-corrector-eq-bis-cstt2}. This equation was studied in \cite{cirant.2014, I}. Although we cannot explicitly compute $\tilde \varpi$, the equation \eqref{1st-corrector-eq-bis-cstt2} fully describes the mechanism of how $\s$ and $\Phi$ affect the expansion. We provide below a choice of $\tilde \Phi$ which allows an additional simplification of the solution of this PDE as sum of one dimensional functions. However, in general we do not expect explicit solution to this PDE.

(iii) The trading direction of the asymptotically optimal strategies in \eqref{eq:optimal.strat} are determined by the composition of the two functions $\Phi_x$ and $\varpi_\xi$. The direction to which $\varpi_\xi$ points is determined by the one $\tilde \varpi_x$ due to the factorization \eqref{eq:factorisation3}. For example, the optimal strategies trading directions would point to the origin if the composition of the vector fields $\tilde \varpi_x$ and $\Phi_x$ would point to the origin. Thus, to understand how the direction of optimal portfolio depends on the various data of the problem one needs to compute $\tilde \varpi$ numerically using the methods in \cite{cacace.camilli.2016}. 

\end{remark}
}
\begin{proof}
Using the expression \eqref{eq:val.frictionless} for the frictionless value and the factorization of $\Phi$, the first corrector equation yields that the function $\tilde \varpi$ defined by \eqref{eq:factorisation3} that might eventually depend on $t,w,s$ solves the PDE
\begin{align*}
\frac{R g(t)w^{1-R+2m^*}}{2} \Bigg|\s \frac{\xi\times s}{w^{1+m^*}}\Bigg|^2&-g(t)w^{-R+3mm^*}\tilde \Phi \bigg(\tilde \varpi_x\left(\frac{\xi\times s}{w^{1+m^*}}\right)\bigg)\\
&+\frac{g(t)w^{-R+3mm^*}}{2R^2}\text{Tr}\left(\tilde \varpi_{xx}\left(\frac{\xi\times s}{w^{1+m^*}}\right)\Sigma\Sigma^\top\right)=a(t,w,s).
\end{align*}
Given our uniqueness assumption for \eqref{1st-corrector-eq-bis-cstt2}, we have that $\tilde \varpi$ does not depend on $t,w,s$ and $a(t,w,s)=\lambda g(t)w^{-R+3mm^*}$
where $(\tilde\varpi,\lambda)$ is given by the solution of \eqref{1st-corrector-eq-bis-cstt2}.
For the second corrector equation, given the fact that $a$ does not depend on $s$, we make the following ansatz
\begin{equation}
\label{eq:def.u.example}
	u(t,w,s)=\lambda w^{3mm^*-R}\bar{g}(t)=U(w)(1-R)\lambda w^{2m^*}\bar{g}(t).
\end{equation}
Then, plugging this in \eqref{correc2}, with the optimal values obtained for $h^0$ in \eqref{eq:strat.frictionless}, we obtain the linear ODE \eqref{eq:bar.g} for $\bar g$. The solution of this Riccati equation is \eqref{sol:tildeg}.

\end{proof}

\section{Examples of Impact Functions}\label{sec:exampl}
\subsection{Main Example of Impact Function}
The main example that we fully treat is the price impact function
\begin{equation}
	\label{eq:def.f.example}
	f_j(s,\theta)=\ka\frac{m-1}{m} s_j \left(\s\left(\s ({\theta}\times{s})\right)^{(\frac{1}{m-1})}\right)_j\mbox{ for some }\ka>0,
\end{equation}
{\color{black}Whether the impact function \eqref{eq:def.f.example} holds for a given market is an empirical question that is beyond the scope of this paper.} To motivate this choice of the price impact function, let us look at the case $m=2$. For this choice of parameter one recovers the quadratic transaction cost of Garleanu \& Pedersen~\cite{GP} (where the impact matrix and the covariation of the market are assumed to be proportional) and Guasoni \& Weber~\cite{GWM}:
\begin{align}
\label{eq:gp-gw-costfunctional}
\theta\cdot f(s,\theta)=\frac{\ka}{2} |\s (\theta \times s)|^2=\frac{\ka}{2}(\theta \times s)^\top \s^2 (\theta \times s)=\frac{\ka}{2}\sum_{j=1}^{d} |(\s (\theta \times s))_{j}|^{2}.
\end{align}
{\color{black} In Garleanu \& Pedersen~\cite{GP} , the rationale for such a cost functional stems from the risk taken by the counterparts of our investor: to accept to hold these assets until they find an end buyer, the market makers must be compensated for risk. This (instantaneous) cost functional quantifies the reward necessary for these market maker to enter the transaction. In \cite{GP}, this reward is quadratic in the trading speed. The price impact function that we present in this section makes the choice of an $\alpha + 1 $-homogeneity for the cost  
\begin{align*}
\theta\cdot f(s,\theta)=\ka\frac{m-1}{m}\sum_{j=1}^{d} |(\s (\theta \times s))_{j}|^{\frac{m}{m-1}}=\ka\frac{m-1}{m}\sum_{j=1}^{d} |(\s (\theta \times s))_{j}|^{1+\a}.
\end{align*}
The second and most important reason for the choice of \eqref{eq:def.f.example} is the fact that for this choice of impact function, the solution to the first corrector equation can be written as sums of one dimensional functions. We use this property of the solution to obtain second derivative estimates on $\varpi$ and show that it is $ \cC_m$ as needed for Assumption~\ref{assum:2ndcorrect}.  Indeed, for this choice of $f$, $\Phi$ is 
\begin{equation}
	\Phi(s,x):=\sup_{\theta\in \R^d}\Bigg\{x\cdot\theta-\ka\frac{m-1}{m}\sum_{j=1}^{d} |(\s (\theta \times s))_{j}|^{\frac{m}{m-1}}\Bigg\}=\frac{1}{m\ka^{m-1}}\sum_{j=1}^d  \Big|\Big(\s^{-1}\Big(\frac{x}{s}\Big)\Big)_j\Big|^{m},\label{eq:Phix.example}
\end{equation}
where $\left(\frac{x}{s}\right)$ stands for $\left(\frac{x_1}{s_1},...,\frac{x_d}{s_d}\right)^\top$ and $\Phi$ has the same form as \eqref{eq:propphi}. 
Therefore, defining 
\begin{align}\label{eq:varpiaux}
\tilde \Phi(x)=\frac{1}{m\ka^{m-1}}\sum_{j=1}^d  \Big|\Big(\s^{-1}\Big({x}\Big)\Big)_j\Big|^{m}
\end{align}
we can use Proposition \ref{factoifsscale} to obtain a first factorization of the first corrector equation. 
Then, we show that \eqref{eq:varpiaux} allows us to write $\tilde \varpi$ defined as the solution \eqref{1st-corrector-eq-bis-cstt2} as sum of one dimensional functions. The form of $f$ in \eqref{eq:def.f.example} is also in line with \cite[Example 3.1]{PST} where a form is postulated for the first corrector equation directly. We provide in Subsection~\ref{ss.otherexample} other examples of impact function and mention what type of estimates are needed for these problems. For the general setting, the second derivative estimates for $\varpi$ needed for Assumption~\ref{assum:2ndcorrect} are believed to be true, but their proof is far beyond the scope of this article and would require one of its own. Therefore, we choose to illustrate our main result for \eqref{eq:def.f.example} only.

 }
 
\subsubsection{Corrector Equations for the Main Example}
For the example \eqref{eq:def.f.example}, in order to simplify the impact of the correlation structure in the term ${\frac{R}{2}|\s x|^2}$ of \eqref{1st-corrector-eq-bis-cstt2}, we slightly change the ansatz \eqref{eq:factorisation3}  to 
$$\varpi(t,w,s,\xi)=g(t) w^{1-R+4m^*}\tilde \varpi\left(\frac{\s(\xi\times s)}{w^{1+m^*}}\right).$$
Then, \eqref{1st-corrector-eq-bis-cstt2} becomes
\begin{align}\label{1st-corrector-eq-bis}
\sum_{j=1}^{d}\left(\frac{R}{2} x_j^2-\frac{1}{m}\ka^{1-m}| { \tilde\varpi_{x_j}(x)} |^m\right)+\frac{1}{2R^2}\text{Tr}\Big(\tilde \varpi_{xx}(x)(\Sigma\s)^\top\Sigma\s\Big)=\lambda
\end{align}
We will solve this equation in terms of $\tilde\varpi^1$, which is the solution of an ODE
\begin{align}\label{eq:first1d}
 (\tilde \varpi^1)''(x)=-x^2+\lambda_m+\frac{m^{-m}}{(m-1)^{1-m}}|(\tilde\varpi^1)'(x)|^m,\mbox{ for }x\in \R,\mbox{ and }\tilde \varpi^1(0)=0,
 \end{align}
where $\lambda_m>0$ is the unique constant such that $\lim_{x\to\pm\infty}\frac{(\tilde\varpi^1)'(x)}{|x|^{2/m}}=\pm m(m-1)^{\frac{1}{m}-1}$. 
The existence of such a solution can be provided by two different methods. One can either note that this equation is in fact \eqref{1st-corrector-eq-bis} reduced to one dimension and use the theory developed in \cite{I},\cite{BM} and \cite{cirant.2014} or, alternatively, remark that
$-(\tilde\varpi^1)'$ is in fact equal to the function defined in \cite[Theorem 6]{guasoni.weber.16} and that $\lambda_m$ is the constant $c_\a$ defined in the same theorem. We now provide additional properties of $\tilde \varpi^1$.

\begin{lemma}\label{lem:bounded.second.derivative}
	The function  $\tilde\varpi^1$ defined by the ODE \eqref{eq:first1d} is convex and has a bounded second order derivative. 
\end{lemma}

\begin{proof}{\color{black} By symmetry, we only show the bound at $+\infty$. Since it is an anti-derivative of an odd, increasing function  (see e.g. \cite[Lemma 3.1]{caye.herdegen.muhlekarbe.18}), $\tilde \varpi^1$ is convex. Therefore we only need an upper bound for $(\tilde \varpi^1)''$.}
The second derivative being a function of the first derivative and $m>2$, we have that $\tilde \varpi^1$ is four times continuously differentiable. $\tilde \varpi^1$ being subquadratic the limit of $(\tilde \varpi^1)''$ at infinity cannot be infinity. Thus, there exists $M>0$ and $y_n\uparrow\infty$ such that $(\tilde \varpi^1)''(y_n)\leq M$. Assume that $(\tilde \varpi^1)''$ is not bounded, meaning, there exists $x_n\to \infty$ such that $ (\tilde \varpi^1)''(x_n)\uparrow \infty$ and $(\tilde \varpi^1)''(x_n)>M$. Thus, for all $x_n$ there exists $y_{m_1(n)}$ and $y_{m_2(n)}$ such that $y_{m_1(n)}\leq x_n\leq y_{m_2(n)}$. Note that $(\tilde \varpi^1)''$ has a local maximum on $[y_{m_1(n)},y_{m_2(n)}]$. Denote $\tilde x_n$ this local maximum and note that $y_{m_1(n)}<\tilde x_n<y_{m_2(n)}$ and $\tilde\varpi'''(\tilde x_n)=0$. We differentiate \eqref{eq:first1d} twice to obtain that at the local maximum $\tilde x_n$ of $(\tilde \varpi^1)''$, we have
$$0\geq  (\tilde \varpi^1)''''(x)=-2+C|(\tilde\varpi^1)'(\tilde x_n)|^{m-2}|(\tilde\varpi^1)''(\tilde x_n)|^{2},$$
for a constant $C$ that only depends on $m$. Thus, 
$$|(\tilde\varpi^1)'(\tilde x_n)|^{m-2} |(\tilde \varpi^1)''(\tilde x_n)|^2\leqslant \frac{2}{C}.$$
The growth of $(\tilde\varpi^1)'$ and $m>2$ gives that for some $n$ large enough the assumption $(\tilde \varpi^1)''(x_n)>M$ is contradicted. We conclude that $(\tilde \varpi^1)''$ is bounded. Note  that repeating the procedure, replacing the $y_n$'s by the $\tilde x_n$'s, we can prove that $(\tilde \varpi^1)''$ converges to $0$ at infinity.
\end{proof}

\begin{remark}
 The result of Lemma~\ref{lem:bounded.second.derivative} is actually stronger than what is necessary for the analysis in the rest of the article. To prove Theorem~\ref{thm:expansion}, the growth conditions defined in the class $\cC_m$ (cf. Definition~\ref{def:Cm}) are enough, i.e. the ones stated below in \eqref{eq:growth.condition.varpi1}.
	
\end{remark}

We now give the following lemma for the wellposedness of the reduced first corrector equation \eqref{1st-corrector-eq-bis}. 
\begin{lemma}
\label{lem:growth.varpi.tilde}
Provided that the diagonal terms of $\left((\Sigma\s)^\top(\Sigma\s)\right)$ are all positive, there exists a unique solution, denoted by  $(\lambda,\tilde \varpi)$, of
\eqref{1st-corrector-eq-bis} satisfying $\tilde \varpi(0)=0$, $\tilde \varpi\geqslant 0$ and $\liminf_{|x|\to \infty}\varpi(x)>0$. 
This solution is given by
$\tilde \varpi(x)=\sum_{j=1}^d \beta_j \tilde\varpi^1(\gamma_j x_j)$, where $\tilde\varpi^1$ is the solution of \eqref{eq:first1d} and 
\begin{align}
\gamma_j&=\left(2\left(\frac{m}{\kappa(m-1)}\right)^{m-1}\frac{R^{3m-1}}{\left((\Sigma\s)^\top(\Sigma\s)\right)^m_{jj}}\right)^{m^*},\notag\\
\beta_j&=2^{-4m^*}\left((\Sigma\s)^\top(\Sigma\s)\right)^{4mm^*-1}_{jj}R^{-1-4m^*}\left(\frac{\kappa(m-1)}{m}\right)^{4(m-1)m^*},\notag\\
\lambda&=\lambda_m \sum_{j=1}^d  \frac{R}{2\gamma_j^2}.\label{eq:sol.lambda}
\end{align}
Here, $\lambda_m$ is also identified in \eqref{eq:first1d}. Additionally, $\tilde\varpi^1$ and
$\tilde \varpi$ are $\cC^2$, convex and satisfy the limit and bounds
\begin{align}
	&\lim_{|x|\to \infty}\frac{\tilde\varpi^1(x)}{|x|^{1+\frac{2}{m}}} = \frac{m^2}{(m+2)(m-1)^{1-\frac{1}{m}}},&&\label{eq:lim.growth.tilde.varpi1}\\
	&\frac{(1+|x|^2)^{\frac{1}{2}+\frac{1}{m}} }{C}-C\leqslant \tilde \varpi(x) \leqslant C(1+|x|^2)^{\frac{1}{2}+\frac{1}{m}}, &&\label{eq:growth.condition.varpi1}\\
	&\mbox{and}\sup_{i\in\{1,...,d\},x\in \R^d}\left|\frac{\tilde \varpi_{x_i} (x)x_i}{\tilde \varpi(x)}\right|+\left|\frac{\tilde \varpi_{x_i}(x)}{|x|}\right|+\left|\frac{\tilde \varpi_{x_i}(x)}{1+|x|^{\frac{2}{m}}}\right|<\infty,\label{eq:growth.condition.varpi2}
\end{align}
where $C$ is a positive constant.

\end{lemma}
\begin{proof}
	By verification, we see that the given function solves the PDE \eqref{1st-corrector-eq-bis}. Since the solution of this PDE is unique (this is a consequence of \cite[Theorem 4.14]{I}), our candidate is the solution of \eqref{1st-corrector-eq-bis}. For the limits, let $\eta>0$ be arbitrarily small. By \cite{guasoni.weber.16}, there exists $x_\eta>0$ such that for all $x\geqslant x_\eta$ it holds
	\begin{equation}
	\label{eq:derivative.varpi1.bound}
		\left(\frac{m}{(m-1)^{1-\frac{1}{m}}}-\eta\right)x^{\frac{2}{m}}\leqslant(\tilde\varpi^1)'(x)\leqslant \left(\frac{m}{(m-1)^{1-\frac{1}{m}}}+\eta\right)x^{\frac{2}{m}}.
	\end{equation}
	Then integrating between $0$ and $x$ for $x\geqslant x_\eta$ we obtain,
	\begin{equation*}
		C^-_\eta+\left(\frac{m}{(m-1)^{1-\frac{1}{m}}}-\eta\right)\frac{m}{m+2}x^{1+\frac{2}{m}} \leqslant \tilde\varpi^1(x)\leqslant C^+_\eta+ \left(\frac{m}{(m-1)^{1-\frac{1}{m}}}+\eta\right)\frac{m}{m+2}x^{1+\frac{2}{m}},
	\end{equation*}
	where $C^\pm_\eta=\int_0^{x_\eta}(\tilde\varpi^1)'(y)dy- \Big(\frac{m}{(m-1)^{1-\frac{1}{m}}}\pm\eta\Big)\frac{m}{m+2}x_\eta^{1+\frac{2}{m}}$. As $\eta$ was arbitrary we obtain the growth behaviour of $\tilde\varpi^1$ at $+\infty$. The reasoning for its growth behaviour at $-\infty$ is exactly the same. Now take $x=(x_1,...,x_d)\in\cS^d$, and $r>0$. We have
	\begin{align}
		\label{eq:lim.growth.tilde.varpi}
		\frac{\tilde \varpi(rx)}{r^{1+\frac{2}{m}}}=\sum_{i=1}^{d}\beta_j\frac{\tilde\varpi^1(\gamma_jrx_j)}{r^{1+\frac{2}{m}}},
	\end{align}
	and then, using the growth behaviour of $\tilde \varpi^1$, we obtain
	\begin{align}
		\lim_{r\to\infty}\frac{\tilde \varpi(rx)}{r^{1+\frac{2}{m}}}=\frac{m^2}{(m+2)(m-1)^{1-\frac{1}{m}}} \sum_{i=1}^{d}\beta_j|\gamma_jx_j|^{1+\frac{2}{m}}.
	\end{align}
	Finally, \eqref{eq:growth.condition.varpi1} is a consequence of~\eqref{eq:lim.growth.tilde.varpi}, \eqref{eq:derivative.varpi1.bound} and the same reasoning on $\cS^d$ as above. Then \eqref{eq:growth.condition.varpi2} is a consequence of the linear growth of $(\tilde \varpi^1)'$ around $0$ and its growth at infinity given in \cite[Theorem 4]{guasoni.weber.16}.
\end{proof}
We summarize here our results for the price impact function \eqref{eq:def.f.example}.
\begin{thm}\label{th:example}
Let $\mu,\sigma, r$ be such that the strategy \eqref{eq:strat.frictionless} does not have any short selling and borrowing, meaning $\pi_i>0$, for all $1\leqslant i\leqslant d$ and $\sum_{i=1}^d  \pi_i<1$. Assume also that the price impact is as in \eqref{eq:def.f.example} and $0<R<1$. Then for all $(t,w,s)\in \cD$, the following expansion of the value function holds 
\begin{align*}
V^\e(t,w,s,h^0(t,w,s))=U(w)\left(g(t)-\lambda (1-R)(w\e)^{2m^*} \bar g(t)\right)+ o(\ets),
\end{align*}
where $g$ is given by \eqref{eq:g}, $\bar g$ solves \eqref{eq:bar.g} and $\lambda$ is given by \eqref{eq:sol.lambda}.
\end{thm}
\begin{proof}
Note that from the closed form solutions obtained for Example~\ref{ex:frictionless}, and the assumptions of the Theorem, Assumption~\ref{assum:frictionless} is satisfied. The definition of the price impact function in \eqref{eq:def.f.example} satisfies Assumption~\ref{assum:phi}. By Lemma~\ref{lem:growth.varpi.tilde}, Assumption~\ref{assum:2ndcorrect} (i) is satisfied. By the weak dynamic programming result of~\cite[Corollary 5.6]{bouchard.touzi.11}, Assumption~\ref{assum:Ve.viscosity} is satisfied. Now, considering the computations above, in order to use our main theorem and Proposition~\ref{factoifsscale}, we need to show that Assumptions~\ref{assum:bound} and the last three items of Assumption~\ref{assum:2ndcorrect} are satisfied, i.e. (i) define a set $\F_{comp}$ containing the function $u$ defined in~\eqref{eq:def.u.example} and prove that it has the properties of Assumption~\ref{assum:2ndcorrect}, (ii) prove the bound~\eqref{eq:control-deviation}. These are the aim of Lemma~\ref{lem:ex.Fcomp} and Proposition~\ref{prop:control-deviation} below.
\end{proof}
\begin{remark}
	\label{rem:barg.bounded}
	{\color{black}

(i) Similarly to \cite{guasoni.weber.16} and \cite{caye.herdegen.muhlekarbe.18}, the utility loss is proportional to a constant found as a part of the solution of the relevant ODE ($c_\alpha$ or $c_p$ respectively in these articles).

(ii) The main computational simplification of our choice \eqref{eq:def.f.example} is that we can compute $\lambda$ as in \eqref{eq:sol.lambda} and obtain the second derivative estimates in \eqref{eq:boundtildevarpi} due to the one dimensional factorization of $\tilde \varpi$ proven in Lemma~\ref{lem:growth.varpi.tilde}. In this case $\lambda$ and therefore the utility loss is proportional to $\kappa^{2(m-1)m^*}$ which represents the size of the price impact in \eqref{eq:def.f.example}.

(iii) The restriction $0<R<1$ is a technical condition and is needed to obtain the bounds \eqref{eq:control-deviation} in the Appendix. Under this condition, we can easily control the utility at final time by increasing the consumption.\footnote{For $R>1$ cumbersome estimates would be needed, and we omit them so as not to drown the already complicated analysis in more technical details.}

(iv) Given the choice of impact function \eqref{eq:def.f.example}, we can show that the state that is controlled via conjectured asymptotically optimal controls has a particular behavior. Indeed one can show that with this control, each component of the vector $\s\frac{H_t^\e-h^0_t}{\varepsilon^{m^*}} $ representing the deviation from the target position multiplied by $\s$, has a one dimensional behavior at the leading order. The factorization of the function $\varpi$ given in Lemma~\ref{lem:growth.varpi.tilde} is then not surprising. For the case $m=2$, at the leading order the mean reversion speed of every component is equal. Thus, we obtain that any direction is a principal portfolio in the sense of \cite{guasoni.weber.16} and the asymptotically optimal strategy points to the target position. However, in the case $m>2$, a straightforward computation shows that, although each component of $\s\frac{H_t^\e-h^0_t}{\varepsilon^{m^*}} $ solves a one dimensional ODE at the leading order, the local mean reversion speed of each component to $0$ might be different. Therefore the conjectured asymptotically optimal portfolio might not locally point to the target position. 

(v) Unlike the one dimensional framework of \cite{guasoni.weber.16,caye.herdegen.muhlekarbe.18} where the expansion is characterized by the solution of a unique ODE \eqref{eq:first1d} that does not depend on the impact function but only on $\alpha$, in the general multidimensional case for each impact function one might need to solve a PDE. It would be interesting to see whether the PDE \eqref{1st-corrector-eq-bis-cstt2} would admit further simplification for a fairly large class of impact functions.   
}
\end{remark}

\subsubsection{Verification of Assumptions~\ref{assum:bound} and~\ref{assum:2ndcorrect} for the Main Example}\label{ss.bounds}

\begin{lemma}\label{lem:ex.Fcomp}
	The class of function
	\begin{align}\label{def:compclass}
		\F_{comp}=\left\{\phi:\cD\mapsto \R: \exists k>0: \sup_{(t,w,s)\in \cD}\frac{|\phi(t,w,s)|}{1+w^k+w^{-k}+\sum_{i=1}^{d}(s_i^k+s_i^{-k})} <\infty\right\}
	\end{align}
	has the comparison property defined in Assumption~\ref{assum:2ndcorrect}. 
\end{lemma}

Note that if a function $\phi\in \F_{comp}$ satisfies the boundedness assumption for a $k>0$, it also satisfies it for all $k'\geqslant k$. This is due to the fact that 
$$\sup_{w>0}\frac{1+w^k}{1+w^{k'}}+\sup_{w>0}\frac{1+w^{-k}}{1+w^{-k'}}<\infty.$$
\begin{proof}
	We now show comparison of viscosity solutions within $\F_{comp}$. Let $u_1, u_2\in\F_{comp}$ be respectively viscosity subsolution and supersolution of \eqref{2nd-corrector-eq} such that $u_1(T,\cdot ){\leq} u_2(T,\cdot)$ on $\R_{++}\times \R_{++}^d$. Take $k>0$ such that 
$$ \sup_{j=1,2}\sup_{(t,w,s)\in \cD}\frac{|u_j(t,w,s)|}{1+w^k+w^{-k}+\sum_{i=1}^{d}(s_i^k+s_i^{-k})} <\infty.$$

A direct computation shows that for $\ell>0$ large enough the function 
\begin{equation*}
	\Gamma:(t,w,s)\mapsto e^{-\ell t}\Big(1+w^{2k}+w^{-2k}+\sum_{i=1}^{d}\big(s_i^{2k}+s_i^{-2k}\big)\Big)
\end{equation*}
is a viscosity supersolution of {\eqref{2nd-corrector-eq}}. The argument to show the comparison property on $\cD$ is the standard dedoubling of variables technique as in the proof of Theorem 4.4.4 and Theorem 4.4.3  in \cite{P09}. By definition of $\Gamma$ we have that for all $\d>0$, 
\begin{align*}	
	&\lim_{w\to 0,\ w>0}u_1(t,w,s)-u_2(t,w,s)-\d\Gamma(t,w,s) =-\infty,\\ &\lim_{s_i\to 0,\ s_i>0}u_1(t,w,s)-u_2(t,w,s)-\d\Gamma(t,w,s) =-\infty,\ \mbox{for}\ 1\leqslant i\leqslant d,\\
	&\lim_{w\to +\infty}u_1(t,w,s)-u_2(t,w,s)-\d\Gamma(t,w,s)  =-\infty,\\ &\lim_{s_i\to +\infty}u_1(t,w,s)-u_2(t,w,s)-\d\Gamma(t,w,s)  =-\infty,\ \mbox{for}\ 1\leqslant i\leqslant d.
\end{align*}
Thus defining
\begin{align*}
	\Phi_\d(t,t',w, w',s,s') =& u_1(t,w,s)-u_2(t',w',s')-\d\Gamma(t',w',s')\\
	&-\frac{1}{2\d}\left(\left|t-t'\right|^2+\left|w-w'\right|^2 +\left|s-s'\right|^2\right),
\end{align*}
the maximizers  $(t_\d,t_\d',w_\d, w_\d',s_\d,s_\d')$ of $\Phi_\d$ exists for all $\delta>0$. 
One can now conclude similarly to the proof of Theorem 4.4.4.  in \cite{P09}.
\end{proof}

The remaining ingredient for the proof of Theorem~\ref{th:example} is to check the Assumption~\ref{assum:bound} and to show that $u^*$ defined in \eqref{eq:def-semilimits-upper2} is in $\F_{comp}$ defined by \eqref{def:compclass} which is proved in Proposition~\ref{prop:control-deviation} below.

\begin{prop}\label{prop:control-deviation}
	Assume that the assumptions of Theorem~\ref{th:example} hold. Then, for all $(t_0,w_0,s_0)\in \cD$ there exists $\e_0>0$ and $r_0>0$ such that 
	\begin{align}\label{eq:control-deviation-bis}
		\sup\bigg\{&\frac{V^0(t,w,s)-V^\e(t,w,s,h)}{\e^{2m^*}}:\\
		&\qquad|t-t_0|+|w-w_0|+|s-s_0|+|h-h^0(t_0,w_0,s_0)|{\leq} r_0,~ \e\in(0,\e_0)\bigg\}<\infty,\notag
	\end{align}
	and  $u^*$ defined in \eqref{eq:def-semilimits-upper2} is in $\F_{comp}$.
	\end{prop}
The proof of \eqref{eq:control-deviation-bis}, which is very technical, is provided in the Appendix. 

\subsection{Further Examples of Impact Functions}\label{ss.otherexample}
\begin{example}\label{ex:2}
Another possible price impact is 
$f_j(s,\theta)=\ka_j{s^{\frac{m}{m-1}}_j\theta_j \left| \theta_j \right|^{\frac{2-m}{m-1}}}$ where the impact on each asset price depends on the trading on this asset. 
In this case, 
$\Phi$ only depends on the vector of the ratios $\left(\frac{x_1}{s_1},...,\frac{x_d}{s_d}\right)^\top$,
\begin{equation*}
\Phi(s,x)=\sum_{i=1}^d \frac{(m-1)^{m-1}}{m^m \kappa_i^{m-1}}\left|\frac{x_i}{s_i}\right|^{\color{black}m}=:\tilde \Phi\left(\frac{x}{s}\right).
\end{equation*}
Similarly to the main example with price impact function \eqref{eq:def.f.example} in the case of a Black-Scholes market, this scaling in $s$ allows to use Proposition~\ref{factoifsscale}. If $\s$ is not diagonal, the couple $(\tilde \varpi, \lambda)$ still solves an equation similar to \eqref{1st-corrector-eq-bis-cstt2} but they cannot be expressed as sums of one dimensional functions.  We are not able to compute $\lambda$ explicitly.
\end{example}
\begin{remark}
	{\color{black}We are not able to fully treat this example. Indeed, currently, there are no estimates available in the literature for the growth of the second derivative of $\tilde \varpi$ in \eqref{eq:boundtildevarpi}. These estimates are necessary to check the Assumption \ref{assum:2ndcorrect} and to proceed with the proof of the estimates of Section~\ref{s.appendix}.}
		
	However, under appropriate assumptions, one can use a refined version of the so called adjoint method to show for Example~\ref{ex:2}, that the second derivative of $\varpi$ satisfies the bound\footnote{This has been pointed out to us by Marco Cirant and is the topic of an ongoing work with him.} 
	\begin{equation*}
		\frac{|\varpi_{\xi\xi}(\xi)|}{(1+|\xi|^2)^{\frac{1}{2}+\frac{1}{m}} }{\leqslant}  K. 
	\end{equation*}
{\color{black}Then, one can use the methodology presented here to obtain the expansion of the utility loss. }
\end{remark}

 \begin{example}\label{ex:3}
The simplest choices for  $f_j$ are given by $f_j(s,\theta)=\ka_j{\theta_j \left|\theta_j \right|^{\frac{2-m}{m-1}}}$ 
and $f_j(s,\theta)=\ka_j{s_j\theta_j \left|\theta_j \right|^{\frac{2-m}{m-1}}}$ for some $\ka_j>0$ and $m>2$ that lead respectively to an additive and multiplicative price impact independent of the stock price. 
However, in these cases, the first corrector equation \eqref{1st-corrector-eq} has a complicated dependence on $t,w$ and $s$ and the factorization of Proposition~\ref{factoifsscale} is not possible. Thus, unlike the first two examples where one needed to solve a unique first corrector equation, it is necessary here to solve a first corrector equation for each $(t,w,s)\in \cD$ (and prove smooth dependence on $(t,w,s)$ of the solution, see Assumption~\ref{assum:2ndcorrect}).
\end{example}

\section{The Remainder Estimates}
\label{s.remainder}
In this section we gather some results that will be useful in the proofs of Proposition~\ref{prop:control-deviation} (especially  Lemma~\ref{lem:dev-control}), and Propositions~\ref{prop:supersol} and~\ref{prop:subsol}. Let us define the rescaled displacement function $\xi^\e$ for $\e>0$ as
\begin{equation}
\xi^\e:=\xi^\e(t,w,s,h)=\frac{h-h^0(t,w,s)}{\es}.
\end{equation}
	Let us denote by $(\nu^\e)_{\e>0}$ a family of functions in $\cC^{1,2,2,2}$ and $(\chi^\e)_{\e>0}$ be a family of functions in $\cC_m$ (recall Definition~\ref{def:Cm}) and define 
	\begin{equation}\label{eq:defpsiee}
	\psi^\e(t,w,s,h):=V^0(t,w,s)-\e^{2m^*}\left(\nu^\e (t,w,s,h)+\e^{2m^*}\chi^\e\left(t,w,s,\frac{h-h^0(t,w,s)}{\e^{m^*}}\right)\right).
	\end{equation}

\begin{prop}
	\label{prop:remainder.estimate}
	For $\e\in(0,1)$, let $\nu^\e\in\cC^{1,2,2,2}$ and $\chi^\varepsilon \in\cC_m$ be real-valued functions and $\psi^\e$ defined as in \eqref{eq:defpsiee} with $\nu^\e$ and $\chi^\varepsilon$. We assume that $\nu^\e$ and its first two derivatives admit bounds uniform in $\e>0$, i.e. $\sup_{\e>0}|\nu^\e_{a}(t,w,s,h)|+|\nu^\e_{bc}(t,w,s,h)|\leqslant D(t,w,s,h)$ where $a\in\{t,w,s,h\}$, $b,c\in\{w,s,h\}$ and $D$ is a locally bounded function.
	
	Then
	\begin{align}
	\cG^\e(\psi^\e)(t,w,s,h)&=R^\e(t,w,s,h,{\color{black}\psi^\e})+\e^{2m^*}\Big(E_2\big(t,w,s,\nu^\e_t, \nu^\e_w, \nu^\e_s,\nu^\e_{ww},\nu^\e_{ws},\nu^\e_{ss}\big)\notag\\
	&\quad+E_1\big(t,w,s,\xi^\e,\chi^\e_\xi(t,w,s,\xi^\e),\chi^\e_{\xi\xi}(t,w,s,\xi^\e)\big)\notag\\
	&\quad+(\pa_w \psi^\e)^{1-m}(\Phi(s,\chi^\e_\xi)-\Phi(s,\chi^\e_{\xi}+\ems \nu^\e_h))\Big),\label{eq:remainder.viscosity}
	\end{align}
	where $\cG^\e$ was defined in \eqref{eq:PDE-veps} and $R^\e$ satisfies the following properties on the set 
	\begin{align}\label{negpower}
	\{\pa_w \psi^\e>\zeta_1\}\cap \left\{\frac{V^0_w-\pa_w \psi^\e}{V^0_w}<\zeta_2\right\}\mbox{ for some }\zeta_1>0 \mbox{ and }\zeta_2<1.
	\end{align}
	\begin{enumerate}
		\item[(i)]Let $(\bar t, \bar w, \bar s)\in\cD$ and $r>0$. Assume that $\nu^\e$ and $\chi^\e$ depend on $\e>0$ but that $\nu^\e$ does not depend on $h$ and that the derivatives of the $\chi^\varepsilon$ up to second order are uniformly bounded  on $B_{r}(\bar t, \bar w, \bar s)\times\R^d$. Then it holds on $B_{r}(\bar t, \bar w, \bar s)\times\R^d$ that
		\begin{equation*}
			\frac{|R^\e(t,w,s,h,\psi^\e)|}{\e^{2m^*}}{\leq} C (1+|h-h^0(t,w,s)|^{2}),
		\end{equation*}
		for some $C>0$ independent of $t,w,s,h$ but depending on the second derivatives bounds for $\chi$ and $(\bar t, \bar w, \bar s,r)$.
		\item[(ii)] Let $(\bar t, \bar w, \bar s,\bar h)\in\cD\times\R^d$, and $r>0$. Assume that $\chi^\e$ does not depend on $\e$ (and write $\chi$). Then it holds on $B_{r}(\bar t, \bar w, \bar s, \bar h)$ that 
		\begin{equation*}
		\left|\frac{R^\e(t,w,s,h,\psi^\e)}{\e^{2m^*}}\right|{\leq} C \left(1+|\xi^\e(t,w,s)|^2\right)^{\frac{1}{2}+\frac{1}{m}},
		\end{equation*}
		for some $C>0$ depending only on $r, \bar t, \bar w$, $ \bar s$ and $ \bar h$.
		\item[(iii)] Assume that $\chi^\e$ does not depend on $\e$ (and write it $\chi$). If the set
		\begin{align*}
			K_1=\left\{\left(t^\e,w^\e,s^\e,\frac{h^\e-h^0(t^\e,w^\e,s^\e)}{\e^{m^*}}\right):\e>0\right\}\subset \cD\times\R^d
		\end{align*} is bounded, then
		\begin{equation*}
		\left|\frac{R^\e(t^\e,w^\e,s^\e,h^\e,\psi^\e)}{\e^{2m^*}}\right|{\leq} C \e^{m^*},
		\end{equation*}
		for some $C>0$ depending only on the bound of the set $K_1$.
	\end{enumerate}
\end{prop}
\begin{remark}
Because the solution of the first corrector equation $\varpi$ is not homogeneous in $\xi$ (unlike in \cite{MMKS}) we have to include a dependence of $\nu$ on $\e$ in the result that we will use in Section~\ref{s.semi}.
\end{remark}
\begin{proof} We drop the dependence of $\nu^\e$ and $\chi^\e$ in $\e$ to simplify notations. Consider first a function $\psi^\e$ as in \eqref{eq:defpsiee}, and define the following feedback control function
	\begin{align}
		c^\e(t,w,s,h) &= -\tilde{U}'(\partial_w\psi^\e(t,w,s,h)),\label{eq:def.c.e}\\
		\tilde \theta^\e(t,w,s,h) &:=\e^{-1}\Phi_x\left(s,-\frac{\pa_h\psi^\e \left(t,w,s,h\right)}{\partial_w\psi^\e(t,w,s,h)}\right)\notag\\
		&=\ems \left(\partial_w\psi^\e\left(t,w,s,h\right)\right)^{1-m}\notag\\
		&\qquad\qquad\qquad\Phi_x\left(s,\chi_\xi\left(t,w,s,\frac{h-h^0(t,w,s)}{\varepsilon^{m^*}}\right)+\ems\nu_h(t,w,s,h)\right).\label{eq:def.theta.e}
	\end{align}
	{\color{black}Note that the trading rate $\tilde\theta^\e$ and the consumption $c^\e$ are functions of $t,w,s$ and $h$ and these functions are the maximizers of the Hamiltonian in the HJB equation \eqref{eq:HJB-veps} evaluated at $\psi^\e$.}
	
	The wealth process and strategy obtained using these controls and started at $(t,w,s,h)\in \cD\times\R^d$ are denoted $\tilde W^{\e,t,w,s,h}$ and $\tilde H^{\e,t,w,s,h}$. We have
	\begin{align*}
		d\tilde H^{\e,t,w,s,h}_u=\tilde\theta^\e(u,\tilde W^{\e,t,w,s,h}_u,S_u,\tilde H^{\e,t,w,s,h}_u)du.
	\end{align*}
	We denote the drift function  of the diffusion $\tilde\Psi^{\e,t,w,s,h}_u=\psi^\e\left(u, \tilde W^{\e,t,w,s,h}_u, S_u, \tilde H^{\e,t,w,s,h}_u\right)$ by $\tilde\mu^{\psi^\e}$. It holds
	\begin{align}\label{eq:tildemuformu}
	\tilde \mu^{\psi^\e}&(t,w,s,h) = V^0_t-\ets \nu_t-\efs \chi_t+\eths\left(h^0_t\right)^\top\chi_\xi\\
	&\notag+(V^0_w-\ets\nu_w-\efs\chi_w+\eths (h^0_w)^\top\chi_{\xi}) rw-c^\e\partial_w\psi^\e\\
	&\notag+(V^0_w-\ets\nu_w-\efs\chi_w+\eths (h^0_w)^\top\chi_{\xi})\sum_{j=1}^d h_js_j(\mu^j-r)-(\partial_w\psi^\e) \sum_{j=1}^{d}\tilde\theta^{\e,j}f^j(s,\e \tilde\theta^\e)\\
	&\notag+\frac{1}{2}(V^0_{ww}-\ets\nu_{ww}-\efs\chi_{ww})\sum_{i,j=1}^{d}h_ih_js_is_j\sigma^i\cdot \sigma^j\\
	&\notag+\left(\eths ((h^0_{w})^\top\chi_{w\xi}+\frac{1}{2}(h^0_{ww})^\top\chi_{\xi})-\frac{1}{2}\ets(h^0_w)^\top\chi_{\xi\xi}h^0_w\right)\sum_{i,j=1}^{d}h_ih_js_is_j\sigma^i\cdot \sigma^j\\
	&\notag+\cL^s (V^0-\ets\nu-\efs\chi)+\eths\sum_{i=1}^{d}(h^0_{s_i})^\top\chi_{\xi}s_i\mu^i\\
	&\notag+\frac{1}{2}\sum_{i,j=1}^{d}(\eths((h^0_{s_j})^\top\chi_{\xi s_i}+(h^0_{s_i})^\top\chi_{\xi s_j}+(h^0_{s_is_j})^\top\chi_{\xi})-\ets(h^0_{s_i})^\top\chi_{\xi \xi}h^0_{s_j})s_is_j\sigma^i\cdot\sigma^j\\
	&\notag+\sum_{i,j=1}^{d}\Big(V^0_{ws_i}-\ets\nu_{ws_i}-\efs\chi_{ws_i}+\eths ( (h^0_{s_i})^\top\chi_{\xi w} + (h^0_{w})^\top\chi_{\xi s_i} + (h^0_{ws_i})^\top\chi_{\xi} )\\
	&\notag\qquad\qquad -\ets (h^0_w)^\top\chi_{\xi\xi}h^0_{s_i} \Big)s_ih_js_j\sigma^i\cdot\sigma^j-(\eths\chi_{\xi}+\ets \nu_h)^\top\tilde\theta^\e.
	\end{align}
	Note that the functional $\cG^\e$ applied to $\psi^\e$ gives
	\begin{align}
	\label{eq:Ge.mu.psie}
	\cG^\e(\psi^\e)&(t,w,s,h) =-\Big(\tilde\mu^{\psi^\e}(t,w,s,h)+U(c^\e)\Big),
	\end{align}
	since the choices made for $\tilde \theta^\e$ and $c^\e$ provide the equalities $\tilde{U}(\partial_w\psi^\e)=U(c^\e)-\partial_w\psi^\e c^\e$ and $(\tilde\theta^\e)^\top\partial_h \psi^\e-\partial_w\psi^\e\sum_{j=1}^{d}\tilde \theta^{\e,j}f^j(s,\e\tilde\theta^\e)=\e^{-1}(\partial_w\psi^\e)^{1-m}\Phi(s,\partial_h\psi^\e)$.
	
	We now reorder the terms in $\tilde\mu^{\psi^\e}$ and group them to facilitate the analysis. The quadratic variation of $h^0$ is given by (omitting the argument $(t,w,s)$)
	\begin{equation}
	\label{eq:quad.var.h0}
	\big(c^{h^0}\big)_{l,m} = \sum_{i=1}^{n}\bigg(\sum_{j=1}^{d}\left(h^ {0,l}_wh^{0,j}+h^{0,l}_{s_{j}}\right)s_j\sigma^{j}_i\bigg)\bigg(\sum_{k=1}^{d}\left(h^ {0,m}_wh^{0,k}+h^{0,m}_{s_{k}}\right)s_k\sigma^{k}_i\bigg).
	\end{equation}
	The formula for the trace of the quadratic variation of $h^0$ multiplied by the Hessian of $\chi$ with respect to $\xi$ is given by
	\begin{align}
	\text{Tr}(c^{h^0}\chi_{\xi\xi}) 
	=&(h^0_w)^\top\chi_{\xi\xi}h^0_w\sum_{i,j=1}^{d}h^0_ih^0_js_is_j\sigma^i\cdot\sigma^j+2\sum_{i,j=1}^{d}(h^0_{s_i})^\top\chi_{\xi\xi}h^0_wh^0_js_is_j\sigma^i\cdot\sigma^j\notag\\
	&+\sum_{i,j=1}^{d}(h^0_{s_i})^\top\chi_{\xi\xi}h^0_{s_j}s_is_j\sigma^i\cdot\sigma^j\label{eq:trace.D2.h0}.
	\end{align}
	The definition of $E_1$, $E_2$ and $T$ are given in Equations \eqref{correc1}, \eqref{correc2} and \eqref{eq:defT}. We regroup
	\begin{align}
	&\tilde\mu^{\psi^\e}_t(w,s,h) = V^0_t+\tilde{U}(V^0_w)+\cL^s V^0\notag\\
	&+ V^0_w r w+ V^0_w\sum_{j=1}^{d} h_j^0 s_j (\mu^j-r)+\sum_{i,j=1}^{d}V^0_{w s_i} s_i h_j^0 s_j\sigma^j \cdot\sigma^i +\frac{ V^0_{ww}}{2} \sum_{i,j=1}^{d} h^0_ih^0_j s_is_j \sigma^i\sigma^j\notag\\
	&+V^0_w\sum_{j=1}^{d} (h_j-h_j^0) s_j (\mu^j-r)+\sum_{i,j=1}^{d}V^0_{w s_i}  (h_j-h_j^0)s_i s_j\sigma^i \cdot\sigma^j\notag\\
	&+\frac{ V^0_{ww}}{2} \sum_{i,j=1}^{d} \left(h_ih_j-h^0_ih^0_j\right) s_is_j \sigma^i\cdot\sigma^j+\ets\Big( (V^0_w)^{1-m}\Phi(s,\chi_\xi)-\frac{1}{2}\text{Tr}(c^{h^0}\chi_{\xi\xi}) j\Big)\notag\\
	&+\ets\Big(-\nu_t-\cL^s\nu-\nu_wrw-\widetilde{U}'(V^0_w)\nu_w-\nu_w\sum_{j=1}^{d} h_j^0 s_j (\mu^j-r)-\sum_{i,j=1}^{d}\nu_{w s_i} s_i h_j^0 s_j\sigma^i \cdot\sigma^j\notag\\
	&-\frac{\nu_{ww}}{2} \sum_{i,j=1}^{d} h^0_ih^0_j s_is_j \sigma^i\cdot\sigma^j\Big) -R^\e -U(c^\e)-(\pa_w \psi^\e)^{1-m}(\Phi(s,\chi_\xi)-\Phi(s,\chi_{\xi}+\ems \nu_h))\Big)\notag\\
	&=-\ets\Big(E_1(t,w,s,\ems(h-h^0),\chi_\xi,\chi_{\xi\xi})+E_2(t,w,s,\nu_t,\nu_w,\nu_{s},\nu_{ww},\nu_{ws},\nu_{ss})\notag\\
	&-(\pa_w \psi^\e)^{1-m}(\Phi(s,\chi_\xi)-\Phi(s,\chi_{\xi}+\ems \nu_h))\Big)-R^\e-U(c^\e).\label{eq:psi.e.expnsion}
	\end{align}
	The first line is obtained using that $c^0$ is the pointwise maximizer (in $w$ and $s$) of $c\mapsto U(c)-V^0_wc$, i.e. $U(c^0)-V^0_wc^0=\tilde{U}(V^0_w)$. Recall that $\cG^0(V^0)=0$ (cf. \eqref{eq:v0}), therefore the first two lines of the right-hand side of the above equation cancel. Thanks to the first order condition \eqref{eq:foc.h0}, the third line (and a part of the fourth) of the right-hand side can then be rewritten as $\frac{1}{2}V^0_{ww}\sum_{i,j=1}^{d}(h_ih_j-2h_ih^0_j+h^0_ih^0_j)s_is_j\sigma^i\cdot\sigma^j=\frac{1}{2}V^0_{ww}|\sum_{i=1}^{d}(h_i-h^0_i)s_i\sigma^i|^2$. This term, added to the fourth line gives $-E_1(t,w,s,\ems(h-h^0),\chi_\xi,\chi_{\xi\xi})$. The forth and fifth lines give $-E_2(t,w,s,\nu,\nu_t,\nu_w,\nu_{s},\nu_{ww},\nu_{ws},\nu_{ss})$. The terms remaining or forcefully introduced into the above equality are gathered in the function $R^\e$, that can be further decomposed into
	\begin{align*}
	R^\e=I^{\e,\tilde{U}}+\ets\Big(&T(t,w,s,\nu_w,\nu_{ww},\nu_{ws};h)-T(t,w,s,\nu_w,\nu_{ww},\nu_{ws};h^0)\\
	&+I^{\e,1}+I^{\e,2}+I^{\e,3}+I^{\e,4}\Big),
	\end{align*}
	where the $I^\e$'s are functions defined below in~\eqref{eq:def.IU}, \eqref{eq:def.I1}, \eqref{eq:def.I2}, \eqref{eq:def.I3} and \eqref{eq:def.I4}.

	We now bound each term constituting $R^\e$. Note that the sets of items (i), (ii) and (iii) are such that $(t,w,s)$ is in a compact.

	$T$ is quadratic in $h$ and does not depend on $\chi$, therefore the bounds in all three cases are obvious. 
	
	Similarly, although the derivatives of $\tilde U$ have negative powers in their argument, the assumptions \eqref{negpower}, the boundedness of $(t,w,s)$ on the considered set and the fact that $\tilde U$, $\nu$, $\chi$, $V^0$ and $\psi^\e$ are continuous allow us to treat these functions as Lipschitz-continuous with respect to the three first variables. Thus, defining $I^{\e,\tilde{U}}$ and using a Taylor expansion twice, there exist $\eta_1,\eta_2\in(0,1)$ depending on $(t,w,s,h)$ with
	\begin{align}
		I^{\e,\tilde{U}}(t,w,s,h,\nu,\chi) =& \tilde{U}(V^0_w)-\tilde{U}\big(\partial_w\psi^\e\big)-\ets \tilde{U}'(V^0_w)\nu_w\label{eq:def.IU}\\
		=&\ets\bigg[\tilde U'\bigg(V^0_w\bigg(1+\eta_1\frac{\pa_w\psi^\e-V^0_w}{V^0_w}\bigg)\bigg)-\tilde U'(V^0_w)\bigg]\nu_w\notag\\
		&+\efs\tilde U'\bigg(V^0_w\bigg(1+\eta_1\frac{\pa_w\psi^\e-V^0_w}{V^0_w}\bigg)\bigg)\pa_w\chi\notag\\
		=&\ets\nu_w\eta_2\big(\pa_w\psi^\e-V^0_w\big)\tilde U''\bigg(V^0_w\bigg(1+\eta_2\frac{\pa_w\psi^\e-V^0_w}{V^0_w}\bigg)\bigg)\notag\\
		&+\efs\tilde U'\bigg(V^0_w\bigg(1+\eta\frac{\pa_w\psi^\e-V^0_w}{V^0_w}\bigg)\bigg)\pa_w\chi.\notag
	\end{align}
	Then, given that $\chi\in\cC_m$  and for $0<\e<1$ the following bounds obtain on the set defined in~\eqref{negpower}.
	\begin{align*}
		\emts|I^{\e,\tilde{U}}(t,w,s,h,\nu,\chi)| &\leq C(\ets|\chi_w|+\es|\chi_\xi|)(t,w,s,\xi^\e)\\
		&\leq C(\es+\ets |\xi^\e|^{1+\frac{2}{m}}+\es|\xi^\e|^{\frac{2}{m}}).
	\end{align*}
	The expression for $I^{\e,1}$ is 
	\begin{align}
		I^{\e,1}(t,w,s,h,\chi) = \ets\Big(&\chi_t+\cL^s\chi+\chi_{w}rw+\chi_{w}\sum_{j=1}^{d}h_js_j(\mu^j-r)\label{eq:def.I1}\\
		&+\sum_{i,j=1}^{d}\chi_{ws_i}s_is_jh_j\sigma^i\cdot\sigma^j+\frac{1}{2}\chi_{ww}\sum_{i,j=1}^{d}h_ih_js_is_j\sigma^i\cdot\sigma^j\Big).\notag
	\end{align}
	If $(t,w,s,h)$ is in a bounded set the fact $\chi\in \cC_m$ implies 
	$$|I^{\e,1}(t,w,s,h,\chi)|\leq C\ets(1+|\xi^\e|^{1+\frac{2}{m}}),$$
	and if $\chi$ has bounded derivatives (as in item (i)), 
	 $$|I^{\e,1}(t,w,s,h,\chi)|\leq C\ets(1+|h-h^0(t,w,s)|^2).$$
	The expression for $I^{\e,2}$ is 
	\begin{align}
	I^{\e,2}(t,w,s,h,\nu,\chi) =((V^0_w)^{1-m}- (\partial_w\psi^\e)^{1-m})\Phi(s,\chi_{\xi})\ \label{eq:def.I2}.
	\end{align}
	Similarly as in \eqref{eq:def.IU}, \eqref{negpower} allows us to treat negative powers of $\pa_w \psi^\e$ as a Lipschitz function. If $(t,w,s,h)$ takes value on a bounded set, $((V^0_w)^{1-m}- (\partial_w\psi^\e)^{1-m})$ can be uniformly bounded for $h-h^0(t,w,s)$ small enough and we obtain that 
	$$|I^{\e,2}(t,w,s,h,\nu,\chi)|\leq C\ets\Phi(s,\chi_{\xi})\leq C\ets(1+ |\xi^\e|^2)\leq C(1+ |\xi^\e|^{1+\frac{2}{m}}).$$
	If the derivatives of $\chi$ are bounded the inequality becomes
	$$|I^{\e,2}(t,w,s,h,\nu,\chi)|\leq C|((V^0_w)^{1-m}- (\partial_w\psi^\e)^{1-m})| \leq C\ets.$$
	The definition of $I^{\e,3}$ and $I^{\e,4}$ are
	\begin{align}
	I^{\e,3}(t,w,s,h,\chi) =& -\es\Big((h^0_t)^\top\chi_\xi+(h^0_w)^\top\chi_\xi(rw+\sum_{i=1}^{d}h_is_i(\mu^i-r))+\sum_{i=1}^{d}(h^0_{s_i})^\top\chi_{\xi}s_i\mu^i\notag\\
	&\qquad+\big((h^0_{w})^\top\chi_{w\xi}+\frac{1}{2}(h^0_{ww})^\top\chi_{\xi}\big)\sum_{i,j=1}^{d}h_ih_js_is_j\sigma^i\cdot \sigma^j\label{eq:def.I3}\\
	&\qquad+\frac{1}{2}\sum_{i,j=1}^{d}((h^0_{s_j})^\top\chi_{\xi s_i}+(h^0_{s_i})^\top\chi_{\xi s_j}+(h^0_{s_is_j})^\top\chi_{\xi})s_is_j\sigma^i\cdot\sigma^j\notag\\
	&\qquad+\sum_{i,j=1}^{d}\big((h^0_{s_i})^\top\chi_{\xi w} + (h^0_{w})^\top\chi_{\xi s_i} + (h^0_{ws_i})^\top\chi_{\xi} \big)h_is_is_j\sigma^i\cdot\sigma^j\Big)\notag
	\end{align}
	and
	\begin{align}
	I^{\e,4}(t,w,s,h,\chi) =& \frac{1}{2}(h^0_w)^\top\chi_{\xi\xi}h^0_w\sum_{i,j=1}^{d}(h_ih_j-h^0_ih^0_j)s_is_j\sigma^i\cdot\sigma^j&\label{eq:def.I4}\\
	&+\sum_{i,j=1}^{d}(h^0_{s_i})^\top\chi_{\xi\xi}h^0_w(h_j-h^0_j)s_is_j\sigma^i\cdot\sigma^j.\notag
	\end{align}
	Again, if $(t,w,s,h)$ is on a bounded domain, $h^0\in\cC^{1,2,2}$ (cf. Assumption~\ref{assum:frictionless}) and $\chi\in \cC_m$ imply the bounds 
	\begin{align*}
		|I^{\e,3}(t,w,s,h,\chi)| \leq C\es (1+|\xi^\e|^{\frac{2}{m}})\leq C,\ \mbox{and}\ |I^{\e,4}(t,w,s,h,\chi)| \leq C (1+|\xi^\e|^{1+\frac{2}{m}}),
	\end{align*}
	and if the derivatives of $\chi$ are bounded 
	\begin{align*}
		|I^{\e,3}(t,w,s,h,\chi)| \leq& C\es (1+|h-h^0(t,w,s)|^2)\\ |I^{\e,4}(t,w,s,h,\chi)| \leq&C (1+|h-h^0(t,w,s)|^2).
	\end{align*}
	Then, results (i) and (ii) are a consequence of the estimates above.	
	Note that if additionally $\xi^\e$ is bounded as in (iii), all terms except $I^{\e,4}$ are bounded by $C \es$. For $I^{\e,4}$ we have 	$h-h^0=\es \xi^\e$. Thus, we also obtain the bound $|I^{\e,4}(t,w,s,h,\chi)|\leq C\es$ thanks to the boundedness of $\xi^\e$, and $0<\e<1$. This conclude the proof of the remainder estimates.
\end{proof}
\section{Proof of the Main Theorem}\label{s.semi}
In this section we prove Theorem~\ref{thm:expansion}.
\subsection{The Semilimits}

Let $V^\e_*$ (resp. $V^{\e,*}$) be the lower (resp. upper) semicontinuous envelope of the function $V^\e$. Assumption~\ref{assum:bound} allows us to define the following semilimits for $(t,w,s)\in\cD$
\begin{align}\label{eq:def-semilimits-upper2}
u^*(t,w,s):=\limsup_{\e\downarrow 0, (t',w',s',h')\to (t,w,s,h^0(t,w,s))}  \frac{V^0(t',w',s')-V^\e_*(t',w',s',h')}{\e^{2m^*}},\\
u_*(t,w,s):=\liminf_{\e\downarrow 0, (t',w',s',h')\to (t,w,s,h^0(t,w,s))} \frac{V^0(t',w',s')-V^{\e,*}(t',w',s',h')}{\e^{2m^*}},\label{eq:def-semilimits-lower2}
\end{align}
where $m^*=\frac{1}{3m-2}$.
By definition $u_*$ is upper semicontinuous, $u^*$ is lower semicontinuous and they satisfy
\begin{equation*}
0{\leq} u_*{\leq} u^*.
\end{equation*}
Define additionally for $\e>0$ 
\begin{align}
		u^{\e*}(t,w,s,h)=\frac{V^0(t,w,s)-V^{\e}_*(t,w,s,h)}{\ets}\geqslant 0\label{eq:ue*u},\\
		u^\e_*(t,w,s,h)=\frac{V^0(t,w,s)-V^{\e,*}(t,w,s,h)}{\ets}\geqslant 0.\label{eq:ue*l}
\end{align}

In the case of a market with sublinear price impact as here (i.e. subquadratic transaction costs), any finite position can be liquidated fast enough so that the loss of utility is negligible at the leading order ($O(\ets)$) and the penalty due to holding the ``wrong'' number of shares ($H^\e_t\neq H^0_t$, or $h\neq h^0(t,w,s)$) is of order strictly higher than $2m^*$, unlike in \cite{MMKS} where the authors had to introduce the adjusted semi-limits.

Finally recall the definition of the function $\xi^\e: (t,w,s,h)\in\cD\times \R^d\mapsto \frac{h-h^0(t,w,s)}{\es}$, which gives the renormalized displacement from the frictionless optimal strategy.  

\subsection{The Supersolution Property}\label{ss.supsol}
\begin{prop}\label{prop:supersol}
Under the assumption of Theorem~\ref{thm:expansion}, the function $u_*$ is a lower semicontinuous viscosity supersolution of the second corrector equation \eqref{2nd-corrector-eq}. 
\end{prop}

\begin{proof}The proof is based on \cite[Proof of Proposition 6.4]{MMKS}. Lower semicontinuity of the function holds by the definition of the function. We now show the viscosity property. 
Let $(t^0,w^0,s^0)\in \cD$ and $\phi\in \cC^{1,2,2}(\cD)$ such that $(t^0,w^0,s^0)$ is a strict minimizer of $u_*-\phi$ on $\cD$ and that $u_*(t ^0,w^0,s^0)-\phi(t^0,w^0,s^0) = 0$. Then, for all $(t,w,s)\in\cD\backslash\{(t^0,w^0,s^0)\}$ the following holds
\begin{align}\label{eq:supsolineq}
0=u_*(t ^0,w^0,s^0)-\phi(t^0,w^0,s^0)<u_*(t,w,s)-\phi(t,w,s).
\end{align}
We want to show that $\phi$ is a supersolution of the second corrector equation \eqref{2nd-corrector-eq} at the point $(t^0,w^0,s^0)$, in other words $-E_2(t^0,w^0,s^0,\phi_t,\phi_w,\phi_s,\phi_{ww},\phi_{ws},\phi_{ss})\geqslant a(t^0,w^0,s^0)$.

By the definition of $u_*$ (see \eqref{eq:def-semilimits-upper2}), there exists a family $(t^\e,w^\e,s^\e,h^\e)\in\cD\times\R^d$ such that 
\begin{align}
&(t^\e,w^\e,s^\e,h^\e)\to (t^0,w^0,s^0 ,h^0(t^0,w^0,s^0)),\quad u^\e_*(t^\e,w^\e,s^\e,h^\e)\to u_*(t^0,w^0,s^0)\notag\\
&\mbox{and }p^\e:=u^\e_*(t^\e,w^\e,s^\e,h^\e)-\phi(t^\e,w^\e,s^\e)\to 0\ \mbox{ as }\e\to 0.\label{eq:def.pe.supersol}
\end{align}
By Assumptions~\ref{assum:frictionless} (continuity of $h^0$ on $\cD$), and~\ref{assum:2ndcorrect} (continuity of $\varpi$ on $\cD\times\R^d$), there exist $\e_0,r_0>0$ such that $\bar B_{r_0}(t^0,w^0,s^0)\subset\cD$ and for all $\e\in(0,\e_0]$ we have 
\begin{align}
\label{eq:ineq.twshe.r0.super}&|(t^\e,w^\e,s^\e)-(t^0,w^0,s^0)|{\leq} \frac{r_0}{2},\qquad |p_\e|{\leq} 1,
\end{align}
Let $M=\sup\{\phi(t,w,s)~|~(t,w,s)\in \bar B_{r_0}(t^0,w^0,s^0)\}+4$ and note that \eqref{eq:ineq.twshe.r0.super} implies $|(t,w,s)-(t^\e,w^\e,s^\e)|^{4}\geqslant (r_0/2)^4$ on $\pa B_{r_0}(t^0,w^0,s^0)$, for $0<\e\leqslant \e_0$.
We can now choose $c_0>0$ such that $c_0 (r_0/2)^4\geqslant M$ and define $\varphi^\e$ and $\varphi^0$ on $\R^{d+2}$ by
\begin{align}\label{eq:def.varphi.e}
\varphi^\e(t,w,s)&=\phi(t,w,s)+p^\e-c_0|(t,w,s)-(t^\e,w^\e,s^\e)|^{4},\\
\varphi^0(t,w,s)&=\phi(t,w,s)-c_0|(t,w,s)-(t^0,w^0,s^0)|^{4}.\notag
\end{align}
Given the choice of these constants, \eqref{eq:ineq.twshe.r0.super} gives
\begin{equation}
\label{eq:boundary.B0.super}
\varphi^\e(t,w,s){\leq} -3~\mbox{ on }~\pa B_{r_0}(t^0,w^0,s^0)~\mbox{ for all }~\e\in(0,\e_0],
\end{equation}
and by definition of $p^\e$ (see \eqref{eq:def.pe.supersol}), we have 
\begin{equation}
\label{eq:ue*l.varphi}
- u^\e_*(t^\e,w^\e,s^\e,h^\e)+\varphi^\e(t^\e,w^\e,s^\e)=0.
\end{equation}
\emph{Claim: There exist a neighborhood around $0$ in $\R^d$ and $C>0$ constant such that on this neighborhood $\sup_{(t,w,s)\in \bar B_{r_0}(t^0,w^0,s^0)}\varpi(t,w,s,\xi)\leqslant C|\xi|^2$ and $\sup_{(t,w,s)\in \bar B_{r_0}(t^0,w^0,s^0)}\varpi(t,w,s,\xi)_\xi\leqslant C|\xi|$.}

This follows from the first corrector equation and the fact that $\varpi(t,w,s,0)=\varpi_{\xi}(t,w,s,0)=0$ for any $(t,w,s)$. $\varpi$ is a solution of the first corrector equation \eqref{1st-corrector-eq}. Then it holds for all $(t,w,s)\in \bar B_{r_0}(t^0,w^0,s^0)$ that $\frac{1}{2}\text{Tr}(c^{h^0}(t,w,s)\varpi_{\xi\xi}(t,w,s,0))=a(t,w,s)$. The inequality $|X|{\leq} \text{Tr}(X)$ for symmetric non-negative matrices, the fact that $\varpi$ is convex, $c^{h^0}$ is positive definite and continuous (cf. Assumption~\ref{assum:frictionless}), and the continuity and positivity of $a$ yield $|\varpi_{\xi\xi}(t,w,s,\xi)|{\leq} C$ on a neighbourhood $\cN$ of $0$ in $\R^d$, for some $C>0$ constant. Then, on $\bar B_{r_0}(t^0,w^0,s^0)\times\cN$ we have $|\varpi_\xi(t,w,s,\xi)|{\leqslant} C|\xi|$ and $|\varpi(t,w,s,\xi)|{\leqslant} C|\xi|^2$, and the claim is proved.

For fixed $(t^0,w^0,s^0)$ and $r_0$, the continuity of $E_2$, $c^{h^0}$, $a$, $\Phi_x$, and $\sigma$, the regularity of $\phi$ and $\varpi$, the $(m-1)$-homogeneity of $\Phi_x$, the fact that $\varpi\in\cC_m$ by Assumption~\ref{assum:2ndcorrect} and the last claim allow us to define the following positive finite constants:
\begin{align}
K_0:=&1+\sup \left\{|E_2(t,w,s,\varphi^0_t,\varphi^0_w,\varphi^0_s,\varphi^0_{ww},\varphi^0_{ws},\varphi^0_{ss})|:(t,w,s)\in \bar B_{r_0}(t^0,w^0,s^0)\right\},\label{eq:def.K0}\\
K_\Sigma:=&1+\sup \left\{\left|c^{h^0}(t,w,s) \right|:(t,w,s)\in \bar B_{r_0}(t^0,w^0,s^0)\right\}\label{eq:def.Ksigma},
\end{align}
\begin{align}
K_\varpi:=&1+\sup \left\{\Big(\frac{|\varpi_{\xi}(\cdot,\xi)|}{|\xi|^{\frac{2}{m}}}+\frac{|\varpi_{\xi\xi}(\cdot,\xi)|+|\varpi (\cdot,\xi)|}{|\xi|^{1+\frac{2}{m}}}\Big)(t,w,s):(t,w,s)\in \bar B_{r_0}(t^0,w^0,s^0),   \ |\xi|{\geq} 1\right\}\notag\\
&+\sup \Big\{\frac{|\varpi(t,w,s,\xi)|}{|\xi|^{1+\frac{2}{m}}}+|\varpi_{\xi\xi}(t,w,s,\xi)|:(t,w,s)\in \bar B_{r_0}(t^0,w^0,s^0),|\xi|{\leq} 1\Big\},\label{eq:def.Kvarpi}
\end{align}
\begin{align}
&K_\varpi':=1+\sup \left\{\left(\frac{|\varpi_{\xi}(\cdot,\xi)|}{|\xi|}+\frac{|\varpi_{\xi\xi}(\cdot,\xi)|+|\varpi (\cdot,\xi)|}{|\xi|^{2}}\right)(t,w,s):(t,w,s)\in \bar B_{r_0}(t^0,w^0,s^0),  \ |\xi|{\geq} 1 \right\} \notag\\
&+\sup \left\{\left(\frac{|\varpi(\cdot,\xi)|}{|\xi|^2}+\frac{|\varpi_\xi(\cdot,\xi)|}{|\xi|}+|\varpi_{\xi\xi}(\cdot,\xi)|\right)(t,w,s):(t,w,s)\in \bar B_{r_0}(t^0,w^0,s^0)~|\xi|{\leq} 1\right\},\label{eq:def.Kvarpi'}
\end{align}
\begin{align}
K_a:=&1+\sup \left\{\left|a(t,w,s)\right|:(t,w,s)\in \bar B_{r_0}(t^0,w^0,s^0)\right\},\\
K_{\Phi_x}:=&1+\sup\left\{\frac{|\Phi_x(s,x)|}{1\vee |x|^{m-1}}:|s-s^0|\leq r_0, x\in\R^d \right\}\label{eq:def.Phia},\\
\gamma_v:=&\inf\bigg\{-\frac{\pa_{ww} V^0(t,w,s)}{2} \bigg|\sum_{j=1}^{d} \xi_j s_j \sigma^j(s)\bigg|^2:(t,w,s)\in \bar B_{r_0}(t^0,w^0,s^0), |\xi|=1\bigg\}.\label{eq:def.gamma.v}
\end{align}
Additionally, it holds by Assumption~\ref{assum:frictionless} that $V^0_w>0$ on $\cD$ so there exists $\iota>0$ such that
\begin{equation} 
	\frac{1}{\iota} {\leq} V^0_w(t,w,s){\leq} \iota\mbox{ for all }(t,w,s)\in \bar B_{r_0}(t^0,w^0,s^0).\label{eq:estimate.iota}
\end{equation}

Similarly to \cite[Lemma A.2]{PST} and \cite[Proof of Proposition 6.4]{MMKS}, there exists $C^*>0$ such that for all $\eta>0$ we can find a function $h^\eta\in \cC^\infty(\R^d,[0,1])$ and $a_\eta>1$ satisfying 
\begin{align}
h^\eta=1\mbox{ on }\bar B_1(0),\quad h^\eta=0\mbox{ on }\bar B^c_{a_\eta}(0),\quad |h^\eta_x(x)|\wedge|xh^\eta_x(x)|{\leq} \eta \quad \mbox{ and }\quad |x|^2|h^\eta_{xx}|{\leq} C^*.\label{eq:property.h.eta}
\end{align}
Fix $\d$ and $\eta$ in $(0,1)$. 
Note that due to the inequality ${1+\frac{2}{m}}<2$ there exists $\xi^{*,\d}>0$, the unique positive solution of
\begin{equation}\label{eq:defdeltastar}
(\xi^{*,\d})^2= \frac{2(\xi^{*,\d})^{1+\frac{2}{m}}dK_\Sigma K_\varpi+2\left(K_a+  K_0\right)+dK_{\Sigma}K'_{\varpi}(C^*+2)}{(1-(1-\d)^m)\gamma_v}.
\end{equation}
Define also the functions
\begin{align} 
H^{\eta,\delta}&:\xi\in \R^d\mapsto (1-\d)h^\eta\left(\frac{\xi}{\xi^{*,\d}}\right)\label{eq:def.H.eta.delta},\\
\psi^{\e,\eta,\delta}(t,w,s,h)&:=V^0(t,w,s)-\e^{2m^*}\varphi^\e(t,w,s)-\e^{4m^*}(\varpi H^{\eta,\d})(t,w,s,\xi^\e),\notag\\
I^{\e,\eta,\d}(t,w,s,h)&:=\frac{V^{\e,*}(t,w,s,h)-\psi^{\e,\eta,\d}(t,w,s,h)}{\e^{2m^*}},\notag
\end{align}
where we make a slight abuse of notation for brevity in writing $(\varpi H^{\eta,\d})(t,w,s,\xi^\e)$ for $$H^{\eta,\d}(\xi^\e(t,w,s,h))\varpi(t,w,s,\xi^\e(t,w,s,h))~\mbox{ for }~(t,w,s,h)\in\cD\times\R^d.$$ We want here to use $\psi^{\e,\eta,\delta}$ as a test function for the viscosity subsolution property of $V^{\e,*}$ (see Assumption~\ref{assum:Ve.viscosity} and Equation~\eqref{eq:PDE-veps}). For this, we need interior maximizers of the functions $V^{\e,*}-\psi^{\e,\eta,\delta}$ (or equivalently of $I^{\e,\eta,\delta}$) in $\bar B_{r_0}(t^0,w^0,s^0)\times \R^d$. However the supremum of $I^{\e,\eta,\delta}$ may not be attained or lie on $\pa\bar B_{r_0}(t^0,w^0,s^0)\times \R^d$ and we therefore need to modify $\psi^{\e,\eta,\delta}$. 

First, note that for the elements of the family $\left(t^\e,w^\e,s^\e,h^\e\right)_{0<\e{\leq} \e_0}$ we have by \eqref{eq:ue*l.varphi} and non-negativity of $\varpi$ (see Assumption~\ref{assum:2ndcorrect})
\begin{equation}
\label{eq:lower.bound.I}
I^{\e,\eta,\d}(t^\e,w^\e,s^\e,h^\e)\geqslant 0.
\end{equation}

Defining $\e_{\eta,\d}:=\e_0\wedge 1\wedge (K_\varpi (a_\eta \xi^{*,\delta})^{1+\frac{2}{m}})^{-1/(2m^*)}$, similarly to \cite[Proof of Proposition 6.4]{MMKS}, we obtain the following inequality on $\bar B_{r_0}(t^0,w^0,s^0)\times\R^d$ for all $0<\e{\leq} \e_{\eta,\d}$
\begin{align}
I^{\e,\eta,\d}(t,w,s,h)&{\leqslant} \varphi^\e(t,w,s)\notag\\
&+\ets |\xi^\e(t,w,s,h)|^{1+\frac{2}{m}} \frac{\varpi\left(t,w,s,\xi^\e(t,w,s,h)\right)}{ |\xi^\e(t,w,s,h)|^{1+\frac{2}{m}}}\one_{\{|\xi^\e(t,w,s,h)|{\leqslant} a_\eta \xi^{*,\delta} \}}\notag\\
&{\leqslant}\varphi^\e(t,w,s)+\ets (a_\eta \xi^{*,\delta})^{1+\frac{2}{m}}K_\varpi \leqslant \varphi^\e(t,w,s)+1,\label{eq:ineq.I.varphie}
\end{align}
where the first inequality holds by \eqref{eq:ue*l} and the estimate $0{\leq} H^{\eta,\delta}(\xi){\leq} \one_{\{|\xi|{\leq} a_\eta \xi^{*,\delta} \}}$ and the second by the definition \eqref{eq:def.Kvarpi} of $K_\varpi$ and the choice of $\e_{\eta,\delta}$.

Since the right-hand side of \eqref{eq:ineq.I.varphie} is uniformly bounded on $\bar B_{r_0}(t^0,w^0,s^0)$ for $0<\e{\leq} \e_{\eta,\d}$, we can pick $(\hat t^{\e,\eta,\d},\hat w^{\e,\eta,\d},\hat s^{\e,\eta,\d},\hat h^{\e,\eta,\d})\in B_{r_0}(t^0,w^0,s^0)\times \R^d$ such that
\begin{equation}
\label{eq:presup.I}I^{\e,\eta,\d}(\hat t^{\e,\eta,\d},\hat w^{\e,\eta,\d},\hat s^{\e,\eta,\d},\hat h^{\e,\eta,\d}) \geqslant \sup_{B_{r_0}(t^0,w^0 ,s^0)\times \R^d}\left\{I^{\e,\eta,\d}(t,w,s,h)\right\}-\frac{\e^{2m^*}}{2}.
\end{equation}

We now add a penalization to $\psi^{\e,\eta,\delta}$ in the direction of $h$. Let $f\in \cC^{2}(\R_+, [0,1])$ be a function satisfying for some $c>0$
\begin{equation}\label{eq:condition.f}
f(0)=1,~f(x)=0~ \mbox{ for }~|x|\geqslant 1,~ 0{\leq} f{\leq} 1~\mbox{ and }~|f'(x)|{\leqslant} c|x|~\mbox{ in a neighborhood of }0.
\end{equation}
Then, define the functions for $\eta\in(0,1]$, $\d\in(0,1)$ and $\e\in(0,\e_{\eta,\d})$
\begin{align*}
\bar \psi^{\e,\eta,\d}(t,w,s,h) =& \psi^{\e,\eta,\delta}(t,w,s,h)-\e^{4m^*}f(|h-\hat h^{\e,\eta,\d}|),\\
\bar I^{\e,\eta,\delta}(t,w,s,h) =& \frac{V^{\e,*}(t,w,s,h)-\bar\psi^{\e,\eta,\d}(t,w,s,h)}{\e^{2m^*}}\\
 =& I^{\e,\eta,\delta}(t,w,s,h) +\ets f(|h-\hat h^{\e,\eta,\d}|),
\end{align*}
and the compact set 
\begin{equation*}
Q^{\e,\eta,\d}:=\{(t,w,s,h): (t,w,s)\in \bar B_{r_0}(t^0,w^0,s^0),~|h-\hat h^{\e,\eta,\d}|{\leq} 1\}.
\end{equation*}
\emph{Claim: there exists $(\tilde t^{\e,\eta,\d}, \tilde w^{\e,\eta,\d},\tilde s^{\e,\eta,\d},\tilde h^{\e,\eta,\d})\in Int(Q^{\e,\eta,\d})$ a maximiser of $V^{*,\e}-\bar \psi^{\e,\eta,\d}$ on\linebreak $B_{r_0}(t^0,w^0,s^0)\times\R^d$}.

The proof of this claim is similar to \cite[Proof of Proposition 6.4]{MMKS}(Step 3). Since $f(0)=1$ the definition of $\bar I^{\e,\eta,\d}$ leads to 
$$\bar I^{\e,\eta,\d}(\hat t^{\e,\eta,\d},\hat w^{\e,\eta,\d},\hat s^{\e,\eta,\d},\hat h^{\e,\eta,\d}) =I^{\e,\eta,\d}(\hat t^{\e,\eta,\d},\hat w^{\e,\eta,\d},\hat s^{\e,\eta,\d},\hat h^{\e,\eta,\d}) +\ets.$$
Furthermore, on $(\bar B_{r_0}(t^0,w^0,s^0)\times\R^d)\backslash Q^{\e,\eta,\d}$, it holds $\bar I^{\e,\eta,\d}(t,w,s,h) =I^{\e,\eta,\d}(t,w,s,h)$. This, with \eqref{eq:presup.I}, gives
\begin{equation}
	\sup_{\bar B_{r_0}(t^0,w^0 ,s^0)\times \R^d}\left\{\bar I^{\e,\eta,\d}(t,w,s,h)\right\}=\sup_{Q^{\e,\eta,\d}}\left\{\bar I^{\e,\eta,\d}(t,w,s,h)\right\}.
\end{equation}
The function $\bar I^{\e,\eta,\d}$ is upper-semicontinuous and $Q^{\e,\eta,\d}$ is compact, so there exists a maximizer $(\tilde t^{\e,\eta,\d},\tilde w^{\e,\eta,\d},\tilde s^{\e,\eta,\d},\tilde h^{\e,\eta,\d})$ of $V^{*,\e}-\bar \psi^{\e,\eta,\d}$ on $Q^{\e,\eta,\d}$. It is also the maximizer on $\bar B_{r_0}(t^0,w^0 ,s^0)\times \R^d$. Now, let $(t,w,s,h)\in\pa\bar B_{r_0}(t^0,w^0,s^0)\times\R^d$. Then, by the bounds $0{\leq} f{\leq} 1$, and $0<\e_{\eta,\delta}\leqslant 1$ and the two inequalities \eqref{eq:boundary.B0.super} and \eqref{eq:ineq.I.varphie} we have
\begin{equation*}
	\bar I^{\e,\eta,\d}(t,w,s,h){\leqslant} I^{\e,\eta,\d}(t,w,s,h)+\ets{\leqslant} -2+\ets{\leqslant} -1.
\end{equation*}
On the other hand, in the interior of $\bar B_{r_0}(t^0,w^0,s^0)\times\R^d$, for the family $\left(t^\e,w^\e,s^\e,h^\e\right)_{0<\e{\leq} \e_{\eta,\d}}$, it holds by \eqref{eq:lower.bound.I} and definition of $\bar I^{\e,\eta,\d}$
\begin{equation*}
	\bar I^{\e,\eta,\d}(t^\e,w^\e,s^\e,h^\e)\geqslant I^{\e,\eta,\d}(t^\e,w^\e,s^\e,h^\e)\geqslant 0
\end{equation*}
and the maximizer is therefore a point of $B_{r_0}(t^0,w^0,s^0)\times\R^d$ and the claim is proved.

Thus, for $\e\in(0,\e_{\eta,\d}]$,  we have a $\cC^{1,2,2,2}(\cD\times\R^d, \R)$ function $\bar \psi^{\e,\eta,\d}$ and  a local strict maximizer of $V^{\e,*}-\bar \psi^{\e,\eta,\d}$ denoted $(\tilde t^{\e,\eta,\d},\tilde w^{\e,\eta,\d},\tilde s^{\e,\eta,\d},\tilde h^{\e,\eta,\d})$. Since $V^{\e,*}$ is a subsolution of \eqref{eq:PDE-veps}, it holds
\begin{equation}\label{eq:visco.subsol.Ge}
	\cG^\e(\bar \psi^{\e,\eta,\d})(\tilde t^{\e,\eta,\d},\tilde w^{\e,\eta,\d},\tilde s^{\e,\eta,\d},\tilde h^{\e,\eta,\d}){\leq} 0.
\end{equation}
Denoting 
\begin{align*}
&\tilde \varpi^{\e,\eta,\d}(t,w,s,\xi)=(\varpi H^{\eta,\d})(t,w,s,\xi)+f(|\e^{m^*}\xi+h^0(t,w,s)-\hat h^{\e,\eta,\d}|)\\
&\mbox{and }~\tilde \xi^{\e,\eta,\d} :=\xi^\e(\tilde t^{\e,\eta,\d},\tilde w^{\e,\eta,\d},\tilde s^{\e,\eta,\d},\tilde h^{\e,\eta,\d})
\end{align*}
the remainder estimate \eqref{eq:remainder.viscosity} of Proposition~\ref{prop:remainder.estimate} for $\bar \psi^{\e,\eta,\d}$ (in that case the function $\nu$ of Proposition~\ref{prop:remainder.estimate} is $\varphi^\e$ and does not depend on $h$) with \eqref{eq:visco.subsol.Ge} yields
\begin{align}
&E_2(\tilde t^{\e,\eta,\d},\tilde w^{\e,\eta,\d},\tilde s^{\e,\eta,\d},\varphi^\e_t , \varphi^\e_s, \varphi^\e_w,\varphi^\e_{sw},\varphi^\e_{ww},\varphi^\e_{ss})\label{eq:E1.E2.Re}\\
&+E_1(\tilde t^{\e,\eta,\d},\tilde w^{\e,\eta,\d},\tilde s^{\e,\eta,\d},\tilde \xi^{\e,\eta,\d},\tilde \varpi^{\e,\eta,\d}_\xi,\tilde \varpi^{\e,\eta,\d}_{\xi\xi})+\frac{R^\e(\tilde t^{\e,\eta,\d},\tilde w^{\e,\eta,\d},\tilde s^{\e,\eta,\d},\tilde h^{\e,\eta,\d},\bar \psi^{\e,\eta,\d})}{\e^{2m^*}}{\leq} 0.\notag
\end{align}
The family $\left\{\left(t^{\e,\eta,\d},\tilde w^{\e,\eta,\d},\tilde s^{\e,\eta,\d}\right)~|~\e\in(0,\e_{\eta,\d}],~\eta\in(0,1),~\d\in(0,1)\right\}$ is bounded. For fixed $\eta\in(0,1),~\d\in(0,1)$,  due to the choice of $f$ and $H^{\eta,\d}$, the assumptions of claim (i) in Proposition~\ref{prop:remainder.estimate} holds (up to reducing $\e_0>0$ for $(\tilde t^{\e,\eta,\d},\tilde w^{\e,\eta,\d},\tilde s^{\e,\eta,\d},\tilde h^{\e,\eta,\d})$ to be in the set defined in~\eqref{negpower}) and we obtain the existence of $C>0$ that may depend on $\eta,\delta\in (0,1)$ but not on $\e>0$ such that 
\begin{equation}
\frac{|R^\e(\tilde t^{\e,\eta,\d},\tilde w^{\e,\eta,\d},\tilde s^{\e,\eta,\d},\tilde h^{\e,\eta,\d},\bar \psi^{\e,\eta,\d})|}{\e^{2m^*}}{\leq} C(1+|\e^{m^*}\tilde \xi^{\e,\eta,\d}|^2).\label{eq:bound.Re.xitilde}
\end{equation}
\emph{Claim: Fix $\eta\in(0,1)$ and $\d\in(0,1)$, the family $\{\tilde \xi^{\e,\eta,\d} :\e\in (0, \e^{\eta,\d})\}$ is bounded by $C_{\tilde \xi}$.} To prove this, we assume that there exists a sequence $(\e_n)_{n\in\N}$ such that $\lim_{n\to \infty}\tilde \xi^{\e_n,\eta,\d}=\infty$ with $\e_n\in (0, \e^{\eta,\d}]$ for all $n\in\N$ and $\e_n\to 0$ as $n\to \infty$ {\color{black}(indeed, the function $\xi^\e$ is continuous and the family $\left\{\left(t^{\e,\eta,\d},\tilde w^{\e,\eta,\d},\tilde s^{\e,\eta,\d}\right)~|~\e\in(0,\e_{\eta,\d}],~\eta\in(0,1),~\d\in(0,1)\right\}$ is bounded and $\e_n$ bounded away from $0$ would imply that  the family $\{\tilde \xi^{\e,\eta,\d} :\e\in (0, \e^{\eta,\d})\}$ is bounded)}. Without loss of generality we can assume that the sequence $(\e_n)$ is decreasing. By definition~\eqref{eq:def.H.eta.delta} of $H^{\eta,\d}$  it holds that $\varpi H^{\eta,\d}$ and its derivatives vanish at $(t,w,s,\tilde \xi^{\e_n,\eta,\d})$ for $(t,w,s)\in B_{r_0}(t^0,w^0,s^0)$ and $n$ large enough (say $n\geqslant n_1$). Then we have for $n\geqslant n_1$,
\begin{equation*}
	\tilde \varpi^{\e_n,\eta,\d}_\xi=\pa_\xi f(|\e^{m^*}\xi+h^0(t,w,s)-\hat h^{\e,\eta,\d}|),\qquad\tilde \varpi^{\e_n,\eta,\d}_{\xi\xi}=\pa_{\xi\xi} f(|\e^{m^*}\xi+h^0(t,w,s)-\hat h^{\e,\eta,\d}|).
\end{equation*}
Furthermore, it holds for $\eta\in(0,1)$, $\d\in(0,1)$ and $\e\in(0,\e^{\eta,\d})$,
\begin{align*}
\pa_\xi f(|\e^{m^*}\xi+h^0(t,w,s)-\hat h^{\e,\eta,\d}|)=&\e^{m^*} f'(|\e^{m^*}\xi+h^0(t,w,s)-\hat h^{\e_n,\eta,\d}|)^\top D^{\e_n,\eta,\d}_1,\\
\pa_{\xi\xi }f(|\e^{m^*}\xi+h^0(t,w,s)-\hat h^{\e,\eta,\d}|)=&\e^{2m^*} f''(|\e^{m^*}\xi+h^0(t,w,s)-\hat h^{\e_n,\eta,\d}|)D^{\e_n,\eta,\d}_1(D^{\e_n,\eta,\d}_1)^\top\\
&+\e^{2m^*} f'(|\e^{m^*}\xi+h^0(t,w,s)-\hat h^{\e_n,\eta,\d}|)D^{\e_n,\eta,\d}_2,
\end{align*}
where $\es D^{\e_n,\eta,\d}_1$ and $\ets D^{\e,\eta,\d}_2$ are the gradient and the Hessian of the function $\xi\in\R^d\mapsto |\es\xi+h^0(t,w,s)-\hat h^{\e_n,\eta,\d}|$ (at the point $(t,w,s,\xi)$ which is omitted in the equation above). With the assumption made on $f'$ (see \eqref{eq:condition.f}), we have for $(t,w,s,\xi)\in\bar B_{r_0}(t^0, w^0, s^0)\times\R^d$
\begin{align}
	\left|\pa_\xi f(|\e^{m^*}\xi+h^0(t,w,s)-\hat h^{\e,\eta,\d}|)\right|{\leq}\e^{m^*} C_f\mbox{ and}\notag\\
	\left|\pa_{\xi\xi} f(|\e^{m^*}\xi+h^0(t,w,s)-\hat h^{\e,\eta,\d}|)\right|{\leq}\e^{2m^*}C_f,\label{eq:tilde.varpi.xi.xixi.supersol}
\end{align}
where $C_f$ is a positive constant that can be chosen independent of $\eta$, $\delta$ and $\e$. Then, by $m$-homogeneity of $\Phi$ (remember that $m>2$) and linearity of the trace, there exists $C_f'>0$ such that 
\begin{equation*}
\left|-\left|V^0_w(t,w,s)\right|^{1-m}\Phi\left(s,\tilde \varpi^{\e_n,\eta,\d}_\xi(t,w,s,\xi)\right)+\frac{1}{2}\text{Tr}\left(c^{h^0}(t,w,s)\tilde \varpi^{\e_n,\eta,\d}_{\xi\xi}(t,w,s,\xi\right)\right|{\leq} \ets C_f'
\end{equation*}
on $\bar B_{r_0}(t^0, w^0, s^0)\times\R^d$. This finally provides the estimate for $E_1$ (see \eqref{correc1})
\begin{align}\label{eq:bound.E1.xitilde}
	&E_1\left(\tilde t^{\e_n,\eta,\d},\tilde w^{\e_n,\eta,\d},\tilde s^{\e_n,\eta,\d}, \tilde \xi^{\e_n,\eta,\d},\tilde \varpi^{\e_n,\eta,\d}_\xi,\tilde \varpi^{\e_n,\eta,\d}_{\xi\xi}\right)\\
	&\geqslant -C_f' \e^{2m^*}-\frac{ V^0_{ww}(\tilde t^{\e_n,\eta,\d},\tilde w^{\e_n,\eta,\d},\tilde s^{\e_n,\eta,\d})}{2} \bigg|\sum_{j=1}^{d} \tilde \xi^{\e_n,\eta,\d}_j \tilde s^{\e_n,\eta,\d}_j \sigma^j(\tilde s^{\e_n,\eta,\d})\bigg|^2\notag\\
	&\geqslant -C'_f \e^{2m^*}+\gamma_v \big| \tilde \xi^{\e_n,\eta,\d} \big|^2.\notag
\end{align}

Note that $E_2(t,w,s,\varphi^\e_t, \varphi^\e_w, \varphi^\e_s, \varphi^\e_{ww}, \varphi^\e_{ws},\varphi^\e_{ss})$ does not depend on $\xi$ (or $h$) and is bounded on $\bar B_{r_0}(t^0, w^0, s^0)$ for $\e\in(0,1)$. Hence, putting together \eqref{eq:E1.E2.Re}, \eqref{eq:bound.Re.xitilde} and \eqref{eq:bound.E1.xitilde}, we obtain for some positive constant $C$
\begin{equation*}
\gamma_v \big| \tilde \xi^{\e,\eta,\d} \big|^2 {\leqslant} C(1+|\e^{m^*}\tilde \xi^{\e,\eta,\d}|^2),
\end{equation*}
which contradicts the convergence of $\tilde\xi^{\e_n,\eta,\d}$ to infinity while $\e_n$ converges to $0$ when $n$ goes to $\infty$. 
Thus, the claim is proved, and there exists a subsequence in $\e>0$ such that 
\begin{align*}
\tilde t^{\e,\eta,\d}\to \tilde t^{\eta,\d},~\tilde w^{\e,\eta,\d}\to \tilde w^{\eta,\d},~ \tilde s^{\e,\eta,\d}\to\tilde s^{\eta,\d},~ \tilde h^{\e,\eta,\d}\to\tilde h^{\eta,\d}=h^0(\tilde t^{\eta,\d},\tilde w^{\eta,\d},\tilde s^{\eta,\d})\\
~\mbox{ and }~ \xi^\e(\tilde t^{\e,\eta,\d},\tilde w^{\e,\eta,\d},\tilde s^{\e,\eta,\d},\tilde h^{\e,\eta,\d})&\to \tilde\xi^{\eta,\d}.
\end{align*}
Using this convergence, the continuity of the functions $\phi$, $(\varphi^\e)_{\e>0}$, $\varpi$, $H^{\eta,\d}$, $V^0_{ww}$, $V^0_w$, $\Phi$,$c^{h^0}$, $R^\e$, $E_1$ and $E_2$ on their domain, the fact that $(\varphi^\e)_{\e>0}$ converges uniformly on $\bar B_{r_0}(t^0,w^0,s^0)$ to $\varphi^0$ as $\e\to 0$, the claim (iii) of  Proposition~\ref{prop:remainder.estimate} (note that again $(\tilde t^{\eta,\d},\tilde w^{\eta,\d},\tilde s^{\eta,\d},\tilde h^{\eta,\d})$ is in the set defined in~\eqref{negpower}, as we have taken the limit $\e\to 0$ of a sequence of elements of the set in \eqref{negpower}), and taking the limit of \eqref{eq:E1.E2.Re} as $\e\to 0$ we obtain the inequality
\begin{align}\label{eq:supeta}
E_1(\tilde t^{\eta,\d}, \tilde w^{\eta,\d},\tilde s^{\eta,\d},\tilde \xi^{\eta,\d},&\pa_\xi (\varpi H^{\eta,\d}),\pa_{\xi\xi}(\varpi H^{\eta,\d}))\\
&{\leqslant} -E_2(\tilde t^{\eta,\d}, \tilde w^{\eta,\d},\tilde s^{\eta,\d},\varphi^0, \varphi^0_s, \varphi^0_w,\varphi^0_{ww},\varphi^0_{ws},\varphi^0_{ss}).\notag
\end{align}
We used as well estimate \eqref{eq:tilde.varpi.xi.xixi.supersol}, to conclude that on $B_{r_0}(t^0,w^0,s^0)\times B_{C_{\tilde \xi}}(0)$, $\tilde \varpi^{\e_n,\eta,\d}_{\xi}$ and $\tilde \varpi^{\e_n,\eta,\d}_{\xi\xi}$ converge uniformly to $\pa_\xi (\varpi H^{\eta,\d})$ and $\pa_{\xi\xi}(\varpi H^{\eta,\d})$ respectively as $\e\to0$ (where $C_{\tilde \xi}$ bounds the family $\{\tilde \xi^{\e,\eta,\d} :\e\in (0, \e^{\eta,\d})\}$ by the previous claim).

Note that $(\tilde t^{\eta,\d},\tilde w^{\eta,\d},\tilde s^{\eta,\d})\in \bar B_{r_0}(t^0,w^0,s^0)$. Direct computation of $\pa_{\xi\xi}(H^{\eta,\d}\varpi)$, the properties \eqref{eq:property.h.eta} of $h^\eta$'s derivatives, the definitions of $K_0,K_{\varpi},K'_{\varpi}$, and $K_\Sigma$ in~\eqref{eq:def.K0}-\eqref{eq:def.Kvarpi'}, the elementary equation $\text{Tr}(A)\leqslant d|A|$, the inequality $0<\delta<1$ and \eqref{eq:supeta} yield the following inequality at the point $P^{\eta,\delta}:=(\tilde t^{\eta,\d}, \tilde w^{\eta,\d},\tilde s^{\eta,\d},\tilde \xi^{\eta,\d})$
\begin{align}
&-\frac{V^0_{ww}(\tilde t^{\eta,\d}, \tilde w^{\eta,\d},\tilde s^{\eta,\d})}{2} \bigg|\sum_{j=1}^{d} \tilde\xi^{\eta,\d}_j \tilde s^{\eta,\d}_j \sigma^j(\tilde s^{\eta,\d})\bigg|^2\notag\\
&-|V^0_w(\tilde t^{\eta,\d}, \tilde w^{\eta,\d},\tilde s^{\eta,\d})|^{1-m}\Phi\left(\tilde s^{\eta,\d},{\pa_\xi (\varpi H^{\eta,\d})}\right)(P^{\eta,\d})\notag\\ 
&{\leqslant} -E_2(\tilde t^{\eta,\d}, \tilde w^{\eta,\d},\tilde s^{\eta,\d},\varphi^0_t, \varphi^0_w, \varphi^0_s,\varphi^0_{ww},\varphi^0_{ws},\varphi^0_{ss}) -\frac{1}{2}\text{Tr}\big(c^{h^0} \pa_{\xi\xi}(\varpi H^{\eta,\d})\big)\big(P^{\eta,\d}\big)\notag\\
& {\leqslant} K_0+\frac{d}{2}K_{\Sigma}\big(K'_{\varpi}C^*+2\eta K'_{\varpi} +K_{\varpi}\big(\big|\tilde\xi^{\eta,\d}\big|^{1+\frac{2}{m}}\vee 1\big)\big).\label{eq:ineq.pre.visco.supersol1}
\end{align}

We also write this term as 
\begin{align}
&-\frac{V^0_{ww}(\tilde t^{\eta,\d}, \tilde w^{\eta,\d},\tilde s^{\eta,\d})}{2} \bigg|\sum_{j=1}^{d} \tilde\xi^{\eta,\d}_j \tilde s^{\eta,\d}_j \sigma^j(\tilde s^{\eta,\d})\bigg|^2\notag\\
&\qquad-(V^0_w)^{1-m}\Phi\big(\tilde s^{\eta,\d},\varpi_\xi H^{\eta,\d}+\varpi H^{\eta,\d}_\xi\big)\big(P^{\eta,\d}\big)\notag\\
&=-\frac{V^0_{ww}(\tilde t^{\eta,\d}, \tilde w^{\eta,\d},\tilde s^{\eta,\d})}{2} \bigg|\sum_{j=1}^{d} \tilde\xi^{\eta,\d}_j \tilde s^{\eta,\d}_j \sigma^j(\tilde s^{\eta,\d})\bigg|^2-(V^0_w)^{1-m}\Phi\big(\tilde s^{\eta,\d},\varpi_\xi H^{\eta,\d}\big)\big(P^{\eta,\d}\big)\notag\\
&\quad+(V^0_w)^{1-m}\big(\Phi\big(\tilde s^{\eta,\d},\varpi_\xi H^{\eta,\d}\big)-\Phi\big(\tilde s^{\eta,\d},\varpi_\xi H^{\eta,\d}+\varpi H^{\eta,\d}_\xi\big)\big)\big(P^{\eta,\d}\big)\notag\\
&=:I_1^{\eta,\d}+I_2^{\eta,\d}.\label{eq:ineq.pre.visco.supersol2}
\end{align}
We first bound $I^{\eta,\d}_1$,
\begin{align}
I_1^{\eta,\d}&\geqslant-\frac{V^0_{ww}(\tilde t^{\eta,\d}, \tilde w^{\eta,\d},\tilde s^{\eta,\d})}{2} \bigg|\sum_{j=1}^{d} \tilde\xi^{\eta,\d}_j \tilde s^{\eta,\d}_j \sigma^j(\tilde s^{\eta,\d}_j)\bigg|^2-(V^0_w)^{1-m}(1-\d)^m \Phi\left(\cdot,\varpi_\xi \right)(P^{\eta,\d})\notag\\
&=-(1-(1-\d)^m)\frac{V^0_{ww}(\tilde t^{\eta,\d}, \tilde w^{\eta,\d},\tilde s^{\eta,\d})}{2} \bigg|\sum_{j=1}^{d} \tilde\xi^{\eta,\d}_j \tilde s^{\eta,\d}_j \sigma^j(s^{\eta,\d}_j)\bigg|^2 \notag\\
&\qquad-(1-\d)^m\bigg(\frac{1}{2}\text{Tr}\big(c^{h^0}  \pa_{\xi\xi}\varpi\big)(P^{\eta,\d}) -a(\tilde t^{\eta,\d}, \tilde w^{\eta,\d},\tilde s^{\eta,\d})\bigg)\notag\\
&\geqslant (1-(1-\d)^m)\gamma_v \big|\tilde\xi^{\eta,\d}\big|^2-(1-\d)^m \left(\frac{d}{2}K_\Sigma K_\varpi\left(\big|\tilde\xi^{\eta,\d}\big|^{1+\frac{2}{m}}\vee 1\right)+K_a\right)\label{eq:ineq.pre.visco.supersol3}.
\end{align}
We obtained the first inequality using that $(V^0_w)^{1-m}$ is non-negative by Assumption~\ref{assum:frictionless}, that $\Phi$ is non-negative by construction and  homogeneous of degree $m$ by Assumption~\ref{assum:phi}, we then used the first corrector equation \eqref{1st-corrector-eq} satisfied by $\varpi$ to obtain the second equality and the definitions of the constants \eqref{eq:def.Ksigma} - \eqref{eq:def.gamma.v} to obtain the last inequality. By the convexity of $\Phi$, and estimate \eqref{eq:estimate.iota}, we have (we drop the argument of the functions for clarity in the next two sets of computations, they are taken at the point $(\tilde t^{\eta,\d}, \tilde w^{\eta,\d},\tilde s^{\eta,\d},\tilde\xi^{\eta,\d})$)
\begin{align}
&|I_2^{\eta,\d}|{\leq} \iota^{m-1} \left(\left|\Phi_x\left(s,\varpi_\xi H^{\eta,\d}\right)\right|+\left|\Phi_x\left(s,\varpi_\xi H^{\eta,\d}+\varpi H^{\eta,\d}_\xi\right)\right|\right)\left|\varpi H^{\eta,\d}_\xi\right|\notag\\
&{\leq} \iota^{m-1}(1-\d)^{m} K_{\Phi_x}\Bigg(1\vee \left|\varpi_\xi \right|^{m-1}+1\vee\left|\varpi_\xi h^\eta\left(\frac{\tilde\xi^{\eta,\d}}{\xi^{*,\d}}\right)+\frac{\varpi}{|\xi^{*,\d}|} h^\eta_\xi\left(\frac{\tilde\xi^{\eta,\d}}{\xi^{*,\d}}\right)\right|^{m-1}\Bigg)\eta \frac{\varpi}{|\tilde\xi^{\eta,\d}|}\notag\\
&{\leq} \iota^{m-1}(1-\d)^{m} 2C_mK_{\Phi_x}\left(1\vee \left|\varpi_\xi \right|^{m-1}+1\vee\left|\frac{\varpi}{|\xi^{*,\d}|} h_\xi\left(\frac{\tilde\xi^{\eta,\d}}{\xi^{*,\d}}\right)\right|^{m-1}\right)\eta \frac{\varpi}{|\tilde\xi^{\eta,\d}|},\label{eq:ineq.pre.visco.supersol4}
\end{align}
where the second inequality is obtained by \eqref{eq:property.h.eta}, the definition \eqref{eq:def.Phia} of $K_{\Phi_x}$ and the $m-1$-homogeneity of $\Phi_x$. For the third inequality, we set $C_m\geqslant 1$ a constant so that \linebreak
$(a+b)^{m-1}{\leq} C_m (a^{m-1}+b^{m-1})$ for all $a,b>0$, and use the estimate $0{\leq} H^{\eta,\d}{\leq} 1$. 

Assume now that $ |\tilde\xi^{\eta,\d}|^{2/m} \geqslant 1$ and recall that $K_{\varpi}\geqslant 1$, then
\begin{align}
|I_2^{\eta,\d}|&{\leq} 2\iota^{m-1}(1-\d)^{m}C_mK_{\Phi_x}\left(1\vee \left|\varpi_\xi \right|^{m-1}+1\vee\left|\eta\frac{\varpi}{|\tilde\xi^{\eta,\d}|} \right|^{m-1}\right)\eta \frac{\varpi}{|\tilde\xi^{\eta,\d}|}\notag\\
&{\leq} 2\iota^{m-1}(1-\d)^{m} C_m K_{\Phi_x} \left(1\vee \left|K_\varpi |\tilde\xi^{\eta,\d}|^{\frac{2}{m}} \right|^{m-1}+1\vee\left|K_\varpi |\tilde\xi^{\eta,\d}|^{\frac{2}{m}} \right|^{m-1}\right)\eta K_\varpi |\tilde\xi^{\eta,\d}|^{\frac{2}{m}}\notag\\
&{\leq} 4\iota^{m-1}(1-\d)^{m} C_m K_{\Phi_x} K_\varpi ^{m}\eta  |\tilde\xi^{\eta,\d}|^{2},\label{eq:ineq.pre.visco.supersol5}
\end{align}
where the first inequality follows from the property of $h^\eta$ (see \eqref{eq:property.h.eta}), the second from the definition \eqref{eq:def.Kvarpi} of $K_{\varpi}$ and the third from the assumption on $\tilde\xi^{\eta,\d}$.
Hence if $ |\tilde\xi^{\eta,\d}| \geqslant 1$, joining together \eqref{eq:ineq.pre.visco.supersol1}, \eqref{eq:ineq.pre.visco.supersol2}, \eqref{eq:ineq.pre.visco.supersol3}, and \eqref{eq:ineq.pre.visco.supersol5}, we get
\begin{align}
&\left( (1-(1-\d)^m)\gamma_v 
-4\iota^{m-1}(1-\d)^{m} C_m K_{\Phi_x} K_\varpi ^{m}\eta\right)  |\tilde\xi^{\eta,\d}|^{2}\label{eq:ineq.constants.supersol}\\
 &\qquad{\leq} (1-\d)^m \left(d\frac{K_\Sigma K_\varpi\left|\tilde\xi^{\eta,\d}\right|^{1+\frac{2}{m}}}{2}+K_a\right)+  K_0+\frac{d}{2}K_{\Sigma}\left(K'_{\varpi}C^*+2\eta K'_{\varpi} +K_{\varpi}\left|\tilde\xi^{\eta,\d}\right|^{1+\frac{2}{m}}\right).\notag
 \end{align}
Let $\eta_{\delta}=\frac{(1-(1-\d)^m)\gamma_v }{8\iota^{m-1}(1-\d)^{m} C_m K_{\Phi_x} K_\varpi ^{m}}$, for $\eta\in (0,\eta_\d\wedge 1)$ we have 
\begin{equation*}
\frac{(1-(1-\d)^m)\gamma_v \left|\xi^{\eta,\d}\right|^2}{2}{\leq} \left( (1-(1-\d)^m)\gamma_v 
-4\iota^{m-1}(1-\d)^{m} C_m K_{\Phi_x} K_\varpi ^{m}\eta\right)  |\tilde\xi^{\eta,\d}|^{2}.
\end{equation*}
Thus, under the assumption $ |\tilde\xi^{\eta,\d}| \geqslant 1$ and for $\eta\in (0,\eta_\d\wedge 1)$, \eqref{eq:ineq.constants.supersol} leads to 
\begin{equation*}
\left|\tilde \xi^{\eta,\d}\right|^2{\leq} \frac{2\left|\tilde\xi^{\eta,\d}\right|^{1+\frac{2}{m}}dK_\Sigma K_\varpi+2\left(K_a+  K_0\right)+dK_{\Sigma}K'_{\varpi}(C^*+2)}{(1-(1-\d)^m)\gamma_v}.
\end{equation*}
This shows, thanks to the definition \eqref{eq:defdeltastar} of $\xi^{*,\d}$ that for all $\eta\in(0,\eta_\d\vee 1)$, $\left|\tilde\xi^{\eta,\d}\right|$ is bounded by $1\vee \xi^{*,\d}$. 
Hence, up to taking a subsequence, as $\eta\to 0$, for every $\delta \in (0,1)$, there exists $(\tilde  t^{\d},\tilde  w^{\d},\tilde  s^{\d},\tilde  \xi^{\d})\in \bar B_{r_0}(t^0,w^0,s^0)\times\R^d$ such that we have the following convergence as $\eta \to 0$,
\begin{equation*}
\tilde t^{\eta,\d}\to\tilde  t^{\d},~\tilde w^{\eta,\d}\to \tilde w^{\d},~
\tilde s^{\eta,\d}\to\tilde s^{\d},~\tilde \xi^{\eta,\d}\to\tilde  \xi^{\d}.
\end{equation*}
Since $h^\eta$ converges uniformly on compacts to $1$ and by continuity of the functions involved, we can now take the limit to $0$ in $\eta $ in \eqref{eq:supeta} to obtain 
\begin{align}\label{eq:finalsuper}
\notag&-E_2(\tilde t^{\d},\tilde w^{\d},\tilde s^{\d},\varphi^0_t, \varphi^0_w, \varphi^0_s,\varphi^0_{ww},\varphi^0_{ws},\varphi^0_{ss}) \geqslant E_1(\tilde t^{\d}, \tilde w^{\d},\tilde s^{\d},\tilde \xi^{\d},(1-\d)\varpi_\xi,(1-\d)\varpi_{\xi\xi})\\ 
\notag&\qquad=-\frac{V^0_{ww}(\tilde t^{\d}, \tilde w^{\d},\tilde s^{\d})}{2} \bigg|\sum_{j=1}^{d} \tilde\xi^\d_j \tilde s^{\d}_j \sigma^j(\tilde s^{\d})\bigg|^2+\frac{1-\d}{2}\text{Tr}\left(c^{h^0}(\tilde t^{\d},\tilde w^{\d},\tilde s^{\d})  \varpi_{\xi\xi}(\tilde t^{\d},\tilde w^{\d},\tilde s^{\d},\tilde \xi^{\d})\right)\\
\notag&\qquad\quad-(V^0_w(\tilde t^{\d},\tilde w^{\d},\tilde s^{\d}))^{1-m}(1-\d)^m\Phi\left(\tilde s^\d, \varpi_\xi(\tilde t^{\d},\tilde w^{\d},\tilde s^{\d},\tilde\xi^{\d})\right)\\
\notag&\qquad=(1-\d)^m a(\tilde t^{\d},\tilde w^{\d},\tilde s^{\d})+((1-\d)^m-1)\frac{V^0_{ww}(\tilde t^{\d},\tilde w^{\d},\tilde s^{\d})}{2} \bigg|\sum_{j=1}^{d} \tilde \xi_j^\d \tilde s^{\d}_j \sigma^j(\tilde s^{\d})\bigg|^2\\
\notag&\qquad\quad+((1-\d)-(1-\d)^m)\frac{1}{2}\text{Tr}\left(c^{h^0}(\tilde t^{\d},\tilde w^{\d},\tilde s^{\d})  \varpi_{\xi\xi}(\tilde t^{\d},\tilde w^{\d},\tilde s^{\d},\tilde\xi^{\d})\right)\\
&\qquad\geqslant(1-\d)^m a(\tilde t^{\d}, \tilde w^{\d},\tilde s^{\d}),
\end{align}
where we used the first corrector equation \eqref{1st-corrector-eq}, the signs of $V^0_{ww}(\tilde t^{\d},\tilde w^{\d},\tilde s^{\d})$, $((1-\d)-(1-\d)^m)$, $((1-\d)^m-1)$, the convexity of $\varpi$, the positive definiteness of $c^{h^0}$ and the fact that the trace of the product of symmetric positive semidefinite matrices is non-negative. Due to the compactness of $\bar B_{r_0}(t^0,w^0,s^0)$, up to a subsequence, we can take the limit of the family $(\tilde t^{\d},\tilde w^{\d},\tilde s^{\d})_{\d\in(0,1)}$ as $\d\to 0$ and obtain
\begin{align*}
\tilde t^{\d}\to \tilde t,\qquad\tilde w^{\d}\to \tilde w,\qquad\tilde s^{\d}\to \tilde s.
\end{align*}
Using \eqref{eq:supsolineq}, one can show by a classical argument of the theory of viscosity solution (see for e.g. \cite{CIL}) that
\begin{equation*}
(\tilde t,\tilde w,\tilde s)=(t^0,w^0,s^0).
\end{equation*}
Additionally, by continuity of $a$, $E_2$, and $(t,w,s)\mapsto c_0|(t,w,s)-(t^0,w^0,s^0)|$ we have the following limits as $\d\to 0$
\begin{align*}
-E_2(\tilde t^{\d},\tilde w^{\d},\tilde s^{\d},\varphi^0_t, \varphi^0_w, \varphi^0_s,\varphi^0_{ww},\varphi^0_{ws},\varphi^0_{ss})&\to -E_2(t^0,w^0,s^0,\phi_t, \phi_w, \phi_s,\phi_{ww},\phi_{ws},\phi_{ss}),\\
(1-\d)^m a(\tilde t^{\d},\tilde w^{\d},\tilde s^{\d})&\to a (t^0,w^0,s^0),
\end{align*}
which gives the supersolution property for $u_*$ via \eqref{eq:finalsuper}.
\end{proof}

\subsection{The Subsolution Property}\label{ss.subsol}
\begin{prop}\label{prop:subsol}
	Under the assumption of Theorem~\ref{thm:expansion}, the function $u^*$ is an upper semicontinuous viscosity subsolution of the second corrector equation \eqref{2nd-corrector-eq}.
\end{prop}
\begin{proof}
The proof is based on \cite[Proof of Proposition 6.3]{MMKS}. 
Let $(t^0,w^0,s^0)\in\cD$ and $\phi\in \cC^{1,2,2}(\cD,\R)$ such that $(t^0,w^0,s^0)$ is a strict maximizer of $u^*-\phi$ on $\cD$. Then, for all $(t,w,s)\in\cD\backslash\{(t^0,w^0,s^0)\}$ it holds
\begin{align}
	\label{eq:subsolineq}
	0=u^*(t^0,w^0,s^0)-\phi(t^0, w^0, s^0) > u^*(t,w,s)-\phi(t,w,s).
\end{align}
By definition of $u^*$ (see \eqref{eq:ue*u}) there exists a family $(t^{{\e}},w^{{\e}},s^{{\e}},h^{{\e}})_{{\e}>0}$ such that 
\begin{align}
	&(t^{{\e}},w^{{\e}},s^{{\e}},h^{{\e}})\to (t^0,w^0,s^0,h^0(t^0,w^0,s^0)),\quad u^{\e,*}(t^{{\e}},w^{{\e}},s^{{\e}},h^{{\e}})\to u^*(t^0,w^0,s^0)\notag\\
	&\mbox{ and }p^\e:=u^{\e,*}(t^{{\e}},w^{{\e}},s^{{\e}},h^{{\e}})-\phi(t^{{\e}},w^{{\e}},s^{{\e}})\to 0\label{eq:def.pe.subsol}.
\end{align}
By Assumption~\ref{assum:bound} and the regularity of $h^0$, there exist $\e_0,r_0>0$, $\a\in(0,r_0)$, $\e_0<1$ such that $\bar B_{r_0}(t^0,w^0,s^0)\subseteq \cD$ , and such that for all $\e\in(0,\e_0)$ we have 
\begin{align*}
	b^*:=\sup\left\{u^{\e*}(t,w,s,h):\left|(t,w,s,h)-(t^0,w^0,s^0,h^0(t_0,w_0,s_0))\right|{\leq} r_0,~ \e\in(0,\e_0)\right\}<\infty,\\
	\mbox{and }|h^0(t,w,s)-h^0(t^0,w^0,s^0)|\leq \frac{r_0}{4}\mbox{ if }|(t,w,s)-(t^0,w^0,s^0)|\leq \a.
\end{align*}
For $(t,w,s)\in (\bar B_\a(t^0,w^0,s^0)\backslash B_{\a/2}(t^0,w^0,s^0))$ and $(t',w',s')\in\bar B_{\a/4}(t^0,w^0,s^0)$ we have
$$\left|(t,w,s)-(t',w',s')\right|^4\geqslant \Big(\frac{\alpha}{4}\Big)^4.$$ Denote $M:=2+b^*+\sup\{-\phi(t,w,s):(t,w,s)\in B_\a(t^0,w^0,s^0)\}<\infty$. Thanks to Assumption~\ref{assum:2ndcorrect} we can define $\d_0:=\inf_{\{(t,w,s)\in\bar B_{r_0},|\xi|\geq \frac{r_0}{4}\}}\frac{\varpi(t,w,s,\xi)}{|\xi|^{1+2/m}}>0$ and
$$c_0:=\frac{M}{(\frac{\alpha}{4})^4\wedge \delta^2_0 (\frac{r_0}{4})^{1+\frac{2}{m}}}.$$

The growth of $\varpi$ in its last variable at infinity assumed in Assumption~\ref{assum:2ndcorrect} provides the inequality $$\varpi^2\left(t^\e,w^\e,s^\e, \frac{h^\e-h^0(t^\e,w^\e,s^\e)}{\es}\right)\leq C(t^\e,w^\e,s^e)\frac{|h^\e-h^0(t^\e,w^\e,s^\e)|^{2+4/m}}{\e^{m^*(2+4/m)}}.$$ Then, the continuity of $h^0$ (Assumption~\ref{assum:frictionless}), and the convergence of $(t^\e,w^\e,s^\e,h^\e)$ to the point $(t^0,w^0,s^0,h^0)$ allows us to choose $\e_0>0$ smaller to also have for $0<\e\leq \e_0$,
\begin{align}
	\label{eq:varpi2.c0}
	\e^{2m^*(1+\frac{2}{m})}\varpi^2\left(t,^\e,w^\e,s^\e, \frac{h^\e-h^0(t^\e,w^\e,s^\e)}{\es}\right)\leq \frac{1}{3c_0}.
\end{align}
Since $m>2$, we can also take $\e_0$ small enough so that 
\begin{align}
	\label{eq:varpi.c0}
	\e^{2m^*}\varpi\left(t^\e,w^\e,s^\e,\frac{h^\e-h^0(t^\e,w^\e,s^\e)}{\es}\right)\leq \frac{1}{3},\, \\
	\left|(t^\e,w^\e,s^\e)-(t^0,w^0,s^0)\right|^4\leq \frac{\a}{4},\ \mbox{ and }\  |p^\e|\leq 1,\mbox{ for }\e\in(0,\e_0).\notag
\end{align}
We now define 
\begin{align*}
	\tilde \varpi^\e(t,w,s,h)&:=c_0\e^{2m^*(1+\frac{2}{m})}\varpi^2\left(t,w,s, \frac{h-h^0(t,w,s)}{\es}\right),\\
	{\color{black}\phi^\e(t,w,s,h)}&{\color{black}:=c_0\Big((t-t^\e)^4 + (w-w^\e)^4+(s-s^\e)^4\Big)+\tilde \varpi^\e(t,w,s,h).}
\end{align*}
By~\eqref{eq:varpi2.c0} we have $|\phi^\e(t^\e,w^\e,s^\e,h^\e)|\leq 1/3$. The definitions of $\d_0$, $c_0$ and $M$ were set so that $\phi^\e\geq M$ if $|h-h^0(t,w,s)|\geq \frac{r_0}{4}$ or if $(t,w,s)\in \bar B_\a(t^0,w^0,s^0)\backslash B_{\a/2}(t^0,w^0,s^0)$, for $\e\in(0,\e_0)$. Now define for $\eta\in(0,1)$
\begin{align*}
	\bar \phi^\e:= &p^\e+\phi,~~\varphi^\e:=\bar \phi^\e +\phi^\e,\\
	\psi^{\e,\eta}(t,w,s,h):=&V^0(t,w,s)-\ets(\bar\phi^\e+\phi^\e)(t,w,s,h)\\
	&-\e^{4m^*}(1+\eta)\varpi\left(t,w,s,\frac{h-h^0(t,w,s)}{\es}\right).
\end{align*}

\emph{Claim: $V^\e_*-\psi^{\e,\eta}$ is a lower semicontinuous function which, for $0<\e\leqslant \e_0$, attains its minimum on $\bar B_\a(t^0,w^0,s^0)\times \bar B_{r_0}(h^0(t^0,w^0,s^0))$ at an interior point $(\tilde t^{\e,\eta}, \tilde w^{\e,\eta},\tilde s^{\e,\eta},\tilde h^{\e,\eta})$ such that $|\tilde h^{\e,\eta}-h^0(\tilde t^{\e,\eta},\tilde w^{\e,\eta},\tilde s^{\e,\eta})|+|\tilde h^{\e,\eta}-h^0(t^0,w^0,s^0)|\leq r_1$ for some $r_1>0$ independent of $\e, \eta$.} Indeed by~\eqref{eq:varpi2.c0}, \eqref{eq:varpi.c0} and the inequality $0<\eta<1$, it holds that $\ems(V^\e_*-\psi^{\e,\eta})(t^\e,w^\e,s^\e,h^\e)<1$, while if $(t,w,s)\in \bar B_\a(t^0,w^0,s^0)\backslash B_{\a/2}(t^0,w^0,s^0)$ or if $(t,w,s)\in \bar B_{\a/2}(t^0,w^0,s^0)$ and $|h-h^0(t,w,s)|\geq \frac{r_0}{4}$, by definition of $b^*$, the bound on $p^\e$, the definition of $M$, $c_0$, the fact that $(t^\e,w^\e,s^\e)\in\bar B_{\a/4}(t^0,w^0,s^0)$ and the non-negativity of $\varpi$, the inequality $\ems(V^\e_*-\psi^{\e,\eta})(t^\e,w^\e,s^\e,h^\e)\geqslant 1$ holds. Furthermore, by the triangular inequality, we can choose $r_1=5r_0/4$.  

Denote $\tilde \xi^{\e,\eta}=\ems(\tilde h^{\e,\eta}-h^0(\tilde t^{\e,\eta}, \tilde w^{\e,\eta},\tilde s^{\e,\eta}))$. Now using the viscosity property of $V_*^\e$ at $(\tilde t^{\e,\eta}, \tilde w^{\e,\eta},\tilde s^{\e,\eta},\tilde h^{\e,\eta})$ for the test function $\psi^{\e,\eta}$, we obtain
\begin{equation*}
	\cG^\e(\psi^{\e,\eta})(\tilde t^{\e,\eta}, \tilde w^{\e,\eta},\tilde s^{\e,\eta},\tilde h^{\e,\eta})\geqslant 0.
\end{equation*}
Using equation~\eqref{eq:remainder.viscosity} from Proposition~\ref{prop:remainder.estimate} applied to $\psi^{\e,\eta}$ with $\nu=\varphi^\e$ and $\chi=(1+\eta)\varpi$, we get
\begin{align}
	0&{\leq} \cG^\e(\psi^{\e,\eta})(\tilde t^{\e,\eta}, \tilde w^{\e,\eta},\tilde s^{\e,\eta},\tilde h^{\e,\eta})
	=\e^{2m^*}E_2(\tilde t^{\e,\eta}, \tilde w^{\e,\eta},\tilde s^{\e,\eta}, \varphi^\e_t,  \varphi^\e_w, \varphi^\e_s, \varphi^\e_{ww}, \varphi^\e_{ws}, \varphi^\e_{ss})\notag\\
	&+\e^{2m^*}E_1(\tilde t^{\e,\eta}, \tilde w^{\e,\eta},\tilde s^{\e,\eta},\tilde \xi^{\e,\eta},(1+\eta)\varpi_\xi(\tilde t^{\e,\eta}, \tilde w^{\e,\eta},\tilde s^{\e,\eta},\tilde \xi^{\e,\eta}),(1+\eta)\varpi_{\xi\xi}(\tilde t^{\e,\eta}, \tilde w^{\e,\eta},\tilde s^{\e,\eta},\tilde \xi^{\e,\eta}))\notag\\
	&+R^\e(\tilde t^{\e,\eta}, \tilde w^{\e,\eta},\tilde s^{\e,\eta},\tilde h^{\e,\eta},\psi^{\e,\eta})\label{eq:expansesol}\\
	&+\ets\left((\pa_w \psi^\e)^{1-m}(\Phi(s,(1+\eta)\varpi_\xi)-\Phi(s,(1+\eta)\varpi_{\xi}+\ems \tilde \varpi^\e_h))\right)(\tilde t^{\e,\eta}, \tilde w^{\e,\eta},\tilde s^{\e,\eta},\tilde h^{\e,\eta})\notag
\end{align}
where we have slightly abused notation for $\varpi_\xi$ since this function does not depend on $h$ but on $\xi$. Note that this last term is in fact non positive. Indeed, 
$$(1+\eta)\varpi_{\xi}+\ems \tilde \varpi^\e_h=(1+\eta)\varpi_{\xi}\bigg(1+\frac{2c_0 \e^{\frac{4m^*}{m}}}{1+\eta}\varpi\bigg)$$ and $\varpi\geq 0$. 
Thus, $(1+\frac{2c_0 \e^{{4m^*/m}}}{1+\eta}\varpi)\geq 1$ and by the $m$-homogeneity of $\Phi$
\begin{align}
	&\ets\Big((\pa_w \psi^\e)^{1-m}(\Phi(s,(1+\eta)\varpi_\xi)-\Phi(s,(1+\eta)\varpi_{\xi}+\ems \tilde \varpi^\e_h))\Big)\notag\\
	&\qquad= \ets(\pa_w \psi^\e)^{1-m}\Phi(s,(1+\eta)\varpi_\xi)\bigg(1-\bigg(1+\frac{2c_0 \e^{\frac{4m^*}{m}}}{1+\eta}2c_0\varpi\bigg)^m\bigg)\leq 0.\label{eq:phihomo}
\end{align}
\emph{Claim: Up to reducing $\e_0$, $\tilde P^{\e,\eta}:=(\tilde t^{\e,\eta}, \tilde w^{\e,\eta},\tilde s^{\e,\eta},\tilde h^{\e,\eta})$ is in the set defined in \eqref{negpower} for $0<\e\leqslant \e_0$.} We need to bound $|(V^0_w-\pa_w \psi^\e)(\tilde P^{\e,\eta})|$. We have
\begin{align*}
	\big(V^0_w-\pa_w \psi^{\e,\eta}\big)(\tilde P^{\e,\eta}) =& \ets\Big(\phi_w(\tilde P^{\e,\eta})+4c_0|\tilde w^{\e,\eta}-w^\e|^3+2c_0\e^{2m^*(1+\frac{2}{m})}(\varpi\pa_w\varpi)(\tilde P^{\e,\eta})\Big)\\
	&+\efs(1+\eta)\pa_w\varpi(\tilde P^{\e,\eta}).
\end{align*}
Since $\varpi\in\cC_m$ by Assumption~\ref{assum:2ndcorrect}, we have on $\bar B_\a(t^0,w^0,s^0)\times \bar B_{r_0}(h^0(t^0,w^0,s^0))$ with $C_0=\sup\{C(t,w,s)~|~(t,w,s)\in\bar B_\a(t^0,w^0,s^0)\}$
\begin{align*}
	\e^{m^*(1+\frac{2}{m})}\varpi\bigg(t,w,s,\frac{h-h^0(t,w,s)}{\es}\bigg) &\leqslant C_0|h-h^0(t,w,s)|^{1+\frac{2}{m}},\\
	\e^{m^*(1+\frac{2}{m})}\bigg(\varpi_w-\ems\big(h^0_w(t,w,s)\big)^\top\varpi_{\xi}\bigg)\bigg(t,w,s,\frac{h-h^0(t,w,s)}{\es}\bigg) &\leqslant C_0|h-h^0(t,w,s)|^{1+\frac{2}{m}}.
\end{align*}
Since $|\tilde h^{\e,\eta}-h^0(\tilde t^{\e,\eta}, \tilde w^{\e,\eta},\tilde s^{\e,\eta})|$ is bounded, we get that $|\emts(V^0_w-\pa_w \psi^{\e,\eta})(\tilde P^{\e,\eta})|$ is too. Then taking $\e_0$ smaller if necessary, and using that $V^0_w$ is bounded away from $0$ on $B_{r_0}(t^0,w^0,s^0)$, the claim obtains. 

Similarly, the scaling of $\varpi$ and its derivatives, yields that $\tilde \varpi^\e$ and its derivatives in fact admit uniform bounds in $\e$. Then, we can apply item (ii) of Proposition~\ref{prop:remainder.estimate} and we have on the bounded set $\{(\tilde t^{\e,\eta}, \tilde w^{\e,\eta},\tilde s^{\e,\eta},\tilde h^{\e,\eta})~|~\e\in(0,\e_0),\eta\in(0,1)\}$
\begin{equation*}
	\frac{| R^\e(\tilde t^{\e,\eta}, \tilde w^{\e,\eta},\tilde s^{\e,\eta},\tilde h^{\e,\eta},\psi^{\e,\eta})|}{\e^{2m^*}}{\leq} C(1+|\tilde \xi^{\e,\eta}|^2)^{\frac{1}{2}+\frac{1}{m}},
\end{equation*}
for $C$ that does not depend on $\e$ nor $\eta$.  Then, thanks to the continuity of $h^0$, $E_2$, the regularity of $\varphi^\e$, and the boundedness of $(\tilde t^{\e,\eta}, \tilde w^{\e,\eta},\tilde s^{\e,\eta},\tilde h^{\e,\eta})_{\{\e\in(0,\e_0],~\eta\in(0,1)\}}$ we obtain the inequality
\begin{align*}
	E_2(\tilde t^{\e,\eta}, \tilde w^{\e,\eta},\tilde s^{\e,\eta}, \varphi^\e_t,  \varphi^\e_w, \varphi^\e_s, \varphi^\e_{ww}, \varphi^\e_{ws}, \varphi^\e_{ss})+&\frac{| R^\e(\tilde t^{\e,\eta}, \tilde w^{\e,\eta},\tilde s^{\e,\eta},\tilde h^{\e,\eta},\psi^{\e,\eta})|}{\ets}\\
	&\qquad\qquad\qquad{\leq} C(1+|\tilde \xi^{\e,\eta}|^2)^{\frac{1}{2}+\frac{1}{m}},
\end{align*}
for some $C>0$ independent of $\e$ and $\eta$. 
This inequality implies, thanks to \eqref{eq:expansesol} and \eqref{eq:phihomo}, the following estimate
\begin{align}\label{eq:Ve.supersolution2}
	\notag&-E_1(\tilde t^{\e,\eta}, \tilde w^{\e,\eta},\tilde s^{\e,\eta},\tilde \xi^{\e,\eta},(1+\eta)\varpi_\xi(\tilde t^{\e,\eta}, \tilde w^{\e,\eta},\tilde s^{\e,\eta},\tilde \xi^{\e,\eta}),(1+\eta)\varpi_{\xi\xi}(\tilde t^{\e,\eta}, \tilde w^{\e,\eta},\tilde s^{\e,\eta},\tilde \xi^{\e,\eta}))\\
	&\leq C(1+|\tilde \xi^{\e,\eta}|^2)^{\frac{1}{2}+\frac{1}{m}}.
\end{align}
	
Additionally, using the definition of $E_1$ from \eqref{correc1}, we have
\begin{align*}
	&-E_1(\tilde t^{\e,\eta}, \tilde w^{\e,\eta},\tilde s^{\e,\eta},\tilde \xi^{\e,\eta},(1+\eta)\varpi_\xi(\tilde t^{\e,\eta}, \tilde w^{\e,\eta},\tilde s^{\e,\eta},\tilde \xi^{\e,\eta}),(1+\eta)\varpi_{\xi\xi}(\tilde t^{\e,\eta}, \tilde w^{\e,\eta},\tilde s^{\e,\eta},\tilde \xi^{\e,\eta}))\\
	&=\frac{V^0_{ww} (\tilde t^{\e,\eta}, \tilde w^{\e,\eta},\tilde s^{\e,\eta})}{2} \bigg|\sum_{j=1}^{d} \tilde \xi^{\e,\eta}_j \tilde s^{\e,\eta} \sigma^j(\tilde s^{\e,\eta})\bigg|^2\\
	&\qquad+(V^0_w(\tilde t^{\e,\eta}, \tilde w^{\e,\eta},\tilde s^{\e,\eta}))^{1-m}\Phi\left(\tilde s^{\e,\eta},\left(1+\eta)\varpi_\xi\right(\tilde t^{\e,\eta}, \tilde w^{\e,\eta},\tilde s^{\e,\eta},\tilde \xi^{\e,\eta})\right)\\
	&\qquad-\frac{1+\eta}{2}\text{Tr}\left(c^{h^0}(\tilde t^{\e,\eta}, \tilde w^{\e,\eta},\tilde s^{\e,\eta}) \varpi_{\xi\xi}(\tilde t^{\e,\eta}, \tilde w^{\e,\eta},\tilde s^{\e,\eta},\tilde\xi^{\e,\eta})\right).
\end{align*}
We now use that $\varpi$ solves the first corrector equation \eqref{1st-corrector-eq} to obtain the equation (we drop the argument of the functions for clarity: $(\tilde t^{\e,\eta}, \tilde w^{\e,\eta},\tilde s^{\e,\eta})$ for $V^0$, its derivatives and $a$, $(\tilde t^{\e,\eta}, \tilde w^{\e,\eta},\tilde s^{\e,\eta},\tilde \xi^{\e,\eta})$ for $\varpi$ and its derivatives)
\begin{align}
	&-E_1(\tilde t^{\e,\eta}, \tilde w^{\e,\eta},\tilde s^{\e,\eta},\tilde \xi^{\e,\eta},(1+\eta)\varpi_\xi,(1+\eta)\varpi_{\xi\xi})= - \eta\frac{V^0_{ww}}{2} \bigg|\sum_{j=1}^{d} \tilde \xi^{\e,\eta}_j \tilde s^{\e,\eta} \sigma^j(\tilde s^{\e,\eta})\bigg|^2\notag\\
	&\qquad\qquad-(1+\eta)a+(V^0_w)^{1-m}\left(\Phi\left(\tilde s^{\e,\eta},(1+\eta)\varpi_\xi\right)-(1+\eta)\Phi\left(\tilde s^{\e,\eta},\varpi_\xi\right)\right).\label{eq:subsol.estimate.xi1}
\end{align}
	
We note that thanks to the boundedness of $(\tilde t^{\e,\eta}, \tilde w^{\e,\eta},\tilde s^{\e,\eta})_{\{\e\in(0,\e_0],~\eta\in(0,1)\}}$, and the continuity of $a$ the term $(1+\eta)a(\tilde t^{\e,\eta}, \tilde w^{\e,\eta},\tilde s^{\e,\eta})$ is bounded uniformly in $\e>0$ and $\eta\in(0,1)$.
Additionally, we have by $m$-homogeneity and non-negativity of $\Phi$ for $\eta\in(0,1)$
$$		\Phi\left(\tilde s^{\e,\eta},(1+\eta)\varpi_\xi\right)-(1+\eta)\Phi\left(\tilde s^{\e,\eta},\varpi_\xi\right)=((1+\eta)^m-(1+\eta))\Phi\left(\tilde s^{\e,\eta},\varpi_\xi\right)\geqslant 0.$$
Putting together this inequality with  \eqref{eq:Ve.supersolution2} and \eqref{eq:subsol.estimate.xi1}, finally yields
\begin{equation*}
	-\eta \frac{V^0_{ww} (\tilde t^{\e,\eta}, \tilde w^{\e,\eta},\tilde s^{\e,\eta})}{2} \bigg|\sum_{j=1}^{d} \tilde \xi^{\e,\eta}_j \tilde s^{\e,\eta} \sigma^j(\tilde s^{\e,\eta})\bigg|^2{\leq} C (1+|\tilde \xi^{\e,\eta}|^2)^{\frac{1}{2}+\frac{1}{m}},
\end{equation*}
for $C>0$ independent of $\e>0$ and $\eta\in(0,1)$. Thanks to the non-degeneracy of $\sigma\sigma^\top$, and the fact that $(\tilde t^{\e,\eta}, \tilde w^{\e,\eta},\tilde s^{\e,\eta})_{\{\e\in(0,\e_0],~\eta\in(0,1)\}}\subseteq B_{r_0}(t^0,w^0,s^0)\subseteq\cD$, we can find $c>0$ such that $$	-\eta \frac{V^0_{ww} (\tilde t^{\e,\eta}, \tilde w^{\e,\eta},\tilde s^{\e,\eta})}{2} \bigg|\sum_{j=1}^{d} \tilde \xi^{\e,\eta}_j \tilde s^{\e,\eta} \sigma^j(\tilde s^{\e,\eta})\bigg|^2\geqslant \eta c|\tilde \xi^{\e,\eta}|^{2}.$$ 
Thus, since $m>2$, we deduce that for all $\eta\in(0,1)$, the family $\{\tilde \xi^{\e,\eta}:\e\in(0,\e_0]\}$ is bounded. 
	
Now, for every $\eta\in(0,1)$, we extract a subsequence of $\{(\tilde t^{\e,\eta},\tilde w^{\e,\eta},\tilde s^{\e,\eta},\tilde h^{\e,\eta},\tilde \xi^{\e,\eta})~|~\e\in(0,\e_0)\}$ converging to $(\tilde t^{\eta}, \tilde w^{\eta},\tilde s^{\eta},\tilde h^{\eta},\tilde \xi^{\eta})$. 
Additionally, the boundedness of $\{\tilde \xi^{\e,\eta}:\e\in(0,\e_0]\}$ allows us to use the last point of Proposition~\ref{prop:remainder.estimate} to pass to the limit in \eqref{eq:expansesol} and obtain that for all $\eta\in (0,1)$, (note that $\varphi^\e$ converges uniformly on compacts to $\phi$ and that $(\tilde t^{\eta},\tilde w^{\eta},\tilde s^{\eta},\tilde h^{\eta})$ as limit of $(\tilde t^{\e,\eta},\tilde w^{\e,\eta},\tilde s^{\e,\eta},\tilde h^{\e,\eta})$ along a subsequence is in the set of~\eqref{negpower})
\begin{align*}
	0&{\leq} E_2(\tilde t^{\eta}, \tilde w^{\eta},\tilde s^{\eta}, \phi_t,  \phi_w, \phi_s, \phi_{ww}, \phi_{ws}, \phi_{ss})\notag\\
	&+E_1(\tilde t^{\eta}, \tilde w^{\eta},\tilde s^{\eta},\tilde \xi^{\eta},(1+\eta)\varpi_\xi(\tilde t^{\eta}, \tilde w^{\eta},\tilde s^{\eta},\tilde \xi^{\eta}),(1+\eta)\varpi_{\xi\xi}(\tilde t^{\eta}, \tilde w^{\eta},\tilde s^{\eta},\tilde \xi^{\eta})).
\end{align*}
We use \eqref{eq:subsol.estimate.xi1} one more time and obtain that 
$$E_2(\tilde t^{\eta}, \tilde w^{\eta},\tilde s^{\eta}, \phi_t,  \phi_w, \phi_s, \phi_{ww}, \phi_{ws}, \phi_{ss})\geq -(1+\eta)a.$$
Note that $\xi^\eta$ is not present in this inequality and $(\tilde t^{\eta}, \tilde w^{\eta},\tilde s^{\eta})$ is bounded for $\eta\in (0,1)$. 
One can now take the limit $\eta\to 0$ to obtain thanks to \eqref{eq:subsolineq} (using classical arguments of the theory of viscosity solutions, see \cite{CIL}) that a subsequence of $(\tilde t^{\eta}, \tilde w^{\eta},\tilde s^{\eta})$ converges to $(t^0,w^0,s^0)$. Then passing to the limit in the equality above we obtain
\begin{equation*}
	-E_2(t^0,w^0,s^0,\phi_t,\phi_w,\phi_s,\phi_{ww},\phi_{ws},\phi_{ss}){\leq} a(t^0,w^0,s^0),
\end{equation*}
which gives the viscosity subsolution property.
\end{proof}

\subsection{The Terminal Condition}
\begin{prop}\label{prop:term}
Under the assumption of Theorem~\ref{thm:expansion}, $u^*$ satisfies
\begin{align*}
\limsup_{(t,w',s')\to (T,w,s)}u^*(t,w',s')= 0\ \mbox{for all }\ (w,s)\in\R_{++}\times\R_{++}^d.
\end{align*}
Thus, the upper semicontinuous extension of $u^*$ to $\cD_T$ and lower semicontinuous extensions of $u_*$ to $\cD_T$ satisfy
$$u_*(T,w,s)=u^*(T,w,s)=0.$$
\end{prop}
\begin{proof}
Due to the inequality 
$0\leq u_*\leq u^*$ it is sufficient to show that $u^*(T,w,s)\leq 0$. 
Assume on the contrary that  $u^*(T,w^0,s^0)\geq 5\delta>0$ for some $(w^0,s^0)\in \R_{++}^{d+1}$. Similarly to the proof of Proposition~\ref{prop:subsol}, 
	by definition of $u_*$ (see \eqref{eq:ue*u}) there exists a sequence $(t^{{\e}},w^{{\e}},s^{{\e}},h^{{\e}})_{{\e}>0}$ such that 
	\begin{align}
	&(t^{{\e}},w^{{\e}},s^{{\e}},h^{{\e}})\to (T,w^0,s^0,h^0(T,w^0,s^0)),\quad u^{\e,*}(t^{{\e}},w^{{\e}},s^{{\e}},h^{{\e}})\to u_*(T,w^0,s^0).\notag
\end{align}
Note that, the functions $u_*$ and $u^*$ are only defined on $\cD$ and then extended by semicontinuity to $\{T\}\times\R_{++}^{d+1}$, and that therefore, we can take $t^\e<T$ for all $\e>0$.
		
Similarly to the proof of Proposition~\ref{prop:subsol}, there exist $\e_0>0$, $r_0\geq\alpha>0$, $c_0>0$ large enough such that for all $\e\in(0,\e_0]$ we have the following estimates 
	\begin{align*}
		&|h^0(t,w,s)-h^0(T,w^0,s^0)|{\leq} \frac{r_0}{4},~\forall (t,w,s)\in \cD~\mbox{ such that }~|(t,w,s)- (T,w^0,s^0)|{\leq} \a,\\
		\notag&|(t^\e,w^\e,s^\e)-(T,w^0,s^0)|{\leq}  \frac{\a}{4},\qquad u^{\e,*}(t^{{\e}},w^{{\e}},s^{{\e}},h^{{\e}})\geq 4\d,\ |h^\e-h^0(t^\e,w^\e,s^\e)|^2\leqslant \frac{\d}{c_0}\\
		&\notag \mbox{ and }~u^{\e,*}(t,w,s,h)-\bar \phi^\e(t,w,s,h)	<0\mbox{ on }B_{\a}\backslash B_{0,\a},
	\end{align*}
	where 
	\begin{align*}
		&\bar \phi^\e(t,w,s,h):=c_0\left(|(t,w,s)-(t^\e,w^\e,s^\e)|^4+ |h-h^0(t,w,s)|^2\right),\\
		&B_{\a}:=(T-\alpha, T)\times B_\a(w^0,s^0)\times B_{r_0}(h^0(T,w^0,s^0)),\\
		&B_{0,\a}:=\Big\{(t,w,s,h)\in B_\a:(t,w,s)\in (T-\frac{\alpha}{2},T)\times B_{\a/2}(w^0,s^0)\\
		&\qquad\qquad\qquad\qquad\qquad\qquad\qquad\qquad\qquad\qquad\qquad\qquad\mbox{ and }h\in B_{r_0/2}(h^0(T,w^0,s^0))\Big\}.
	\end{align*}
	Define the functions $\phi^\e(t,w,s,h):=\d\frac{T-t}{T-t^\e}+\bar \phi^\e(t,w,s,h)$ and $\psi^\e:=V^0-\ets \phi^\e$ . Then, similarly to the proof of \cite[Proposition 6.5]{MMKS} the function $V^\e_*-\psi^\e$ admits a local minimizer $(\tilde t^\e,\tilde w^\e,\tilde s^\e,\tilde h^\e)\in \bar B_\a$ satisfying 
$u^{\e,*}(\tilde t^\e,\tilde w^\e,\tilde s^\e,\tilde h^\e)\geq \d$ and $t^\e<T$. Indeed, we have $\emts(V^\e_*-\psi^\e)(t^\e,w^\e,s^\e, h^\e)\leqslant -2\d$ and on $B_{\a}\backslash B_{0,\a}$ it holds that $(V^\e_*-\psi^\e)(t,w,s,h)>0$.

Now, using that $V^\e_*$ is a supersolution of \eqref{eq:final.cond.Ve} by Assumption~\ref{assum:Ve.viscosity}, and applying Proposition~\ref{prop:remainder.estimate} (ii) to $\psi^\e$ with $\nu=\phi^\e$ and $\varpi=0$, Equation~\eqref{eq:remainder.viscosity} yields
	\begin{align}
		0&{\leq} \cG^\e(\psi^{\e})(\tilde t^{\e}, \tilde w^{\e},\tilde s^{\e},\tilde h^{\e})
		=\e^{2m^*}E_2(\tilde t^{\e}, \tilde w^{\e},\tilde s^{\e}, \phi^\e_t,  \phi^\e_w, \phi^\e_s, \phi^\e_{ww}, \phi^\e_{ws}, \phi^\e_{ss})\notag\\
		&+\e^{2m^*}E_1(\tilde t^{\e}, \tilde w^{\e},\tilde s^{\e},\tilde \xi^{\e},0,0)+| R^\e(\tilde t^{\e}, \tilde w^{\e},\tilde s^{\e},\tilde h^{\e},\psi^{\e})|\notag\\
		&+\ets(\pa_w \psi^\e(\tilde t^{\e}, \tilde w^{\e},\tilde s^{\e},\tilde h^{\e}))^{1-m}(\Phi(\tilde s^\e,0)-\Phi(\tilde s^\e,\ems \bar\phi_h(\tilde t^{\e}, \tilde w^{\e},\tilde s^{\e},\tilde h^{\e}))),\label{eq:Ve.term}
	\end{align}
with $$\frac{| R^\e(\tilde t^{\e}, \tilde w^{\e},\tilde s^{\e},\tilde h^{\e},\psi^{\e})|}{\ets}{\leq} C(1+|\tilde \xi^{\e}|^2)^{\frac{1}{2}+\frac{1}{m}}.$$
Note that up to reducing $\e_0>0$ and taking $0<\d\leqslant \d_\e$ for some small enough $\delta_\e$, the point $(\tilde t^\e, \tilde w^\e, \tilde s^\e)$ is in the set defined in~\eqref{negpower}. Now, the last term in \eqref{eq:Ve.term} is 
$$-(V^0_w(\tilde t^{\e}, \tilde w^{\e},\tilde s^{\e})-\ets\bar \phi^\e_w(\tilde t^{\e}, \tilde w^{\e},\tilde s^{\e},\tilde h^{\e}))^{1-m}\Phi(\tilde s^\e,2c_0\tilde\xi^\e)=-C_\e |\tilde\xi^\e|^m,$$ 
where $C_\e$ is bounded from below away from $0$ as $\e\to 0$ and $\tilde \xi^\e=\ems(\tilde h^\e-h^0(\tilde t^{\e}, \tilde w^{\e},\tilde s^{\e}))$.

The set $\{(\tilde t^{\e}, \tilde w^{\e},\tilde s^{\e})~|~\e\in(0,\e_0)\}$ is bounded, $E_2$ is continuous, $\phi^\e$ and its first and second order derivatives in $w$ and $s$ are $0$ at $(t^\e,w^\e,s^\e,h^\e)$. This, with the definition of $E_1$ and $E_2$ in~\eqref{correc1} and \eqref{correc2}, yields
\begin{align*}
&\frac{\d}{T-t^\e}+\frac{V^0_{ww} (\tilde t^{\e}, \tilde w^{\e},\tilde s^{\e})}{2} \bigg|\sum_{j=1}^{d} \tilde \xi^\e_j \tilde s^\e_j \sigma^j\bigg|^2\leq C(1+|\tilde \xi^{\e}|^2)^{\frac{1}{2}+\frac{1}{m}}-C_\e |\tilde\xi^\e|^m.
\end{align*}
Thus, using that $\sigma\sigma^\top$ is positive definite, that $s$ is bounded away from $0$ on $\{(\tilde t^{\e}, \tilde w^{\e},\tilde s^{\e})~|~\e\in(0,\e_0)\}$ and that $V^0_{ww}$ is negative and bounded from above by $-C<0$ on $\{(\tilde t^{\e}, \tilde w^{\e},\tilde s^{\e})~|~\e\in(0,\e_0)\}$, we obtain
		$$\frac{\delta}{T-t^\e}-C| \tilde \xi^\e|^2  \leq C(1+|\tilde \xi^{\e}|^2)^{\frac{1}{2}+\frac{1}{m}}-C_\e |\tilde\xi^\e|^m.$$
Now, $m>2$ implies that both $ \tilde \xi^\e$ and $\frac{\delta}{T-t^\e}$ are bounded. This is a contradiction with $t^\e\to T$ and the result is proved.
\end{proof}
\begin{proof}[Proof of Theorem~\ref{thm:expansion}]
A combination of Propositions~\ref{prop:supersol},~\ref{prop:subsol} and~\ref{prop:term} allows us to claim that $u^*$ and $u_*$ are respectively viscosity subsolution and supersolution of the second corrector equation \eqref{2nd-corrector-eq} with $0$ final condition. They also satisfy $u^*\geq u_*$ due to their definition as limsup and liminf. 
Thanks to the Assumption~\ref{assum:2ndcorrect} (ii) we also have $u^*\leq u_*$. Denote $u=u^*=u_*$ which is the unique viscosity solution of the  \eqref{2nd-corrector-eq}.
We now have the following inequalities
\begin{align*}
\liminf_{\e\downarrow 0} \frac{ V^0(t,w,s)-V^\e(t,w,s,h^0(t,w,s))}{\e^{2m^*}} &\geqslant\liminf_{\e\downarrow 0} \frac{ V^0(t,w,s)-V^{\e,*}(t,w,s,h^0(t,w,s))}{\e^{2m^*}} \\
&\geqslant  u_*(t,w,s)= u(t,w,s)= u^*(t,w,s) \\
&\geqslant \limsup_{\e\downarrow 0} \frac{ V^0(t,w,s)-V^\e_*(t,w,s,h^0(t,w,s))}{\e^{2m^*}}\\
& \geqslant\limsup_{\e\downarrow 0} \frac{ V^0(t,w,s)-V^\e(t,w,s,h^0(t,w,s))}{\e^{2m^*}} .
\end{align*}
The reverse inequality between the supremum and infimum limits being trivial we have that 
$$\lim_{\e\downarrow 0} \frac{ V^0(t,w,s)-V^\e(t,w,s,h^0(t,w,s))}{\e^{2m^*}} =u(t,w,s).$$
\end{proof}

\section{Appendix}\label{s.appendix}

The Appendix is dedicated to the proof of Proposition~\ref{prop:control-deviation}.
First,  in Section~\ref{s.def.candidate}, we define the strategies that we use to obtain the bound \eqref{eq:control-deviation-bis}. Then in Section~\ref{sec:study.generator}, we study the drift of the process $\Psi^{\e,t,w,s,h}$ defined in \eqref{eq:def.Psi}. In Section~\ref{s.local.boundedness}, we bound the renormalized loss of utility due to price impact; see Lemmas~\ref{lem:dev-control},~\ref{lem:bound.T.Re},~\ref{lem:contcons} and~\ref{lem:moment.W0}. Finally, we provide the proofs of Proposition~\ref{prop:control-deviation} and of Lemma~\ref{lem:probexit} in Section~\ref{s:proof.prop.lem}.

\subsection{Candidate Asymptotically Optimal Strategies}\label{s.def.candidate}
Let $0<\e\leq 1$. Consider the following function (note that it is of the form of the functions studied in Section~\ref{s.remainder} and that it corresponds to the candidate value function expansion obtained in Theorem~\ref{th:example})
\begin{align}
	\label{eq:def.psie}
	&\psi^\e(t,w,s,h):=V^0(t,w,s)-\e^{2m^*}\left(u (t,w,s)+\e^{2m^*}\varpi\left(t,w,s,\frac{h-h^0(t,w,s)}{\e^{m^*}}\right)\right)\\
	&=g(t)U(w)-(w\e)^{2m^*} \lambda w^{1-R} \bar g(t)-(w\e)^{4m^*}\lambda g(t)w^{1-R}\tilde \varpi\left(\frac{\s}{(w\e)^{m^*}} \left(\frac{h\times s}{w}-\pi\right)\right).\notag
\end{align}
Also denote the set of admissible states 
\begin{align}
\label{eq:admissible.states}
	\cA:=\bigg\{(t,w,s,h)\in\cD\times \R^d:h_i> 0\mbox{ for all } 1\leqslant i\leqslant d\ \mbox{ and }\sum_{i=1}^{d} \frac{s_ih_i}{w}< 1 \bigg\}.
\end{align}
We will need the following property of $\pa_w\psi^\e$.

\begin{lemma}\label{lem:varpsie}
There exists $c_W>0$ such that if $(t,w,s,h)\in\cA$ we have that 
\begin{align}\label{eq:boundedpawpsi}
1+c_W(w\e)^{\frac{1}{m}}\geq \frac{\pa_w \psi^\e(t,w,s,h)+\ets \pa_w u(t,w,s)}{g(t) w^{-R}}\geq 1-c_W(w\e)^{\frac{1}{m}}\notag\\
\mbox{ and }1+c_W((w\e)^{2m^*}+(w\e)^{\frac{1}{m}})\geq \frac{\pa_w \psi^\e(t,w,s,h)}{g(t) w^{-R}}\geq 1-c_W((w\e)^{2m^*}+(w\e)^{\frac{1}{m}}),
\end{align}
for all $(t,w,s,h)\in\cD\times \R^d$.
\end{lemma}
\begin{proof}
By differentiation we obtain 
\begin{align*}
&\pa_w \psi^\e(t,w,s,h)\\
&=g(t)w^{-R}-(1-R+2m^*)\bar g(t)\lambda w^{-R} (w\e)^{2m^*}\\
&-(1-R+4m^*)(w\e)^{4m^*}\lambda g(t)w^{-R}\tilde \varpi\left(\frac{\s}{(w\e)^{m^*}} \left(\frac{h\times s}{w}-\pi\right)\right)\\
&+(w\e)^{3m^*}\lambda g(t)w^{-R}\left(\s\left((1+m^*)\frac{h\times s}{w}-m^*\pi\right)\right)\cdot\tilde \varpi_x\left(\frac{\s}{(w\e)^{m^*}} \left(\frac{h\times s}{w}-\pi\right)\right).
\end{align*}
Note that the continuity of the second derivative $\tilde \varpi_{xx}$ and the fact that $\tilde \varpi(0)=0=\tilde \varpi_x(0)$ imply that for $|\xi|\leq 1$, $\tilde \varpi(\xi)\leq C|\xi|^2$ and $|\tilde \varpi_x(\xi)| \leq C|\xi|$. Combined with the bounds in Lemma 4.2 these inequalities yield
\begin{align}
\label{eq:lower.bound.varpi}
	\tilde \varpi(\xi) \leq C|\xi|^{1+\frac{2}{m}} \mbox{ and } |\tilde \varpi_x(\xi)| \leq C|\xi|^{\frac{2}{m}}\ \mbox{for}\ \xi\in\R^d.
\end{align}
Now, considering the boundedness from above and away from zero of $g, \bar g$, and the fact that $\frac{s^ih^i}{w}$ is uniformly bounded on $\cA$ for all $1\leq i\leq d$, we obtain that 
\begin{align*}
|(1-R+2m^*)\frac{\bar g(t)}{g(t)}(w\e)^{2m^*}|&\leq c_W(w\e)^{2m^*},\\
(1-R+4m^*)(w\e)^{4m^*}\lambda \tilde \varpi\left(\frac{\s}{(w\e)^{m^*}} \left(\frac{h\times s}{w}-\pi\right)\right)&\leq \frac{1}{2} c_W (w\e)^{\frac{1}{m}},\\
(w\e)^{3m^*}\lambda \left(\s\left((1+m^*)\frac{h\times s}{w}-m^*\pi\right)\right)\cdot\tilde \varpi_x\left(\frac{\s}{(w\e)^{m^*}} \left(\frac{h\times s}{w}-\pi\right)\right)&\\
\leq  \frac{1}{2}c_W (w\e)^{3m^*-(\frac{2}{m})m^*}&=\frac{1}{2}c_W (w\e)^{\frac{1}{m}},
\end{align*}
for some $c_W>0$.
\end{proof}
Define the feedback control functions
\begin{align}
\label{eq:def.ce.thetae}
c^0(t,&w,s):=-\tilde U'(V^0_w (t,w,s))=g(t)^{-\frac{1}{R}}w,\\
\theta^\e_j(t,w,&s,h):=\e^{-1}\Phi_{x_j}\left(s,\frac{-\e^{3m^*}}{V^0_w(t,w,s)}\varpi_\xi\left(t,w,s,\frac{h-h^0(t,w,s)}{\varepsilon^{m^*}}\right)\right)\label{eq:def.thetae}\\
=&\ems V_w^0\left(t,w,s\right)^{1-m}\Phi_{x_j}\left(s,-\varpi_\xi\left(t,w,s,\frac{h-h^0(t,w,s)}{\varepsilon^{m^*}}\right)\right)\notag\\
=&-(w\e)^{-m^*}\sum_{i=1}^d  \left(\frac{ w }{\kappa^{m-1}s^j}\left|\tilde \varpi_{x_i}\left(X(t,w,s,h)\right)\right|^{m-2}\tilde \varpi_{x_i}\left(X(t,w,s,h)\right)\right)(\s^{-1})^{i,j}.\notag
\end{align}
Note that in this example, the function $\Phi_x$ defined in \eqref{eq:Phix.example} is odd and the functions $u$ and $\varpi$ are the solutions of the corrector equations \eqref{1st-corrector-eq} and \eqref{2nd-corrector-eq} whose properties are listed in Lemma~\ref{lem:growth.varpi.tilde}. We fix an initial condition $(t,w,s,h)\in \cD\times\R^d$ and consider $\e\in(0,(T-t)^{1/2m^*})$. We denote by $ W^{\e,t,w,s,h}$ and $ H^{\e,t,w,s,h}$ the state variables controlled with the above $c^0$ and $\theta^\e$ starting at $(t,w,s,h)$ up to the stopping time \eqref{eq:def.tau.e}. Additionally, we define the rescaled portfolio weights displacement 
\begin{align}
&X^{\e,t,w,s,h}_u = \s\left(\frac{\frac{ H^{\e,t,w,s,h,1}_uS^1_u}{ W^{\e,t,w,s,h}_u}-\pi^1}{(\e  W^{\e,,t,w,s}_u)^{m^*}},\ldots, \frac{\frac{ H^{\e,t,w,s,h,d}_uS^d_u}{ W^{\e,t,w,s,h}_u}-\pi^d}{(\e  W^{\e,t,w,s,h}_u)^{m^*}}\right),\label{eq:rescaled.ptf.weight}\\
&\tau^{\e,t,w,s,h} = \inf\bigg\{u\in[t,T]\bigg|| W^{\e,t,w,s,h}_u-W^0_u|\geqslant \frac{\pi^*}{2}W^0_u\mbox{ or }W^0_u\leqslant \es\mbox{ or }\e W^{0}\geqslant\frac{2}{2+\pi^*} \Big(\frac{1}{4c_{W}}\wedge 1\Big)^{\frac{1}{2m^*}}\notag\\
&\quad\mbox{ or }( W^{\e,t,w,s,h}_u)^{\frac{(m+2)m^*}{m}}\left(1+ \tilde \varpi\left(X^{\e,t,w,s,h}_u\right)\right) \geqslant  \bigg(\frac{(\lambda_{min}\pi^*)^2}{16C_{\tilde \varpi}^{\frac{2m}{m+2}}d^2 \ets}\bigg)^{\frac{m+2}{2m}}\bigg\}\wedge(T-\ets),\label{eq:def.tau.e}
\end{align}
where $\pi^*$ and $C_{\tilde \varpi}$ are given by $\pi^*:=\inf_{1\leqslant i\leqslant d}\pi_i \wedge (1-\sum_{i=1}^d \pi_i)>0$ and
$C_{\tilde \varpi}:=\sup_{x}\frac{|x|^{1+2/m}}{1+\tilde \varpi(x)}<\infty$ (cf. \eqref{eq:lower.bound.varpi}) and $c_{W}$ is the constant defined in Lemma \eqref{lem:varpsie}.
This complicated form of $\tau^{\e,t,w,s,h}$ is due to two reasons. First, we are only able to express the strong mean reversion to the frictionless position by applying It\^{o}'s formula to $\tilde \varpi(X^{\e,t,w,s,h}_u)$(see \eqref{eq:dynamics.varpi}). Indeed, application of It\^{o}'s formula to quantities such as  
\begin{equation*}
\sum_{i=1}^d  \bigg|\frac{H^{\e,t,w,s,h,i}_rS^i_r}{W^{\e,t,w,s,h}_r}-\pi^i\bigg|^2\mbox{ or }\bigg|\frac{H^{\e,t,w,s,h,i}_rS^i_r}{W^{\e,t,w,s,h}_r}-\pi^i\bigg|^2 \mbox{ for some }i
\end{equation*}
do not provide any expression that could allow us to claim that these quantities are small for $\e>0$ small. 
The second reason is that if we do not multiply $(1+ \tilde \varpi(X^{\e,t,w,s,h}_u))$ by $( W^{\e,t,w,s,h}_u)^{\frac{(m+2)m^*}{m}}$ we will have to study the hitting time of a process to a random barrier. It appears that $( W^{\e,t,w,s,h}_u)^{\frac{(m+2)m^*}{m}}(1+ \tilde \varpi(X^{\e,t,w,s,h}_u))$ is the simplest expression allowing us to define the liquidation time as the hitting time of a constant barrier by a mean reverting process. 

Now, we consider an investor following the control $c^0$ and $\theta^\e$ as given on $\llbracket t, \tau^{\e,t,w,s,h}\rrbracket$ and who liquidates her invested position on $ \llbracket \tau^{\e,t,w,s,h}, \tau^{\e,t,w,s,h}+\eeta\rrbracket$, and consumes the \emph{remaining cash}\footnote{This ensures the strict positivity of the wealth until final time and that the lump sum consumption at the final time is strictly positive.} at rate $c^0(t,W^{\e,t,w,s,h}_t, S_t)$ on $\llbracket \tau^{\e,t,w,s,h}, T\rrbracket$. 
The controls for $W^{\e,t,w,s,h}$ and $H^{\e,t,w,s,h}$ are 
\begin{align}\label{eq:trading.rate.cand}
	\dot{H}^{\e,t,w,s,h}_r =\begin{cases} 
	\theta^\e\left(r, W^{\e,t,w,s,h}_r, S_r, H^{\e,t,w,s,h}_r\right) &\text{on}\ \llbracket t, \tau^{\e,t,w,s,h}\rrbracket\\
	-\emeta{H}^{\e,t,w,s,h}_{\tau^{\e,t,w,s,h}}  &\text{on}\ \llbracket \tau^{\e,t,w,s,h}, \tau^{\e,t,w,s,h}+\eeta\rrbracket\\
		0 &\text{on}\ \llbracket \tau^{\e,t,w,s,h}+\eeta,T\rrbracket,
		\end{cases}
\end{align}
with the consumption process
\begin{align}
\label{eq:consumption.cand}C^{\e,t,w,s,h}_r =\begin{cases} 
	c^0\left(r, W^{\e,t,w,s,h}_r, S_r\right) &\text{on}\ \llbracket t, \tau^{\e,t,w,s,h}\rrbracket\\
	C^{\e,t,w,s,h}_r ,&\text{on}\ \llbracket \tau^{\e,t,w,s,h}, \tau^{\e,t,w,s,h}+\eeta\rrbracket\\
		c^0\left(r, W^{\e,t,w,s,h}_r, S_r\right) &\text{on}\ \llbracket \tau^{\e,t,w,s,h}+\eeta,T\rrbracket,
		\end{cases}
\end{align}
where the choice of the consumption on $\llbracket \tau^{\e,t,w,s,h}, \tauet+\ets\rrbracket$ is kept so that
\begin{align*}
	\e W^{\e,t,w,s,h}_{r}\leq \frac{1}{(4c_W)^{\frac{1}{2m^*}}}\wedge 1\wedge 2\e W^0_r
\end{align*} 
on this interval. It is also chosen large enough so that if
\begin{align}
\label{cond:exit}
\tau^{\e,t,w,s,h}<T-\eeta\mbox{ then }W^{\e,t,w,s,h}_{\tau^{\e,t,w,s,h}+\ets}\leq 1\wedge\frac{1}{\e(4c_W)^{\frac{1}{2m^*}}},
\end{align}
and small enough so that
\begin{align*}
	W^{\e,t,w,s,h}_{r}\geq \frac{\pi^*}{2} W^0_{\tauet}\mbox{ for }r\in\llbracket \tau^{\e,t,w,s,h}, \tauet+\ets\rrbracket.
\end{align*}
The first of these last two inequalities implies in particular by Lemma~\ref{lem:varpsie} that along admissible portfolios it holds
\begin{align}
\label{eq:bound.along.strat}
	\frac{1}{2}g(t)w^{-R}\leq \pa_w \psi^\e(t,w,s,h)\leq \frac{3}{2}g(t)w^{-R}.
\end{align}
On $\llbracket\tau^{\e,t,w,s,h}+\ets, T\rrbracket$, $W^{\e,t,w,s,h}$ satisfies the SDE $dW^{\e,t,w,s,h}_z=(r-g(z)^{-\frac{1}{R}})W^{\e,t,w,s,h}_zdz$. Its supremum has therefore moments of all positive and negative orders on this interval. We have additionally
\begin{align*}
\frac{W^0_r}{C}\leq W^{\e,t,w,s,h}_r&\leq C W^0_r\mbox{ on }\llbracket t,\tau^{\e,t,w,s,h}\rrbracket.
\end{align*}
Finally, note that this strategy is indeed admissible as proven in Lemma~\ref{lem:probexit}. 

Define also on $[t,T]$ the processes 
\begin{align}
\Psi^{\e,t,w,s,h}_r &:= \psi^\e(r, W^{\e,t,w,s,h}_r,S_r, H^{\e,t,w,s,h}_r)\label{eq:def.Psi},\\
\varPi^{\e,t,w,s,h}_{\xi,r} &:=\varpi_\xi\left(r, W^{\e,t,w,s,h}_r,S_r,\frac{H^{\e,t,w,s,h}_r-h^0( r,W^{\e,t,w,s,h}_r,S_r)}{\es}\right)\notag\\
&=g(r) (W^{\e,t,w,s,h}_r)^{-R+3m^*}\Big(s\times \s\tilde \varpi_x\left(X^{\e,t,w,s,h}_r\right)\Big).\label{eq:def.varPi}
\end{align}
The following lemma provides a bound on the probability of stopping stricly before $T-\eeta$ when using the controls we just defined. Its proof will be given in Section~\ref{s:proof.prop.lem} after the study of the generator of $\Psi^{\e,t,w,s,h}$ in Section~\ref{sec:study.generator} and the necessary auxiliary Lemmas stated and proved in Section~\ref{s.local.boundedness}.
\begin{lemma}\label{lem:probexit}
	The control defined above is admissible and there exists $C\in \F_{comp}$ such that for all $(t,w,s)\in \cD$, there exists $\d_1>0$ such that for all $h\in \R^d$ satisfying $|\frac{h_is_i}{w}-\pi_i|{\leq} \delta_1$ we have   
	\begin{equation*}
		\P(\tau^{\e,t,w,s,h}< T-\ets){\leq} \ets C(t,w,s).
	\end{equation*}
\end{lemma}

\subsection{Semi-Martingale Decomposition of \eqref{eq:def.Psi}}
\label{sec:study.generator}
Define the following process,
\begin{align*}
	&\tilde G^{\e,t,w,s,h}_r = \ets\left(\left(\partial_w\psi^\e\left(r,W^{\e,t,w,s,h}_r,S_r,H^{\e,t,w,s,h}_r\right)\right)^{-m}-\left(V^0_w(r,W^{\e,t,w,s,h}_r,S_r)\right)^{-m}\right)\\
	&\quad\times\partial_w\psi^\e\left(r,W^{\e,t,w,s,h}_r,S_r,H^{\e,t,w,s,h}_r\right)\Phi_x\left(S_r,-\varPi^{\e,t,w,s,h}_{\xi,r}\right)f\left(S_r,\Phi_x\left(S_r,-\varPi^{\e,t,w,s,h}_{\xi,r}\right)\right)\\
	&+\ets\left(\left(\partial_w\psi^\e\left(r,W^{\e,t,w,s,h}_r,S_r,H^{\e,t,w,s,h}_r\right)\right)^{1-m}-\left(V^0_w(r,W^{\e,t,w,s,h}_r,S_r)\right)^{1-m}\right)\\
	&\quad\times\varPi^{\e,t,w,s,h}_{\xi,r}\cdot\Phi_x\left(S_r,-\varPi^{\e,t,w,s,h}_{\xi,r}\right).
\end{align*}
On $\llbracket \tau^\e,\tau^\e+\eeta\rrbracket$, $\tilde G^{\e,t,w,s,h}$ is given by
\begin{align*}
	&\tilde G^{\e,t,w,s,h}_r = \ets\partial_w\psi^\e\left(r,W^{\e,t,w,s,h}_r,S_r,H^{\e,t,w,s,h}_r\right)^{1-m}\Phi_x\left(S_r,-\varPi^{\e,t,w,s,h}_{\xi,r}\right)\\
	&\quad\quad f\left(S_r,\Phi_x\left(S_r,-\varPi^{\e,t,w,s,h}_{\xi,r}\right)\right)\\
	&-\e^{\frac{m^*(m-2)}{m-1}}\partial_w\psi^\e\left(r,W^{\e,t,w,s,h}_r,S_r,H^{\e,t,w,s,h}_r\right)H^{\e,t,w,s,h}_{\tau^\e}\cdot f\left(S_r,{H^{\e,t,w,s,h}_{\tau^\e}}\right)\notag\\
	&+\ets\left(\partial_w\psi^\e\left(r,W^{\e,t,w,s,h}_r,S_r,H^{\e,t,w,s,h}_r\right)\right)^{1-m}\varPi^{\e,t,w,s,h}_{\xi,r}\cdot\Phi_x\left(s,-\varPi^{\e,t,w,s,h}_{\xi,r}\right)\notag\\
		&+\es \varPi^{\e,t,w,s,h}_{\xi,r}H^{\e,t,w,s,h}_{\tau^\e},
\end{align*}
and on $\llbracket \tau^\e+\ets,T\rrbracket$ we define
\begin{align*}
	\tilde G^{\e,t,w,s,h}_r &= \ets\partial_w\psi^\e\left(r,W^{\e,t,w,s,h}_r,S_r,H^{\e,t,w,s,h}_r\right)^{1-m}\Phi_x\left(S_r,-\varPi^{\e,t,w,s,h}_{\xi,r}\right)\\
	&\qquad\qquad\qquad\qquad\cdot\bigg[f\left(S_r,\Phi_x\left(S_r,-\varPi^{\e,t,w,s,h}_{\xi,r}\right)\right)+\varPi^{\e,t,w,s,h}_{\xi,r}\bigg]\\
	&=- \ets\partial_w\psi^\e\left(r,W^{\e,t,w,s,h}_r,S_r,H^{\e,t,w,s,h}_r\right)^{1-m}\Phi\left(S_r,-\varPi^{\e,t,w,s,h}_{\xi,r}\right)\leq 0.
\end{align*}
$\tilde G^{\e,t,w,s,h}$ takes into account the difference in drift of $\Psi^{\e,t,w,s,h}$ if it had been conrolled by $c^0$ (same control for the consumption) and the optimizer of the Hamiltonians in $h$ in the functional $\cG^\e(\psi^\e)$. This allows us to relate the remainder estimate of Section~\ref{s.remainder} (Equation~\eqref{eq:psi.e.expnsion}) to the drift of $\Psi^{\e,t,w,s,h}$ below (see Equation~\eqref{eq:Ge.psie}). We now give the following estimates for $\tilde G^{\e,t,w,s,h}$.
\begin{lemma}\label{lem:conttidleG}
It holds,
\begin{align*}
|\tilde G^{\e,t,w,s,h}_r| &\leq C(\e W^{\e,t,w,s,h}_r)^{2m^*}(W^{\e,t,w,s,h}_r)^{1-R}\mbox{ on }\llbracket t,\tau^\e\rrbracket,\\
|\tilde G^{\e,t,w,s,h}_r|&\leq C \Big((W^{\e,t,w,s,h}_r)^{1-R}+ (W^{\e,t,w,s,h}_r)^{-R}(W^{0}_{\tauet})^{\frac{m}{m-1}}\\
&\qquad\qquad\qquad+ (W^{\e,t,w,s,h}_r)^{\frac{1}{m}-R}(W^{0}_{\tauet})\Big) \mbox{ on }\llbracket \tau^\e,\tau^\e+\e^{2m^*}\rrbracket,\\
|\tilde G^{\e,t,w,s,h}_r|&\leq C(W^{\e,t,w,s,h}_r)^{1-R}\leq C \mbox{ on }\llbracket \tau^\e+\e^{2m^*},T\rrbracket.
\end{align*}
\end{lemma}
\begin{proof}
First note that by Lemma~\ref{lem:varpsie}, and the fact that $g$ is bounded away from $0$ on $[0,T]$ we have for the chosen control $H^{\e,t,w,s,h}$ and $C^{\e,t,w,s,h}$ the following inequality on $\llbracket t,\tauet\rrbracket$
\begin{align*}
&\left|\partial_w\psi^\e\left(r,W^{\e,t,w,s,h}_r,S_r,H^{\e,t,w,s,h}_r\right)-\left(V^0_w(r,W^{\e,t,w,s,h}_r,S_r)\right)\right|\\
&\qquad\qquad\qquad\leq C (W^{\e,t,w,s,h}_r)^{-R} ((W^{\e,t,w,s,h}_r\e)^{2m^*}+(W^{\e,t,w,s,h}_r\e)^{\frac{1}{m}}),
\end{align*}
which yields, by Taylor's expansion and \eqref{eq:bound.along.strat}
\begin{align*}
	&\Big|\big(\partial_w\psi^\e\big)^{-m}-\big(V^0_w\big)^{-m}\Big|\big(r,W^{\e,t,w,s,h}_r,S_r,H^{\e,t,w,s,h}_r\big)\\
	&\qquad\qquad\qquad\leq C\big((W^{\e,t,w,s,h}_r)^{-R}\big)^{-1-m}|\partial_w\psi^\e-V^0_w|\\
	&\qquad\qquad\qquad\leq C\big(W^{\e,t,w,s,h}_r\big)^{Rm}((\e W^{\e,t,w,s,h}_r)^{2m^*}+(\e W^{\e,t,w,s,h}_r)^{\frac{1}{m}}),\\
	&\Big|\big(\partial_w\psi^\e\big)^{1-m}-\big(V^0_w\big)^{1-m}\Big|\big(r,W^{\e,t,w,s,h}_r,S_r,H^{\e,t,w,s,h}_r\big)\\
	&\qquad\qquad\qquad\leq C\big(W^{\e,t,w,s,h}_r\big)^{R(m-1)}((\e W^{\e,t,w,s,h}_r)^{2m^*}+(\e W^{\e,t,w,s,h}_r)^{\frac{1}{m}}).
\end{align*}
Seconds, by definition of $f$ in \eqref{eq:def.f.example} and $\Phi_x$ in \eqref{eq:Phix.example}, it holds for some constant $C>0$,
\begin{align*}
	\Phi\left(S_r,-\varPi^{\e,t,w,s,h}_{\xi,r}\right) \leq& C\Bigg|\frac{\varPi^{\e,t,w,s,h}_{\xi,r}}{S_r}\Bigg|^{m},\\
	0\leq\varPi^{\e,t,w,s,h}_{\xi,r}\cdot\Phi_x\left(S_r,-\varPi^{\e,t,w,s,h}_{\xi,r}\right) \leq& C\Bigg|\frac{\varPi^{\e,t,w,s,h}_{\xi,r}}{S_r}\Bigg|^{m},\\
	0\leq f\left(S_r,\Phi_x\left(S_r,-\varPi^{\e,t,w,s,h}_{\xi,r}\right)\right)\cdot \Phi_x\left(S_r,-\varPi^{\e,t,w,s,h}_{\xi,r}\right) \leq&C\Bigg|\frac{\varPi^{\e,t,w,s,h}_{\xi,r}}{S_r}\Bigg|^{m}.
\end{align*}
Finally, the definition of $\varPi^{\e,t,w,s,h}$ in \eqref{eq:def.varPi}, the estimate~\eqref{eq:lower.bound.varpi} on $\tilde \varpi_x$ and the fact that $\frac{H^{\e,t,w,s,h,i} S^i}{W^{\e,t,w,s,h}}$ is uniformly bounded for an admissible strategy and $1\leq i\leq d$, provides the bound
\begin{align*}
	\Bigg|\frac{\varPi^{\e,t,w,s,h}_{\xi,r}}{S_r}\Bigg| \leq  C(W^{\e,t,w,s,h}_r)^{-R+3m^*} \frac{1}{|  (\e W^{\e,t,w,s,h}_r)^{m^*}|^{\frac{2}{m}}}\leq C\e^{-\frac{2m^*}{m}} (W^{\e,t,w,s,h}_r)^{-R+\frac{1}{m}}.
\end{align*}
Now the result follows from combining the definition of $\tilde G^{\e,t,w,s,h}$ on the three stochastic intervals, the estimates we just stated, the fact that $g$ is bounded and bounded away from $0$ (see Remark~\ref{rem:g}), the definition of $f$, $\varPi^{\e,t,w,s,h}_\xi$, and $\tauet$ (in \eqref{eq:def.f.example}, \eqref{eq:def.varPi} and \eqref{eq:def.tau.e}), the fact that $\frac{H^{\e,t,w,s,h,i} S^i}{W^{\e,t,w,s,h}}$ is uniformly bounded for an admissible strategy and the inequalities \eqref{eq:lower.bound.varpi} and \eqref{eq:bound.along.strat}.§
\end{proof}

Denote $\mu^{\psi^\e}$  the drift of the diffusion $\Psi^{\e,t,w,s,h}_u=\psi^\e\left(u,  W^{\e,t,w,s,h}_u, S_u,  H^{\e,t,w,s,h}_u\right)$.
Note that in \eqref{eq:tildemuformu} we computed the drift of $\psi^\e$ applied to processes controlled by a different strategy. In \eqref{eq:tildemuformu} the controls appear in three different lines. Thus, 
\begin{align*}
	\mu^{\psi^\e}_z=&\tilde \mu^{\psi^\e}_z+\partial_w \psi^\e(z,W^{\e,t,w,s,h}_z,S_z,H^{\e,t,w,s,h}_z)\Big[c^\e(z,W^{\e,t,w,s,h}_z,S_z,H^{\e,t,w,s,h}_z)-C^{\e,t,w,s,h}_z \Big]\\&+\tilde G^{\e,t,w,s,h}_z.
\end{align*}
Here, the second term corresponds to the difference in consumption between the control used in \eqref{eq:tildemuformu} and \eqref{eq:consumption.cand}, the last term corresponds to the difference in strategy $\theta$ between \eqref{eq:tildemuformu} and \eqref{eq:trading.rate.cand}. Thus, using \eqref{eq:Ge.mu.psie} we obtain
\begin{align}
	d&\Psi^{\e,t,w,s,h}_z=-\left[\cG^\e(\psi^\e)(z,W^{\e,t,w,s,h}_z,S_z,H^{\e,t,w,s,h}_z)+U\left(c^\e\left(z,W^{\e,t,w,s,h}_z,S_z,H^{\e,t,w,s,h}_z\right)\right)\right]dz\notag\\
	&+\partial_w \psi^\e(z,W^{\e,t,w,s,h}_z,S_z,H^{\e,t,w,s,h}_z)\Big[c^\e(z,W^{\e,t,w,s,h}_z,S_z,H^{\e,t,w,s,h}_z)-C^{\e,t,w,s,h}_z \Big]dz\notag\\
	&+\tilde G^{\e,t,w,s,h}_zdz+dM_z, \mbox{ on } [t, T],\label{eq:Ge.psie}
\end{align}
where $M$ is a martingale. The first line corresponds to the generator computed in Section~\ref{s.remainder}, the second line and the  process $\tilde G^{\e,t,w,s,h}$ are the contributions due to the fact that consumption, respectively the candidate strategy $\theta^\e$, are not the maximizers of the Hamiltonian in \eqref{eq:PDE-veps}. 
 
\subsection{Local Boundedness of the Renormalized Loss of Utility}
\label{s.local.boundedness}
We also need to define $\tilde W^{0,t,w,s,h}$, the frictionless wealth process started at $(t,w,s)$ when the investor does not consume ($c\equiv0$ in the wealth dynamics) and follows $\theta^\e$. Let $\tilde \cL$ be the generator of the diffusion $(\cdot,\tilde W^{0,t,w,s,h}_\cdot,S_\cdot, \ems(h- h^0(\cdot, \tilde W^{0,t,w,s,h}_\cdot,S_\cdot)))$. We now provide the main decomposition for the renormalized loss of utility associated with the controls $H^{\e,t,w,s,h}$ and $C^{\e,t,w,s,h}$.
Before we proceed, let us define for a function $\psi^\e$ of the form defined in \eqref{eq:def.psie} and $\e>0$ the remainder functional 
\begin{align}
\label{eq:def.cR.epsi}
\cR^\e(\psi^\e)(t,w,s,h):=(m-1)\Phi(s,\varpi_\xi)(V^0_w)^{-m}\Big(\pa_w\psi^\e&-V^0_w\Big)\\
&-\tilde U'(V_w^0)\frac{\pa_w \psi^\e-V^0_w+\ets u_w}{\ets}.\notag
\end{align}

\begin{lemma}
\label{lem:dev-control}
For all $(t,w,s)\in \cD$, and $h\in\R^d$ it holds
\begin{align}\label{eq:mainbound1}
& \frac{V^0(t,w,s)-V^\e(t,w,s,h)}{\e^{2m*}}- u (t,w,s)\\
&{\leqslant}\E\Big[\int_{t}^{T}T(r,W^{\e,t,w,s,h}_r,S_r, u_w,u_{ww},u_{ws};H^{\e,t,w,s,h}_r)\notag\\
&\qquad\qquad\qquad\qquad-T(r,W^{\e,t,w,s,h}_r,S_r, u_w,u_{ww}, u_{ws};h^0(r,W^{\e,t,w,s,h}_r,S_r))dr\Big]\notag\\
&+\ets\bigg(\E\bigg[ \varpi\bigg(T,\tilde W^0_T,S_T, \frac{h-h^0(T,\tilde W^0_T,S_T)}{\es}\bigg)\bigg]\notag\\
&\qquad\qquad-\E\Big[\varpi\Big(T,W^{\e,t,w,s,h}_T,S_T,-\ems h^0(T, W^{\e,t,w,s,h}_T,S_T)\Big)\Big]\bigg)\notag\\
&+\E\bigg[\int_{t}^{T} \Big(\cR^\e(r,W^{\e,t,w,s,h}_r,S_r,H^{\e,t,w,s,h}_r)-\emts\tilde  G^{\e,t,w,s,h}_rdr\Big)\bigg]\notag\\
&-\emts\E\bigg[\int_{t}^{T}\partial_w \psi^\e(r,W^{\e,t,w,s,h}_r,S_r,H^{\e,t,w,s,h}_r)\notag\\
&\qquad\qquad\qquad\qquad\qquad\qquad\qquad\qquad\times\Big(c^\e(r,W^{\e,t,w,s,h}_r,S_r,H^{\e,t,w,s,h}_r)-C^{\e,t,w,s,h}_r \Big)dr\bigg]\notag\\
&+\emts\E\bigg[\int_{t}^{T}\Big(U\Big(c^\e(r,W^{\e,t,w,s,h}_r,S_r,H^{\e,t,w,s,h}_r)\Big)-U(C^{\e,t,w,s,h}_r)\Big)dr\bigg].\notag
\end{align}
\end{lemma}
\begin{proof}
Thanks to \eqref{eq:Ge.psie} and the fact that $H^{\e,t,w,s,h}_T=0$ we have
\begin{align*}
&\psi^\e(t,w,s,h)=\E\left[\Psi^{\e,t,w,s,h}_T\right]+\E\bigg[\int_{t}^{T} \Big(\cG^\e(\psi^\e)(z,W^{\e,t,w,s,h}_z,S_z,H^{\e,t,w,s,h}_z)-\tilde G^{\e,t,w,s,h}_z\\
&\quad-\partial_w \psi^\e(z,W^{\e,t,w,s,h}_z,S_z,H^{\e,t,w,s,h}_z)\Big[c^\e(z,W^{\e,t,w,s,h}_z,S_z,H^{\e,t,w,s,h}_z)-C^{\e,t,w,s,h}_z\Big]\\
&\quad+U(c^\e(z,W^{\e,t,w,s,h}_z,S_z,H^{\e,t,w,s,h}_z))\Big)dz\bigg].
\end{align*} 
By definition of $\psi^\e$ and the boundary conditions of $V^0$ and $u$ (see Equations~\eqref{eq:final.cond.V0} and \eqref{1st-corrector-eq}) it holds
\begin{align*}
&\E\left[\Psi^{\e,t,w,s,h}_T\right] = \E\left[V^0(T,W^{\e,t,w,s,h}_T,S_T)\right]-\ets\E\left[u(T,W^{\e,t,w,s,h}_T,S_T)\right]\\
&\qquad\qquad-\efs\E\left[\varpi\left(T,W^{\e,t,w,s,h}_T,S_T,\frac{H^{\e,t,w,s,h}_T-h^0\big(T,W^{\e,t,w,s,h}_T,S_T\big)}{\es}\right)\right]\\
&\qquad=\E\left[U\big(W^{\e,t,w,s,h}_T\big)\right]-\efs\E\left[\varpi\Big(T,W^{\e,t,w,s,h}_T,S_T,-\ems h^0\big(T,W^{\e,t,w,s,h}_T,S_T\big)\Big)\right].
\end{align*}
{Note that the condition $H^{\e,t,w,s,h}_T=0$ implies that the position in cash at final time is indeed $W^{\e,t,w,s,h}_T$.
Thus, given the admissibility of the strategy and the terminal condition~\eqref{eq:final.cond.Ve} for $V^\e$, we have $\E\left[U\big(W^{\e,t,w,s,h}_T\big)+\int_t^TU(C^{\e,t,w,s,h}_z)dz\right]\leq V^\e(t,w,s,h)$  and obtain}
\begin{align*}
\psi^\e&(t,w,s,h){\leqslant} \E\left[\int_{t}^{T}\left(\cG^\e(\psi^\e)(z,W^{\e,t,w,s,h}_z,S_z,H^{\e,t,w,s,h}_z)-\tilde G^{\e,t,w,s,h}_z\right)dz\right]\\
&+\E\bigg[\int_{t}^{T}\bigg(-\partial_w \psi^\e(z,W^{\e,t,w,s,h}_z,S_z,H^{\e,t,w,s,h}_z)\Big[c^\e(z,W^{\e,t,w,s,h}_z,S_z,H^{\e,t,w,s,h}_z)-C^{\e,t,w,s,h}_z\Big]\\
&+U(c^\e(z,W^{\e,t,w,s,h}_z,S_z,H^{\e,t,w,s,h}_z))-U(C^{\e,t,w,s,h}_z)\bigg)dz\bigg]\\
&+V^\e(t,w,s,h)-\efs\E\left[\varpi\left(T,W^{\e,t,w,s,h}_T,S_T,-\ems h^0\left(T,W^{\e,t,w,s,h}_T,S_T\right)\right)\right],
\end{align*}
which implies 
\begin{align}
&\frac{V^0(t,w,s)-V^\e(t,w,s,h)}{\e^{2m*}}{\leq}  \frac{1}{\ets}\E\left[\int_t^{T} \left(\cG^\e(\psi^\e)(z,W^{\e,t,w,s,h}_z,S_z,H^{\e,t,w,s,h}_z)-\tilde G^{\e,t,w,s,h}_z\right)dr\right]\notag\\
&+u (t,w,s)+\e^{2m^*}\varpi\left(t,w,s,\frac{h-h^0(t,w,s)}{\e^{m^*}}\right)\notag\\
&+\emts\E\bigg[\int_{t}^{T}-\partial_w \psi^\e(z,W^{\e,t,w,s,h}_z,S_z,H^{\e,t,w,s,h}_z)\Big[c^\e(z,W^{\e,t,w,s,h}_z,S_z,H^{\e,t,w,s,h}_z)-C^{\e,t,w,s,h}_z\Big]\notag\\
&+U(c^\e(z,W^{\e,t,w,s,h}_z,S_z,H^{\e,t,w,s,h}_z))-U(C^{\e,t,w,s,h}_z)dz\bigg]\notag\\
&-\ets\E\left[\varpi\left(T,W^{\e,t,w,s,h}_T,S_T,-\ems h^0\left(T,W^{\e,t,w,s,h}_T,S_T\right)\right)\right],\label{eq:cont-dev1}
\end{align}
where the left-hand side is non-negative by construction.

We note that the remainder estimates of Section~\ref{s.remainder} for this choice of $\psi^\e$ give
\begin{align*}
&\emts\cG^\e(\psi^\e)(t,w,s,h)=T(t,w,s, u_w,u_{ww},u_{ws};h)-T(t,w,s, u_w,u_{ww}, u_{ws};h^0(t,w,s))\\
&+\ets \tilde \cL\left(\varpi \right)-\frac{1}{2}\text{Tr}\left(c^{h^0}\varpi_{\xi\xi}\right)+\left((V^0_w)^{1-m}-(\pa_w\psi^\e)^{1-m}\right)\Phi(s,-\varpi_\xi)\\
&+\frac{\tilde U(V^0_w)-\tilde U(\pa_w\psi^\e)-\ets\tilde U'(V^0_w)u_w}{\ets}.
\end{align*}
This follows indeed from \eqref{eq:Ge.mu.psie} and \eqref{eq:psi.e.expnsion}, the fact that for the solutions $u$ and $\varpi$ of the corrector equations \eqref{1st-corrector-eq} and \eqref{2nd-corrector-eq}, it holds 
\begin{equation*}
E_1(t,w,s,\xi,\varpi_\xi,\varpi_{\xi\xi})+E_2(t,w,s,u,u_t, u_w, u_s,u_{ww},u_{ws},u_{ss})=0,
\end{equation*}
and the equation (obtained by direct computation from \eqref{eq:trace.D2.h0} and the definitions \eqref{eq:def.I1}, \eqref{eq:def.I3} and \eqref{eq:def.I4})
\begin{equation*}
	I^{\e,1}+I^{\e,3}+I^{\e,4}=\ets\tilde\cL(\varpi)-\frac{1}{2}\text{Tr}\big(c^{h^0}\varpi_{\xi\xi}\big).
\end{equation*}

Recall that $\Phi(s,\cdot)$ is even for the choice made for $f$ in \eqref{eq:def.f.example}. Due to the convexity of $x\to x^{1-m}$ and $\tilde U$ the following two inequalities hold for all $(t,w,s,h)$,
\begin{align*}
\left((V^0_w)^{1-m}-(\pa_w\psi^\e)^{1-m}\right)\Phi(s,-\varpi_\xi){\leq} (m-1)\Phi(s,\varpi_\xi)(V^0_w)^{-m}\left(\pa_w\psi^\e-V^0_w\right),\\
\frac{\tilde U(V^0_w)-\tilde U(\pa_w\psi^\e)-\ets\tilde U'(V^0_w)u_w}{\ets}{\leq} -\tilde U'(V_w^0)\frac{\pa_w \psi^\e-V^0_w+\ets u_w}{\ets}.
\end{align*}
Hence, by the convexity of $\xi \mapsto \varpi(t,w,s,\xi)$ and the fact that $c^{h^0}$ is non-negative (the trace of the product of symmetric non-negative matrices is non-negative) we obtain 
\begin{align*}
&\emts\cG^\e(\psi^\e)(t,w,s,h){\leq} T(t,w,s, u_w,u_{ww},u_{ws};h)-T(t,w,s, u_w,u_{ww}, u_{ws};h^0(t,w,s))\\
&+\ets\tilde \cL\left(\varpi \right)+(m-1)\Phi(s,\varpi_\xi)(V^0_w)^{-m}\left(\pa_w\psi^\e-V^0_w\right)-\tilde U'(V_w^0)\frac{\pa_w \psi^\e-V^0_w+\ets u_w}{\ets}.
\end{align*}
With the definition of $\cR^\e$ in \eqref{eq:def.cR.epsi}, we obtain 
\begin{align*}
&\frac{1}{\ets}\E\left[\int_t^{T} \cG^\e(\psi^\e)(W^{\e,t,w,s,h}_r,S_r,H^{\e,t,w,s,h}_r)dr\right]\\
&{\leq}\E\left[\int_t^{T} (T(\cdot;H^{\e,t,w,s,h}_r)-T(\cdot;h^0(r,W^{\e,t,w,s,h}_r,S_r))(r,W^{\e,t,w,s,h}_r,S_t, u_w,u_{ww},u_{ws})dr\right]\\
&+\E\left[\ets \varpi\left(T,\tilde W^0_T,S_T,\ems (h-h^0(T,\tilde W^0_T,S_T))\right)\right]-\ets \varpi\left(t,w,s,\ems (h-h^0(t,w,s))\right)\\
&+ \E\left[\int_t^{T}\cR^\e(r,W^\e_r,S_r,H^\e_r)dr\right].
\end{align*}
Combining this inequality with \eqref{eq:cont-dev1} we conclude the proof. 
\end{proof}
The following lemmas allow us to locally bound the renormalized loss of utility.

\begin{lemma}
	\label{lem:bound.T.Re}
		Let $u$ be the function defined in \eqref{eq:def.u.example} and $T$ defined in \eqref{eq:defT}. Then,  for $(t,w,s,h)\in\cA$ (defined in \eqref{eq:admissible.states}) we have
	\begin{align*}
	\bigg|&T(t,w,s, u_w, u_{ww},u_{ws};h)-T(t,w,s, u_w, u_{ww},u_{ws};h^0(t,w,s))+\cR^\e(t,w,s,h)|\\
		& \leqslant C\left(1+w^{1+2m^*-R}+w^{1+\frac{1}{m}-R}\right) \leq C\left(1+w^{k}\right)
			\end{align*}
	holds for some $k>0$.
\end{lemma}
\begin{proof}
Note that thanks to the admissibility of the strategies
we have that $\frac{h\times s}{w}-\frac{h^{0}(t,w,s)\times s}{w}$ is uniformly bounded. 
Thus $|T(t,w,s, u_w, u_{ww},u_{ws};h)|\leq C w^{1-R+2m^*}$ for admissible strategies. 
Similarly, to the proof of Lemma~\ref{lem:conttidleG} we have
\begin{align*}
\left|\Phi\left(s,\varpi_\xi\left(t,w,s,\frac{h-h^0(t,w,s)}{\ems}\right)\right)\right|&\leq C\emts w^{1-Rm}\\
|(V^0_w)^{-m}|&\leq C w^{Rm}\\
|\pa_w\psi^\e-V^0_w|&\leq C((w\e)^{2m^*}+(w\e)^{\frac{1}{m}})w^{-R}\\
\Big|\tilde U'(V_w^0)\frac{\pa_w \psi^\e-V^0_w+\ets u_w}{\ets}\Big|&\leq C \e^{\frac{1}{m}}w^{1-R+\frac{1}{m}}.
\end{align*}
Thus,
\begin{align*}
&|\cR^\e(\psi^\e)(t,w,s,h)|\leq C\left(1+w^{1+2m^*-R}+w^{1+\frac{1}{m}-R}\right).
\end{align*}
\end{proof}

\begin{lemma}\label{lem:contcons}
Define the process 
\begin{align}
K_r:=&-\partial_w \psi^\e(r,W^{\e,t,w,s,h}_r,S_r,H^{\e,t,w,s,h}_r)\Big[c^\e(r,W^{\e,t,w,s,h}_r,S_r,H^{\e,t,w,s,h}_r)-C^{\e,t,w,s,h}_r \Big]\notag\\
&\quad+U(c^\e(r,W^{\e,t,w,s,h}_r,S_r,H^{\e,t,w,s,h}_r))-U(C^{\e,t,w,s,h}_r)\mbox{ for }r\in[t,T]. 
\end{align}
Then, the admissibility of the strategy implies that 
\begin{align*}
|K_r|\leq &C(W^{\e,t,w,s,h}_r\e)^{2m^*}|W^{\e,t,w,s,h}_r|^{1-R}\mbox{ on } \llbracket t, \tau^{\e,t,w,s,h}\rrbracket,\\
|K_r|\leq &C(W^{\e,t,w,s,h}_r)^{1-R}\mbox{ on } \llbracket \tau^{\e,t,w,s,h},T\rrbracket.
\end{align*}

\end{lemma}
\begin{proof}
First, on $\llbracket t,\tau^{\e,t,w,s,h}\rrbracket$ and $\llbracket \tau^{\e,t,w,s,h}+\ets, T\rrbracket$, it holds $$C^{\e,t,w,s,h}_r:=c^0(r,W^{\e,t,w,s,h}_r,S_r)=-\tilde U'(V^0_w(t,w,s)).$$ We also have for $R<1$, $U(C^{\e,t,w,s,h}_r)\geq 0$. 
Thus, using \eqref{eq:bound.along.strat}, the definition of $c^\e$ in \eqref{eq:def.c.e} and of $\tilde U$ in \eqref{eq:def.Utilde} and the fact that $g$ is bounded and bounded away from $0$, the following inequality holds on $\llbracket\tau^\e,T\rrbracket$
\begin{align*}
|K_r|\leq&|\partial_w \psi^\e(r,W^{\e,t,w,s,h}_r,S_r,H^{\e,t,w,s,h}_r)|\Big(|\tilde U'\big(\partial_w\psi^\e(r,W^{\e,t,w,s,h}_r,S_r,H^{\e,t,w,s,h}_r)\big)|\notag\\
&\quad+|\tilde U'(V^0_w(t,w,s))|\Big)+U(c^\e(r,W^{\e,t,w,s,h}_r,S_r,H^{\e,t,w,s,h}_r))\leq C(W^{\e,t,w,s,h}_r)^{1-R},
\end{align*}
for some $C>0$. Now using Lemma~\ref{lem:varpsie} similarly to the proof of Lemma~\ref{lem:conttidleG} we have the inequalities
\begin{align*}
&\Big|\big(\partial_w\psi^\e\big)^{-\frac{1}{R}}-\big(V^0_w\big)^{-\frac{1}{R}}\Big|\big(r,W^{\e,t,w,s,h}_r,S_r,H^{\e,t,w,s,h}_r\big)\\
&\qquad\qquad\qquad\leq C\big(W^{\e,t,w,s,h}_r\big)((\e W^{\e,t,w,s,h}_r)^{2m^*}+(\e W^{\e,t,w,s,h}_r)^{\frac{1}{m}}),\\
&\Big|\big(\partial_w\psi^\e\big)^{-\frac{1-R}{R}}-\big(V^0_w\big)^{-\frac{1-R}{R}}\Big|\big(r,W^{\e,t,w,s,h}_r,S_r,H^{\e,t,w,s,h}_r\big)\\
&\qquad\qquad\qquad\leq C\big(W^{\e,t,w,s,h}_r\big)^{1-R}((\e W^{\e,t,w,s,h}_r)^{2m^*}+(\e W^{\e,t,w,s,h}_r)^{\frac{1}{m}}).
\end{align*}
Then, on the interval $\llbracket t, \tau^{\e,t,w,s,h}\rrbracket$, by definition \eqref{eq:def.tau.e} of $\tauet$ and using \eqref{eq:bound.along.strat} we have that 
\begin{align*}
|K_r|\leq C(W^{\e,t,w,s,h}_r\e)^{2m^*}|W^{\e,t,w,s,h}_r|^{1-R}.
\end{align*}
\end{proof}

The proof of Lemma~\ref{lem:probexit} requires moments existence for $W^{\e}$ for which we need first the following result on $W^0$.
\begin{lemma}
	\label{lem:moment.W0}
	In the Black-Scholes setting, the supremum of $W^0$ over $[t,T]$ has moments of all orders
	\begin{equation*}
		\E_t\bigg[\sup_{u\in[t,T]}\Big(W^{\e,t,w,s,h}_u\Big)^{\eta}\bigg]+ \E_t\bigg[\int_{t}^{T}\Big(W^{\e,t,w,s,h}_u\Big)^{\eta}du\bigg]+\E_t\bigg[\sup_{u\in[t,T]}\Big(W^0_u\Big)^\eta\bigg]=C^{t,\eta}_0 w^\eta <\infty,
	\end{equation*}
	where $\E_t$ denotes the expectation conditional on $\cF_t$ and $\eta>0$.
\end{lemma}
\begin{proof}
The result is a direct consequence of the fact that $W^0$ is a geometric Brownian motion, the fact that by the definition of our strategies we have $\frac{W^0}{C}\leq W^\e\leq C W^0$ for some $C>0$ until $\tauet$, the SDE satisfied by $W^{\e,t,w,s,h}$ on $\llbracket \tauet, \tauet +\ets\rrbracket$ and on  $\llbracket \tauet+\ets, T\rrbracket$, the fact that $H^{\e,t,w,s,h}_r$ is strictly decreasing on $\llbracket \tauet, \tauet +\ets\rrbracket$ and the $\frac{m}{m-1}$ homogeneity of $\theta\mapsto\theta\cdot f(s,\theta)$.
\end{proof}

\subsection{Proofs of Proposition~\ref{prop:control-deviation} and Lemma~\ref{lem:probexit}}
\label{s:proof.prop.lem}
\begin{proof}[Proof of Proposition~\ref{prop:control-deviation}]
Combining the inequality \eqref{eq:mainbound1} of Lemma~\ref{lem:dev-control} with Lemmas~\ref{lem:conttidleG},~\ref{lem:bound.T.Re} and~\ref{lem:contcons} we obtain the inequality 
\begin{align*}
	&\frac{V^0(t,w,s)-V^\e(t,w,s,h)}{\e^{2m*}} - u (t,w,s){\leqslant}\E\Big[\int_{t}^{T}C\left(1+(W^{\e,t,w,s,h}_r)^{k}\right)dr\Big]\\
	&\qquad\qquad+\E\bigg[\int_{t}^{\tau^{\e,t,w,s,h}} C\emts(\e W^{\e,t,w,s,h}_r)^{2m^*}(W^{\e,t,w,s,h}_r)^{1-R}dr\bigg]\notag\\
	&\qquad\qquad+\E\bigg[\int_{\tau^{\e,t,w,s,h}}^{\tau^{\e,t,w,s,h}+\ets}C\emts \Big((W^{\e,t,w,s,h}_r)^{1-R}+ (W^{\e,t,w,s,h}_r)^{-R}(W^{0}_{\tauet})^{\frac{m}{m-1}}\\
	&\qquad\qquad\qquad\qquad\qquad\qquad\qquad\qquad\qquad\qquad\qquad+ (W^{\e,t,w,s,h}_r)^{\frac{1}{m}-R}(W^{0}_{\tauet})\Big)dr\bigg]\\
	&\qquad\qquad+\E\bigg[\int_{\tau^{\e,t,w,s,h}+\ets}^{T} C\emts(W^{\e,t,w,s,h}_r)^{1-R}dr\bigg]\notag\\
	&\qquad\qquad+\ets\E\bigg[ \varpi\bigg(T,\tilde W^0_T,S_T, \frac{h-h^0(T,\tilde W^0_T,S_T)}{\es}\bigg)\bigg]\\
	&\qquad\qquad-\ets\E\bigg[ \varpi\bigg(T,W^{\e,t,w,s,h}_T,S_T, \frac{h-h^0(T,W^{\e,t,w,s,h}_T,S_T)}{\es}\bigg)\bigg],
\end{align*}
where $k>0$ is the constant of Lemma~\ref{lem:bound.T.Re}. Now with the boundedness of the wealth on $(\tau^{\e,t,w,s,h}+\ets, T)$ (indeed after $\tauet+\ets$, $W^{\e,t,w,s,h}$ satisfies a linear, deterministic ODE with starting value satisfying \eqref{cond:exit}) and Lemma~\ref{lem:probexit}, the moments (of all positive and negative orders) of $W^0$, $\tilde W^0$ and $W^{\e,t,w,s,h}$ in Lemma~\ref{lem:moment.W0}, the definition of $\varpi$ in Lemma~\ref{lem:growth.varpi.tilde} and the growth of $\tilde \varpi $ in $\xi$ we obtain for some constant $C>0$ and some positive function $ C$ of $w$
\begin{align*}
	\frac{V^0(t,w,s)-V^\e(t,w,s,h)}{\e^{2m*}}- u (t,w,s) {\leqslant}&C(w)+C\emts\P(\tau^{\e,t,w,s,h}<T-\ets)\\
	&\qquad+\e^{2m^*-m^*\frac{2+m}{m}}\E\bigg[ |\tilde W^0_T|^{1-R+\frac{1}{m}}\bigg].\notag
\end{align*}
Note that $\tilde W^0$ is dominated by the wealth of an investor investing in a frictionless market with interest rate $r+\sup\{g(t)~|~t\in[0,T]\}$ and following the strategy given in Example~\ref{ex:frictionless} and has therefore finite $\cS^p$ norm by Lemma~\ref{lem:moment.W0}.
Note that the right hand side is in $\F_{comp}$ due to Lemma~\ref{lem:probexit} (in fact the last term goes to $0$ as $\e\to 0$, since $m>2$). This proves as well that $u^*\in\F_{comp}$.
\end{proof}
\begin{proof}[Proof of Lemma~\ref{lem:probexit}] \emph{Part 1: Bounds on the drift and volatility of the SDE satisfied by $X^{\e,t,w,s,h}$:}
	
	Let $(t,w,s,h)\in\cD\times\R^d$ and $\e\in(0,1]$. 
	Recall also the constants
	$C_{\tilde \varpi}:=\sup_{x}\frac{|x|^2}{1+\tilde \varpi^{\frac{2m}{m+2}}(x)}<\infty$ (it exists, see Lemma~\ref{lem:growth.varpi.tilde}) and  
	the stopping time $\tau^{\e,t,w,s,h}$ in \eqref{eq:def.tau.e}.
	
	Writing $C_{2,\tilde \varpi}:=\sup_{x} \frac{\tilde \varpi(x)}{|x|^{1+2/m}}<\infty$ (see Lemma~\ref{lem:growth.varpi.tilde}), assume that $(\e,t,w,s,h)\in(0,1)\times\cD\times\R^d$ is such that
	\begin{align}
		\label{eq:apinitial}
		(\e w)^{\frac{(m+2)m^*}{m}}+C_{2,\tilde \varpi}\left|\sum_{i=1}^{d} \left (\frac{h^is^i}{w}-\pi^i\right)^2\right|^{\frac{1}{2}-\frac{1}{m}}{\leq} \frac{(\pi^*)^2}{12 C_{\tilde \varpi}},
	\end{align}
	and all the estimates below will be uniform in these quantities, provided that \eqref{eq:apinitial} holds.
	 By the definition \eqref{eq:def.tau.e} of $\tauet$ and of $C_{\tilde \varpi}$, we have on $\llbracket t,\tau^\e\rrbracket$ the two inequalities
	\begin{align}
		\Bigg(\frac{(\lambda_{min}\pi^*)^2}{16C_{\tilde \varpi}^{\frac{2m}{m+2}} d^2 }\Bigg)^{\frac{m+2}{2m}}\geqslant& ( \e W^{\e,t,w,s,h}_u)^{\frac{(m+2)m^*}{m}}\left(1+ \tilde \varpi\left(X^{\e,t,w,s,h}_u\right)\right) \notag\\
		\geqslant& \frac{( \e W^{\e,t,w,s,h}_u)^{\frac{(m+2)m^*}{m} }|X^{\e,t,w,s,h}_u|^{1+\frac{2}{m}}}{C_{\tilde \varpi}},\label{eq:ineq.X.epsi}\\
		\e W^{\e,t,w,s,h}\geqslant& \bigg(\frac{1}{8c_{W}}\wedge 1\bigg)^{\frac{1}{2m^*}}~~ \mbox{and}~~| W^{\e,t,w,s,h}_u-W^0_u| \leq \frac{\pi^*}{2}W^0_u.\label{eq:ineq.W.epsi.W0}
	\end{align}
	This implies that on $\llbracket t,\tau^\e\rrbracket$ we have
	\begin{equation}
	\label{eq:hs/w.bounded}
		{W}^{\e,t,w,s,h}>0~~ \mbox{and}~~ \frac{(\lambda_{min}\pi^*)^2}{16 d^2C_{\tilde \varpi}^{\frac{2m}{m+2}} }\geqslant \frac{\lambda^2_{min}}{C_{\tilde \varpi}^{\frac{2m}{m+2}}}\sum_{i=1}^d  \left|\frac{{H}^{\e,t,w,s,h,i}_rS^i_r}{W^{\e,t,w,s,h}_u}- \pi^i\right|^2.
	\end{equation}
	This provides a useful inequality 
	\begin{align}\label{eq:bound.h.X}
		 \frac{\lambda^2_{min}}{C_{\tilde \varpi}^{\frac{2m}{m+2}}}\sum_{i=1}^d  \left|\frac{{H}^{\e,t,w,s,h,i}_rS^i_r}{W^{\e,t,w,s,h}_u}- \pi^i\right|^2\leq( \e W^{\e,t,w,s,h}_u)^{\frac{(m+2)m^*}{m}}\left(1+ \tilde \varpi\left(X^{\e,t,w,s,h}_u\right)\right).
	\end{align}
	These estimates also imply that for all $1\leqslant i\leqslant d$, the proportion of wealth invested in asset $i$ satisfies
	\begin{equation}
	\label{eq:bound.pi.epsi}
		-\frac{\pi^*}{4d}\leqslant \frac{{H}^{\e,t,w,s,h,i}_uS^i_u}{{W}^{\e,t,w,s,h}_u}-\pi^i\leqslant \frac{\pi^*}{4d}.
	\end{equation}
	on $\llbracket t,\tau^{\e,t,w,s,h}\rrbracket$. This last inequality, combined with the fact that $\pi^*$ is less than $\pi^i$ for all $i$, yields that the proportion of wealth in each asset is positive, $H^{\e,t,w,s,h,i}>0$. Summing \eqref{eq:bound.pi.epsi} in $i$, the definition of $\pi^*$ also implies that the fraction of wealth in cash is positive, and it is larger or equal to $\frac{3\pi^*}{4}$. Thus the amount of wealth in cash is larger than 
		\begin{align}\label{eq:lowerboundcash}
		\frac{3\pi^*}{4}{W}^{\e,t,w,s,h}_u \geq \frac{3\pi^*}{4}\frac{\pi^*}{2}W^0_u\mbox{ for }u\in[t,\tau^\e],
		\end{align}
		 the strategy is therefore admissible up to time $\tau^\e$.
	
	Given the price impact in \eqref{eq:def.f.example}, and the strategy defined in~\eqref{eq:trading.rate.cand} and \eqref{eq:def.thetae} the dynamics of the wealth \eqref{eq:frictionalwealthdynamics} becomes on  $\llbracket t,\tau^\e\rrbracket$
	\begin{align*}
	\frac{d{W}^{\e,t,w,s,h}_u}{{W}^{\e,t,w,s,h}_u} =& \left(r-g(u)^{-\frac{1}{R}}\right)du+\sum_{j=1}^{d}\frac{{H}^{\e,t,w,s,h,j}_uS^j_u}{{W}^{\e,t,w,s,h}_u}((\mu^j-r)du+\sigma^jdB_u)\\
	&-(\e{W}^{\e,t,w,s,h}_u)^{2m^*}\sum_{j=1}^{d}\frac{m-1}{m\kappa^{m-1}}\left|{\tilde{ \varpi}_{x_j}}\left(X^{\e,t,w,s,h}_u\right)\right|^{m}du.
	\end{align*}
	We can directly compute for any $p\neq 0$, 
	\begin{align}
	\label{eq:Ito.Wp}
	\frac{d\big({W}^{\e,t,w,s,h}_u\big)^{p}}{p\big({W}^{\e,t,w,s,h}_u\big)^{p}} =&\Big(r-g(u)^{-\frac{1}{R}}\Big)du+\sum_{j=1}^{d}\frac{{H}^{\e,t,w,s,h,j}_uS^j_u}{{W}^{\e,t,w,s,h}_u}((\mu^j-r)du+\sigma^jdB_u)\\
	&+\frac{p-1}{2}\bigg|\sum_{j=1}^{d}\frac{{H}^{\e,t,w,s,h,j}_uS^j_u}{{W}^{\e,t,w,s,h}_u}\sigma^j\bigg|^2du\notag\\
	&-\big(\e{W}^{\e,t,w,s,h}_u\big)^{2m^*}\sum_{j=1}^{d}\frac{m-1}{\kappa^{m-1} m}\left|{\tilde{ \varpi}_{x_j}}\left(X^{\e,t,w,s,h}_u\right)\right|^{m}du.\notag
	\end{align}
	Thus, the dynamics of the investment proportion displacement is given by (remember that the frictionless investment proportions $\pi_i$'s are constant in the Black-Scholes model):
	\begin{align*}
	&d\bigg(\frac{{H}^{\e,t,w,s,h,i}_uS^i_u}{{W}^{\e,t,w,s,h}_u}\bigg)=\frac{{H}^{\e,t,w,s,h,i}_uS^i_u}{{W}^{\e,t,w,s,h}_u}\Bigg((\mu^i-r+g(u)^{-\frac{1}{R}})du+\sigma^idB_u\\
	&\qquad-\sum_{j=1}^{d}\frac{{H}^{\e,t,w,s,h,j}_uS^j_u}{{W}^{\e,t,w,s,h}_u}(\mu^j-r+(\sigma^j)^\top\sigma^i)du-\sum_{j=1}^{d}\frac{{H}^{\e,t,w,s,h,j}_uS^j_u}{{W}^{\e,t,w,s,h}_u}\sigma^jdB_u\\
	&\qquad+(\e{W}^{\e,t,w,s,h}_u)^{2m^*}\sum_{j=1}^{d}\frac{m-1}{m\kappa^{m-1}}\left|{\tilde{ \varpi}_{x_j}}\left(X^{\e,t,w,s,h}_u\right)\right|^{m}du+\bigg|\sum_{j=1}^{d}\frac{{H}^{\e,t,w,s,h,j}_uS^j_u}{{W}^{\e,t,w,s,h}_u}\sigma^j\bigg|^2du\Bigg)\\
	&\qquad-(\e{W}^{\e,t,w,s,h}_u)^{-m^*}\frac{1}{ \kappa^{m-1}}\sum_{j=1}^d  \left|{\tilde{ \varpi}_{x_j}}\left(X^{\e,t,w,s,h}_u\right)\right|^{m-2}{\tilde{ \varpi}_{x_j}}\left(X^{\e,t,w,s,h}_u\right)(\s^{-1})^{i,j}du.
	\end{align*}
	We now compute the evolution of $X^{\e,t,w,s,h,i}$ and obtain,
	\begin{align*}
	&dX^{\e,t,w,s,h,i}_u=\sum_{k=1}^{d} \frac{\s_{i,k}{H}^{\e,t,w,s,h,k}_uS^k_u}{\es\big({W}^{\e,t,w,s,h}_u\big)^{(1+m^*)}}\bigg((\mu^k-r+g(u)^{-\frac{1}{R}})du\\
	&-\sum_{j=1}^{d}\frac{{H}^{\e,t,w,s,h,j}_uS^j_u}{{W}^{\e,t,w,s,h}_u}(\mu^j-r+(\sigma^j)^\top\sigma^k)du-\sum_{j=1}^{d}\frac{{H}^{\e,t,w,s,h,j}_uS^j_u}{{W}^{\e,t,w,s,h}_u}\sigma^jdB_u+\sigma^kdB_u\\
	&+(\e{W}^{\e,t,w,s,h}_u)^{2m^*}\sum_{j=1}^{d}\frac{m-1}{m\kappa^{m-1}}\Big|{\tilde{ \varpi}_{x_j}}\big( X^{\e,t,w,s,h}_u\big)\Big|^{m}du+\bigg|\sum_{j=1}^{d}\frac{{H}^{\e,t,w,s,h,j}_uS^j_u}{{W}^{\e,t,w,s,h}_u}\bigg|^2du\bigg)\\
	&-m^*X^{\e,t,w,s,h,i}_u\bigg((r-g(u)^{-\frac{1}{R}})du+\sum_{j=1}^{d}\frac{{H}^{\e,t,w,s,h,j}_uS^j_u}{{W}^{\e,t,w,s,h}_u}((\mu^j-r)du+\sigma^jdB_u)\\
	&-(\e{W}^{\e,t,w,s,h}_u)^{2m^*}\sum_{j=1}^{d}\frac{m-1}{m\kappa^{m-1}}\left|{\tilde{ \varpi}_{x_j}}\left(X^{\e,t,w,s,h}_u\right)\right|^{m}du-\frac{m^*+1}{2}\bigg|\sum_{j=1}^{d}\frac{{H}^{\e,t,w,s,h,j}_uS^j_u}{{W}^{\e,t,w,s,h}_u}\sigma^j\bigg|^2du\bigg)\\
	&-m^*\sum_{k=1}^{d} \frac{\s_{i,k}{H}^{\e,t,w,s,h,k}_uS^k_u}{\es\big({W}^{\e,t,w,s,h}_u\big)^{(1+m^*)}}\bigg(\sum_{j=1}^{d}\frac{{H}^{\e,t,w,s,h,j}_uS^j_u}{{W}^{\e,t,w,s,h}_u}\sigma^j\bigg)^\top\bigg(\sigma^k-\sum_{j=1}^{d}\frac{{H}^{\e,t,w,s,h,j}_uS^j_u}{{W}^{\e,t,w,s,h}_u}\sigma^j\bigg)du\\
	&-(\e{W}^{\e,t,w,s,h}_u)^{-2m^*}\frac{1}{\kappa^{m-1}} \left|\tilde \varpi_{x_i}\left(X^{\e,t,w,s,h}_u\right)\right|^{m-2}\tilde \varpi_{x_i}\left(X^{\e,t,w,s,h}_u\right)du.
	\end{align*}
	Note that due to the finiteness of $C_{\tilde\varpi_x}:=\sup_{i,x}\frac{|\tilde \varpi_{x_i}(x)|}{|x|^{2/m}}$ we have the inequality 
	\begin{equation}
	\label{eq:bound.sum.tilde.varpi.xj}
		(\e{W}^{\e,t,w,s,h}_u)^{2m^*}\sum_{j=1}^{d}\left|{\tilde{ \varpi}_{x_j}}\left(X^{\e,t,w,s,h}_u\right)\right|^{m}\leqslant C^m_{\tilde\varpi_x}\lambda^2_{max}\sum_{i=1}^d  \bigg|\frac{{H}^{\e,t,w,s,h,i}_uS^i_u}{{W}^{\e,t,w,s,h}_u}-\pi_i\bigg|^2,
	\end{equation}
	where $\lambda_{max}$ is the largest eigenvalue of the matrix $\s$.
	Additionally, the equalities 
	\begin{align*}
		\frac{{H}^{\e,t,w,s,h,k}_uS^k_u}{\es \big({W}^{\e,t,w,s,h}_u\big)^{1+m^*}} &= \bigg((\s^{-1})X^{\e,t,w,s,h}_u+\frac{\pi}{(\e {W}^{\e,t,w,s,h}_u)^{m^*}}\bigg)_k\\
		\sum_{k=1}^{d} \frac{\s_{i,k}{H}^{\e,t,w,s,h,k}_uS^k_u}{\es\big({W}^{\e,t,w,s,h}_u\big)^{(1+m^*)}} &= X^{\e,t,w,s,h,i}_u+\frac{\big(\s\pi\big)_i}{(\e {W}^{\e,t,w,s,h}_u)^{m^*}}
	\end{align*}
	allow us to claim that there exist processes $Y^{1,i}$ valued in $\R$, and $ Y^{2,i}$ valued in $\R^d$, function of the state variables and with growth in $\frac{{H}^{\e,t,w,s,h,j}_uS^j_u}{{W}^{\e,t,w,s,h}_u}$ at most quadratic such that 
	\begin{align*}
		&dX^{\e,t,w,s,h,i}_u = -(\e {W}^{\e,t,w,s,h}_u)^{-2m^*}\frac{1}{ \kappa^{m-1}} \left|\tilde \varpi_{x_i}\left(X^{\e,t,w,s,h}_u\right)\right|^{m-2}\tilde \varpi_{x_i}\left(X^{\e,t,w,s,h}_u\right)du\\
		&\quad+\Big((1+m^*)X^{\e,t,w,s,h,i}_u+\frac{\big(\s\pi\big)_i}{ (\e{W}^{\e,t,w,s,h}_u)^{m^*}}\Big) \left(Y^{1,i}_udu+(Y^{2,i}_u)^\top dB_u\right)\\
		&\quad+\frac{m^*(m^*+1)}{2}\frac{\big(\s\pi\big)_i}{ (\e{W}^{\e,t,w,s,h}_u)^{m^*}}\bigg|\sum_{j=1}^{d}\frac{{H}^{\e,t,w,s,h,j}_uS^j_u}{{W}^{\e,t,w,s,h}_u}\sigma^j\bigg|^2\\
		&\quad+\sum_{k=1}^{d} \frac{\s_{i,k}{H}^{\e,t,w,s,h,k}_uS^k_u}{\es\big({W}^{\e,t,w,s,h}_u\big)^{(1+m^*)}}\bigg(\Big(\mu^k-(1+m^*)\sum_{j=1}^{d}\frac{{H}^{\e,t,w,s,h,j}_uS^j_u}{{W}^{\e,t,w,s,h}_u}\sigma^j\cdot \sigma^k\Big)du+\sigma^kdB_u\bigg),
	\end{align*}
	where we can define the processes $Y$ as follows
	\begin{align*}
		Y^{1,i}_u =& \bigg(-r+g(u)^{-\frac{1}{R}}-\sum_{j=1}^{d}\frac{{H}^{\e,t,w,s,h,j}_uS^j_u}{{W}^{\e,t,w,s,h}_u}(\mu^j-r)\\
		&+(\e{W}^{\e,t,w,s,h}_u)^{2m^*}\sum_{j=1}^{d}\frac{m-1}{m\kappa^{m-1}}\left|{\tilde{ \varpi}_{x_j}}\left(X^{\e,t,w,s,h}_u\right)\right|^{m}+\bigg|\sum_{j=1}^{d}\frac{{H}^{\e,t,w,s,h,j}_uS^j_u}{{W}^{\e,t,w,s,h}_u}\sigma^j\bigg|^2\bigg),\\
		Y^{2,i}_u =& -\sum_{j=1}^{d}\frac{{H}^{\e,t,w,s,h,j}_uS^j_u}{{W}^{\e,t,w,s,h}_u}\sigma^j.
	\end{align*}
	By \eqref{eq:bound.pi.epsi} and \eqref{eq:bound.sum.tilde.varpi.xj}, the processes $Y^{1,i}$, $Y^{2,i}$ are bounded (uniformly in $(u,\omega)$) on $\llbracket t,\tau^{\e,t,w,s,h}\rrbracket$. We now compute (remember that by Lemma~\eqref{lem:growth.varpi.tilde}, $\tilde\varpi_{x_ix_j}\equiv 0$ for $i\neq j$)
	\begin{align}\label{eq:meanreversion}
	&d\tilde \varpi(X^{\e,t,w,s,h}_u)=-(\e {W}^{\e,t,w,s,h}_u)^{-2m^*}\sum_{i=1}^d  \frac{1}{\kappa^{m-1}} |\tilde \varpi_{x_i}(X^{\e,t,w,s,h}_u)|^m du\\
	&+\sum_{i=1}^{d} \tilde \varpi_{x_i}(X^{\e,t,w,s,h}_u)\Big((1+m^*)X^{\e,t,w,s,h,i}_u+\frac{\big(\s\pi\big)_i}{ (\e {W}^{\e,t,w,s,h}_u)^{m^*}}\Big) \Big(Y^{1,i}_udu+\big(Y^{2,i}_u\big)^\top dB_u\Big)\notag\\
	&+\frac{m^*(m^*+1)}{2}\sum_{i=1}^{d} \tilde \varpi_{x_i}(X^{\e,t,w,s,h}_u)\frac{\big(\s\pi\big)_i}{ (\e{W}^{\e,t,w,s,h}_u)^{m^*}}\bigg|\sum_{j=1}^{d}\frac{{H}^{\e,t,w,s,h,j}_uS^j_u}{{W}^{\e,t,w,s,h}_u}\sigma^j\bigg|^2du\notag\\
	&+\sum_{i,k=1}^{d} \frac{\tilde\varpi_{x_i}\big(X^{t,w,s,h}_u\big)\s_{i,k}{H}^{\e,t,w,s,h,k}_uS^k_u}{\es\big({W}^{\e,t,w,s,h}_u\big)^{(1+m^*)}}\bigg(\Big(\mu^k-(1+m^*)\sum_{j=1}^{d}\frac{{H}^{\e,t,w,s,h,j}_uS^j_u}{{W}^{\e,t,w,s,h}_u}\sigma^j\cdot \sigma^k\Big)du+\sigma^kdB_u\bigg)\notag\\
	&+\frac{1}{2}\sum_{i=1}^{d}\tilde \varpi_{x_ix_i}(X^{\e,t,w,s,h}_u)\bigg|\sum_{k=1}^{d} \frac{\s_{i,k}{H}^{\e,t,w,s,h,k}_uS^k_u}{\es\big({W}^{\e,t,w,s,h}_u\big)^{(1+m^*)}}\big( Y^{2,i}+\sigma^k\big)\bigg|^2du.\notag
	\end{align}
	Now define
	\begin{align*}
	&Y^{3}_u =-\sum_{i=1}^d  \frac{1}{\kappa^{m-1}} |\tilde \varpi_{x_i}(X^{\e,t,w,s,h}_u)|^m+\frac{mR}{4}|X^{\e,t,w,s,h}_u|^2\\
	&+\sum_{i=1}^{d} \tilde \varpi_{x_i}(X^{\e,t,w,s,h}_u)\Big((1+m^*)(\e {W}^{\e,t,w,s,h}_u)^{2m^*} X^{\e,t,w,s,h,i}_u+(\e {W}^{\e,t,w,s,h}_u)^{m^*}{\big(\s\pi\big)_i}\Big) Y^{1,i}_u\notag\\
	&+\frac{m^*(m^*+1)}{2}\sum_{i=1}^{d} \tilde \varpi_{x_i}(X^{\e,t,w,s,h}_u)\big(\s\pi\big)_i (\e{W}^{\e,t,w,s,h}_u)^{m^*}\bigg|\sum_{j=1}^{d}\frac{{H}^{\e,t,w,s,h,j}_uS^j_u}{{W}^{\e,t,w,s,h}_u}\sigma^j\bigg|^2\notag\\
	&+\sum_{i,k=1}^{d} (\e {W}^{\e,t,w,s,h}_u)^{m^*}\frac{\tilde\varpi_{x_i}\big(X^{t,w,s,h}_u\big)\s_{i,k}{H}^{\e,t,w,s,h,k}_uS^k_u}{{W}^{\e,t,w,s,h}_u}\Big(\mu^k-(1+m^*)\sum_{j=1}^{d}\frac{{H}^{\e,t,w,s,h,j}_uS^j_u}{{W}^{\e,t,w,s,h}_u}\sigma^j\cdot \sigma^k\Big)\notag\\
	&+\frac{1}{2}\sum_{i=1}^{d}\tilde \varpi_{x_ix_i}(X^{\e,t,w,s,h}_u)\bigg|\sum_{k=1}^{d} \frac{\s_{i,k}{H}^{\e,t,w,s,h,k}_uS^k_u}{{W}^{\e,t,w,s,h}_u}\big( Y^{2,i}+\sigma^k\big)\bigg|^2,\notag\\
	&Y^{4}_u = \sum_{i=1}^{d} \tilde \varpi_{x_i}(X^{\e,t,w,s,h}_u)\bigg(\Big((1+m^*)(\e {W}^{\e,t,w,s,h}_u)^{m^*}X^{\e,t,w,s,h,i}_u+{\big(\s\pi\big)_i}\Big) \big(Y^{2,i}_u\big)^\top \\
	&\qquad\qquad+\sum_{k=1}^{d} \frac{\s_{i,k}{H}^{\e,t,w,s,h,k}_uS^k_u}{{W}^{\e,t,w,s,h}_u}\sigma^k\bigg),
	\end{align*}
	so that 
	\begin{align*}
		d\tilde \varpi(X^{\e,t,w,s,h}_u) =& (\e{W}^{\e,t,w,s,h}_u)^{-2m^*}\left(Y^3_u -\frac{mR}{4}|X^{\e,t,w,s,h}_u|^2\right) du+(\e {W}^{\e,t,w,s,h}_u)^{-m^*} \big(Y^4_u\big)^\top dB_u.
	\end{align*} 
Now, the $Y^{1,i}$'s and the proportion of wealth invested in each asset are bounded on $\llbracket t,\tau^{\e,t,w,s,h}\rrbracket$ (see \eqref{eq:bound.pi.epsi}), so for some $C>0$ it holds
\begin{equation}
\label{eq:bound.muk.Y1}
	\Big|\mu^k-(1+m^*)\sum_{j=1}^{d}\frac{{H}^{\e,t,w,s,h,j}_uS^j_u}{{W}^{\e,t,w,s,h}_u}\sigma^j\cdot \sigma^k+Y^{1,i}_u\Big|\leqslant C.
\end{equation} 
Note that on $\llbracket t,\tau^\e\rrbracket$, $\e {W}^{\e,t,w,s,h}$ is bounded by definition~\eqref{eq:def.tau.e} of $\tauet$, and $\tilde\varpi_{xx}$ is bounded by Lemma~\ref{lem:bounded.second.derivative}.  Thus on this interval we have for some $C>0$,
\begin{align}\label{eq:boundY31}
	Y^{3}_u \leq&-\sum_{i=1}^d  \frac{1}{\kappa^{m-1}} |\tilde \varpi_{x_i}(X^{\e,t,w,s,h}_u)|^m+\frac{mR}{4}|X^{\e,t,w,s,h}_u|^2\\
	&+C\sum_{i=1}^{d}| \tilde \varpi_{x_i}(X^{\e,t,w,s,h}_u)X^{\e,t,w,s,h,i}_u|+| \tilde \varpi_{x_i}(X^{\e,t,w,s,h}_u)|+C.\notag
\end{align}
Note that by Lemma~\ref{lem:growth.varpi.tilde},
	\begin{equation}
	\label{eq:bounds.tilde.varpi.appendix}
		\sup_{i\in\{1,...,d\},x\in \R^d}\left|\frac{\tilde \varpi_{x_i} (x)x_i}{\tilde \varpi(x)}\right|+\left|\frac{\tilde \varpi_{x_i}(x)}{|x|}\right|+\left|\frac{\tilde \varpi_{x_i}(x)}{1+|x|^{\frac{2}{m}}}\right|<\infty.
	\end{equation}
	Thanks to \eqref{1st-corrector-eq-bis} (we have assumed that the matrix $(\Sigma \s)^\top(\Sigma \s)$ has positive diagonal terms) and the convexity of $\tilde \varpi$ it holds that 
	\begin{equation}
	\label{eq:bound.sum.varpi.x.lambda}
	-\sum_{i=1}^d  \kappa^{1-m} |\tilde \varpi_{x_i}(X^{\e,t,w,s,h}_u)|^m \leqslant m\lambda-\frac{mR}{2}|X^{\e,t,w,s,h}_u|^2.
	\end{equation}
	Thus, with \eqref{eq:boundY31} and \eqref{eq:bound.sum.varpi.x.lambda}, we obtain  
	\begin{align}\label{eq:boundY32}
		Y^{3}_u \leq&m\lambda-\frac{mR}{4}|X^{\e,t,w,s,h}_u|^2+C\left( \tilde \varpi(X^{\e,t,w,s,h}_u)+|X^{\e,t,w,s,h}_u|^\frac{2}{m}+1\right).
	\end{align}
	Similarly due to the growth of $\varpi$ the function $$x\mapsto m\lambda-\frac{mR}{4}|x|^2+C\left( \tilde \varpi(x)+|x|^\frac{2}{m}+1\right)$$
	is bounded from above and we obtain that $Y^3$ is bounded from above by a constant.  
	Similarly there exists a constant $C>0$ such that 
	$-\frac{mR}{4}|x|^2\leq C-\frac{(1+\tilde \varpi(x))^{\frac{2m}{m+2}}}{C}$.
	We then obtain the dynamics of $\tilde \varpi(X^{\e,t,w,s,h})$ as
		\begin{align}\label{eq:dynamics.varpi}
		d\tilde \varpi(X^{\e,t,w,s,h}_u) =& (\e{W}^{\e,t,w,s,h}_u)^{-2m^*}\left(\tilde Y^3_u -\frac{(1+ \varpi(X^{\e,t,w,s,h}_u) )^{\frac{2m}{m+2}}}{C}\right) du\\
		&+(\e {W}^{\e,t,w,s,h}_u)^{-m^*} \big(Y^4_u\big)^\top dB_u,\notag
	\end{align}
for some process $\tilde Y^3$ uniformly bounded from above and $|Y^4_u|\leq C |\tilde \varpi_{x}(X^{\e,t,w,s,h}_u)|$.
	
\emph{Part 2: Bound of $\P(\tau^\e< T-\ets)$.} When not needed, we drop the superscripts $t,w,s,h$ for notational simplicity. We denote by $\cX_u:=( W^{\e,t,w,s,h}_u)^{\frac{(m+2)m^*}{m}}\left(1+ \tilde \varpi\left(X^{\e,t,w,s,h}_u\right)\right)$ and $T^\e:=T-\ets$ and define for $\e\in(0,1)$
\begin{align*}
	A^\e&:=\Bigg\{\sup_{r\in[t,\tau^\e]}\cX_r \geqslant  \bigg(\frac{(\lambda_{min}\pi^*)^2}{16C_{\tilde \varpi}^{\frac{2m}{m+2}}d^2 \ets}\bigg)^{\frac{m+2}{2m}}\Bigg\},\, B^\e:=\left\{\sup_{r\in[t,\tau^\e]}\left|\frac{W^\e_r}{W^0_r}-1\right|\geqslant \frac{\pi^*}{2}\right\},\\
	C^\e&:=\bigg\{\inf_{r\in[t,\tau^\e]}{W^0_r}\leq \e^{m^*}\mbox{ or }\sup_{r\in[t,\tau^\e]}{\e^{1/2} W^0_r} \geqslant \frac{2}{2+\pi^*}\Big(\frac{1}{4c_{W}}\wedge 1\Big)^{\frac{1}{2m^*}}\bigg\}.
\end{align*}
We now compute 
\begin{align*}
	\P\left(\tau^\e<T-\ets\right)&{\leq}\P\left(A^\e\cup B^\e \cup C^\e\right){\leq}\P\left(A^\e\cap (B^\e)^c \cap (C^\e)^c\right)+ \P\left(B^\e\cap (C^\e)^c\right)+\P\left(C^\e\right).
\end{align*}
Note that $\P\left(C^\e\right)$ is related to the exit time of Brownian motion and can be easily estimated to any polynomial order (by Markov's inequality for example) and we obtain for example that $$\P\left(C^\e\right)\leq C(w) \e^{3m*}.$$
	
We estimate the other terms  separately. 
	
	{\it Step 1: Estimation of $\P\bigg[\sup_{t\in[0,T^\e]}\cX_t \geqslant \bigg(\frac{(\lambda_{min}\pi^*)^2}{16C_{\tilde \varpi}^{\frac{2m}{m+2}}d^2 \ets}\bigg)^{\frac{m+2}{2m}} \cap (B^\e)^c \cap (C^\e)^c\bigg]$: }
	Denote $K^\e$ and $M^\e$ respectively the finite variation and martingale part of $\cX$ and $\cM^\e$ (use \eqref{eq:Ito.Wp}, \eqref{eq:meanreversion} and It\^o's formula) the set
	\begin{align*}
	\cM^\e:=&\Bigg\{\cX_{\tau^\e} = \bigg(\frac{(\lambda_{min}\pi^*)^2}{16C_{\tilde \varpi}^{\frac{2m}{m+2}}d^2 \ets}\bigg)^{\frac{m+2}{2m}}\cap\sup_{t\in[0,T^\e]}K_t^\e \geqslant \bigg(\frac{(\lambda_{min}\pi^*)^2}{32C_{\tilde \varpi}^{\frac{2m}{m+2}}d^2 \ets}\bigg)^{\frac{m+2}{2m}}\\
	&\qquad\qquad\qquad\qquad\qquad\cap (B^\e)^c \cap (C^\e)^c\cap \sup_{t\in[0,T^\e]}M_t^\e \leqslant \bigg(\frac{(\lambda_{min}\pi^*)^2}{48C_{\tilde \varpi}^{\frac{2m}{m+2}}d^2 \ets}\bigg)^{\frac{m+2}{2m}}\Bigg\}.
	\end{align*}
	Then we have the following inequalities, 
	\begin{align*}
	\P\Bigg[\sup_{t\in[0,T^\e]}\cX_t \geqslant & \bigg(\frac{(\lambda_{min}\pi^*)^2}{16C_{\tilde \varpi}^{\frac{2m}{m+2}}d^2 \ets}\bigg)^{\frac{m+2}{2m}} \cap (B^\e)^c \cap (C^\e)^c\Bigg]\\
	&\qquad\qquad\leq\P\Bigg[\sup_{t\in[0,T^\e]}M_t^\e \geqslant \bigg(\frac{(\lambda_{min}\pi^*)^2}{48C_{\tilde \varpi}^{\frac{2m}{m+2}}d^2 \ets}\bigg)^{\frac{m+2}{2m}} \cap (B^\e)^c \cap (C^\e)^c\Bigg]+ \P[\cM^\e].	\end{align*}
	To find $K^\e$ and $M^\e$ we apply It\^{o}'s formula to $\cX$, using \eqref{eq:Ito.Wp} and \eqref{eq:dynamics.varpi} and obtain
		\begin{align*}
		&d\cX_t = ( W^{\e,t,w,s,h}_u)^{\frac{(2-m)m^*}{m}}\emts\left(\tilde Y^3_u -\frac{(1+\tilde \varpi(X^{\e,t,w,s,h}_u))^{\frac{2m}{m+2}}}{C}\right) du\\
		&+\left(Y^6_u\cX_u +\ems  (W^{\e,t,w,s,h}_u)^{\frac{2m^*}{m}}\big(Y^4_u\big)^\top Y^7_u\right) du+\bigg(\ems ( {W}^{\e,t,w,s,h}_u)^{\frac{2m^*}{m}} \big(Y^4_u\big)^\top+\cX_u \big(Y^8_u\big)^\top \bigg) dB_u,
	\end{align*}
	for $Y^6$ bounded from above, $Y^7$ and $Y^8$ uniformly bounded. 
	Note that the inequality $\sup_{x}\frac{|x||\tilde \varpi_x(x)|}{1+\tilde \varpi(x)}<\infty$ implies that there exist $Y^9$ bounded from above and $Y^{10}$ uniformly bounded such that 
	\begin{align}\label{eq:SDE.cX}
		&d\cX_t = ( W^{\e,t,w,s,h}_u)^{\frac{(2-m)m^*}{m}}\emts\left(\tilde Y^3_u -\frac{(1+\tilde \varpi(X^{\e,t,w,s,h}_u))^{\frac{2m}{m+2}}}{C}\right) du+\cX_u(Y^9_udu+Y^{10}_udB_u).
	\end{align}
	Note that 
	$$K^\e_r=\int_t^r  ( W^{\e,t,w,s,h}_u)^{\frac{(2-m)m^*}{m}}\emts\left(\tilde Y^3_u -\frac{(1+\tilde \varpi(X^{\e,t,w,s,h}_u))^{\frac{2m}{m+2}}}{C}\right) +\cX_uY^9_u du.$$
	Thus, on $\cM^\e$, there exists $u$ such that 
\begin{align*}
	\ets( W^{\e,t,w,s,h}_u)^{\frac{(m-2)m^*}{m}}\cX_uY^9_u+\tilde Y^3_u \geq \frac{(1+\tilde \varpi(X^{\e,t,w,s,h}_u))^{\frac{2m}{m+2}}}{C}=( W^{\e,t,w,s,h}_u)^{-2m^*}\frac{\cX_u^{\frac{2m}{m+2}}}{C}.
\end{align*}	Define 
\begin{align*}
	\bar t^\e:=\sup&\left\{u: \ets( W^{\e,t,w,s,h}_u)^{\frac{(m-2)m^*}{m}}\cX_uY^9_u+\tilde Y^3_u\geq( W^{\e,t,w,s,h}_u)^{-2m^*}\frac{\cX_u^{\frac{2m}{m+2}}}{C} \right\}.
\end{align*}
Note that on $\cM^\e$, $\cX^\e_{\tau^\e}=\Big(\tfrac{(\lambda_{min}\pi^*)^2}{16C_{\tilde \varpi}^{\frac{2m}{m+2}}d^2 \ets}\Big)^{\frac{m+2}{2m}}$ and the processes above are continuous. Thus on the event $\{\sup_{t\in[0,T]} \e W_t^0{\leq} 1\}\subset\{\sup_{t\in[0,T]} \e^{1/2} W_t^0{\leq} 1\}$ we have $\bar t^\e< \tau^\e$. Note that at $\bar t^\e$ we have 
$$\ets( W^{\e,t,w,s,h}_{\bar t^\e})^{2m^*}(( W^{\e,t,w,s,h}_{\bar t^\e})^{\frac{-(m+2)m^*}{m}}\cX_{\bar t^\e})Y^9_{\bar t^\e}+\tilde Y^3_{\bar t^\e}\geq \frac{( ( W^{\e,t,w,s,h}_{\bar t^\e})^{\frac{-(m+2)m^*}{m}}\cX_{\bar t^\e})^{\frac{2m}{m+2}}}{C}. $$
Thus, due to the boundedness of $\ets( W^{\e,t,w,s,h}_{\bar t^\e})^{2m^*}$, $Y^9$ and the boundedness from above of $\tilde Y^3$,  $(( W^{\e,t,w,s,h}_{\bar t^\e})^{\frac{-(m+2)m^*}{m}}\cX_{\bar t^\e})$ is bounded from above by the positive root of the equation 
$Cy+C=y^{\frac{2m}{m+2}}$ and 
$$\cX_{\bar t^\e}\leq C( W^{\e,t,w,s,h}_{\bar t^\e})^{\frac{(m+2)m^*}{m}}\leq C( \e^{1/2} W^{\e,t,w,s,h}_{\bar t^\e})^{\frac{(m+2)m^*}{m} } \e^{\frac{-(m+2)m^*}{2m}}\leq C\e^{\frac{-(m+2)m^*}{2m}}.$$
Additionally on $\llbracket\bar t^\e,\tau^\e\rrbracket$, $K$ is decreasing and  we are on the event $\sup_{t\in[0,\tau^\e]}M_t^\e {\leq}\Big(\frac{(\lambda_{min}\pi^*)^2}{48C_{\tilde \varpi}^{\frac{2m}{m+2}}d^2 \ets}\Big)^{\frac{m+2}{2m}}$ and 
$\cX_{\tau^\e} = \bigg(\frac{(\lambda_{min}\pi^*)^2}{16C_{\tilde \varpi}^{\frac{2m}{m+2}}d^2 \ets}\bigg)^{\frac{m+2}{2m}}= \bigg(\frac{(\lambda_{min}\pi^*)^2}{16C_{\tilde \varpi}^{\frac{2m}{m+2}}d^2}\bigg)^{\frac{m+2}{2m}}\e^{\frac{-(m+2)m^*}{m}}\gg \cX_{\bar t^\e}$.
Thus, $\P\left(\cM^\e\right)=0$ for $\e>0$ small enough.  This equality reflects the fact that a mean reverting process can only become large thanks to the contribution of its diffusive part. 

	Note that 
		$$\frac{d\langle M^\e_u\rangle}{du}\leq   C\cX_u^2.$$
	By Markov and BDG inequalities we have 
	\begin{align*}
	&\P\Bigg[\sup_{t\in[0,T^\e]}(M_t^\e)^2 \geqslant \bigg(\frac{(\lambda_{min}\pi^*)^2}{48C_{\tilde \varpi}^{\frac{2m}{m+2}}d^2 \ets}\bigg)^{\frac{m+2}{m}} \cap (B^\e)^c \cap (C^\e)^c\Bigg]\\
	&\qquad{\leq}  \Bigg(\frac{48C_{\tilde \varpi}^{\frac{2m}{m+2}}d^2 \ets}{(\lambda_{min}\pi^*)^2}\Bigg)^{\frac{m+2}{m}}\E\left[\sup_{t\in[0,T^\e]}(M_t^\e)^2 \right]\leq C\e^{2m^*\frac{m+2}{m}}\E\left[\int_t^{\tau^{\e}}\cX_u^2 du\right].
	\end{align*}
	We define $\a=\frac{2(m^2+3m-2)}{m(m+2)}>2$, $p=\frac{1}{2mm^*}>1$. Now, similarly to \eqref{eq:SDE.cX} applying  It\^{o}'s formula to $\cX^{\a}$ between $t$ and $\tau^\e$ produces a martingale part and an absolutely continuous part that can be divided in two elements: a dominating mean-reverting one and a slowly increasing one. Shifting half of the mean-reverting part to the left hand side, the other half is used to bound the rest of the absolutely continuous part by a constant $C>0$ and we obtain
	$$\E\left[\cX^\alpha_{\tau^\e}+\int_t^{\tau^\e}\emts \big(W^{\e,t,w,s,h}\big)^{-\frac{1}{m}}\cX^{\frac{3m^2+4m-4}{m(m+2)}}_u du\right]\leq C(\cX_t^\a+1)\leq C(w)\e^{-m^*(1+\frac{2}{m})\a}. $$
	Note that ${\frac{3m^2+4m-4}{m(m+2)}}=\frac{1}{mm^*}=2p$.
	Applying the reverse Holder inequality, the bound on the moments of $W$ implies that 
	$$\E\left[\int_t^{\tau^\e} \cX^2_u du\right]^{p}\leq C(w)\e^{2m^*(1-(1/2+1/m)\a)}=C(w)\e^{-\frac{2}{m^2}}.$$
	This finally implies that 
\begin{align*}
	\P\Bigg[\sup_{t\in[0,T^\e]}(M_t^\e)^2 \geqslant \bigg(\frac{(\lambda_{min}\pi^*)^2}{48C_{\tilde \varpi}^{\frac{2m}{m+2}}d^2 \ets}\bigg)^{\frac{m+2}{m}} \cap (B^\e)^c \cap (C^\e)^c\Bigg] &\leqslant C(w)\e^{2m^*\frac{m+2}{m}}\e^{-\frac{2m^*}{m}}\\
	&=C(w)\e^{2m^*}.
\end{align*}
		
	{\it Step 2: Bounding $\P[B^\e\cap   (C^\e)^c]$: } 
	One can compute the dynamics of ratio of the frictional wealth to the frictionless wealth as (see \eqref{eq:dyn.W0} and \eqref{eq:frictionalwealthdynamics})
	\begin{align*}
	d\bigg(\frac{W^\e_u}{W^0_u}\bigg)\frac{W^0_u}{W^\e_u}=&\sum_{j=1}^{d}\Big(\frac{{H}^{\e,t,w,s,h,j}_uS^j_u}{{W}^{\e,t,w,s,h}_u}-\pi_j\Big)\big((\mu_j-r)du+\sigma^j dB_u\big)-\frac{W^\e_u}{W^0_u}\sum_{j,k=1}^{d}\frac{{H}^{\e,t,w,s,h,j}_uS^j_u}{{W}^{\e,t,w,s,h}_u}\pi_k\sigma^j\cdot\sigma^k\\
	&-(\e{W}^{\e,t,w,s,h}_u)^{2m^*}\sum_{j=1}^{d}\frac{m-1}{m\kappa^{m-1}}\left|{\tilde{ \varpi}_{x_j}}\left(X^{\e,t,w,s,h}_u\right)\right|^{m}du\\
	&=\sum_{j=1}^{d}\Big(\frac{{H}^{\e,t,w,s,h,j}_uS^j_u}{{W}^{\e,t,w,s,h}_u}-\pi_j\Big)(\a_u du+\beta_u dB_u),
	\end{align*}
	on $\llbracket t,\tau^{\e,t,w,s,h}\rrbracket $ for some bounded $\a$ and $\beta$. Then we have using \eqref{eq:bound.h.X}
	\begin{align*}\P[B^\e&\cap   (C^\e)^c]\leq \P\Bigg[\Bigg(\int_t^{\tau^{\e}}\sum_{j=1}^{d}\left|\frac{{H}^{\e,t,w,s,h,j}_uS^j_u}{{W}^{\e,t,w,s,h}_u}-\pi_j\right|\a_tdt\Bigg)^2\geq \frac{(\pi^*)^2}{16}\cap (C^\e)^c\Bigg]\\
	&\qquad+\P\Bigg[\bigg(\sup_{t\leq r\leq \tau^\e}\int_t^{r}\sum_{j=1}^{d}\left|\frac{{H}^{\e,t,w,s,h,j}_uS^j_u}{{W}^{\e,t,w,s,h}_u}-\pi_j\right|\beta_tdB_t\bigg)^2\geq \frac{(\pi^*)^2}{16}\cap (C^\e)^c\Bigg]\\
	&\leq C\E\left[\int_t^{\tau^\e}\sum_{i=1}^d\left|\frac{{H}^{\e,t,w,s,h,j}_uS^j_u}{{W}^{\e,t,w,s,h}_u}-\pi_j\right|^2du \1_{\left\{ (C^\e)^c\right\}} \right]\leq C\ets\E\left[\int_t^{\tau^\e}\cX_u^{\frac{2m}{m+2}}du \right],
	\end{align*}
	and $\E\left[\int_t^{\tau^\e}\cX_u^{\frac{2m}{m+2}}du \right]$ can be bounded by an element in $\F_{comp}$. 
\end{proof}	
	
{\footnotesize	
\bibliographystyle{plain}
 }
\end{document}